\documentclass[a4paper, 11pt]{amsart}
\usepackage{amsmath,amsthm,amssymb}
\usepackage[T1]{fontenc}
\usepackage{url}
\usepackage{amsmath,amsthm,amssymb}
\usepackage{verbatim}
\usepackage{mathrsfs}
\usepackage{graphics}
\usepackage{graphicx}
\usepackage{subfigure}
\usepackage{hyperref}
\usepackage{mathtools}
\usepackage{color}
\usepackage{enumitem}

\allowdisplaybreaks[1]

\newtheorem{theorem}{Theorem}[section]
\newtheorem{proposition}[theorem]{Proposition}
\newtheorem{lemma}[theorem]{Lemma}
\newtheorem{corollary}[theorem]{Corollary}

\theoremstyle{remark}
\newtheorem{remark}[theorem]{Remark}
\newtheorem{definition}{Definition}

\numberwithin{equation}{section}

\newcommand{\vep}{\varepsilon}
\newcommand{\R}{{\mathbb{R}}}

\newcommand{\C}{{\mathbb{C}}}
\newcommand{\Z}{{\mathbb{Z}}}
\newcommand{\N}{{\mathbb{N}}}

\newcommand{\GG}{\mathcal{G}}

\newcommand{\Arg}{\operatorname{Arg}}

\newcommand{\dist}{\operatorname{dist}}
\newcommand{\ov}{\overline}

\newcommand{\norm}[1]{\left\|#1\right\|} 
\newcommand{\M}{\mathcal{M}}
\newcommand{\pa}{\operatorname{P_a}}
\newcommand{\wpa}{\operatorname{\widehat{P}_a}}
\newcommand{\wpb}{\operatorname{\widehat{P}_b}}
\newcommand{\po}{\operatorname{P_0}}
\newcommand{\pag}{\operatorname{P_aG}}
\newcommand{\pog}{\operatorname{P_0G}}
\newcommand{\pbg}{\operatorname{P_bG}}

\makeatletter
\newcommand{\customlabel}[2]{\protected@write\@auxout{}{\string\newlabel{#1}{{#2}{\thepage}{#2}{#1}{}}}\hypertarget{#1}{#2}}
\makeatother

\begin{document}

\title[Deviation of Birkhoff integrals for locally Hamiltonian flows]
{New phenomena in deviation of Birkhoff integrals for locally Hamiltonian flows}

\author[K.\ Fr\k{a}czek]{Krzysztof Fr\k{a}czek}
\address{Faculty of Mathematics and Computer Science, Nicolaus
Copernicus University, ul. Chopina 12/18, 87-100 Toru\'n, Poland}
\email{fraczek@mat.umk.pl}

\author[M. Kim]{Minsung Kim}
\email{mkim16@mat.umk.pl}

\address{}
\email{}
\date{\today}

\thanks{}
\maketitle

\begin{abstract}
We consider smooth locally Hamiltonian flows on compact surfaces of genus $g\geq 2$ to prove their deviation of Birkhoff integrals for smooth observables.
Our work generalizes results of Forni \cite{Fo2} and Bufetov \cite{Bu} which prove the existence of deviation spectrum of Birkhoff integrals for observables whose jets vanish at sufficiently high order around fixed points of the flow.
They showed that ergodic integrals can display a power spectrum of behaviours with
exactly $g$ positive exponents related to the positive Lyapunov  exponents of the cocycle so-called Kontsevich-Zorich, a renormalization cocycle over the Teichm\"uller flow on a stratum of the moduli space of translation surfaces.

Our paper extends the study of the spectrum of deviations of ergodic integrals beyond the case of observables whose jets vanish at sufficiently high order around fixed points.
We  prove the existence of some extra terms in deviation spectrum related to non-vanishing of the derivatives of observables at fixed points. The proof of this new phenomenon is based on tools developed in
the recent work of the first author and  Ulcigrai \cite{Fr-Ul2} for locally Hamiltonian flows having only (simple) non-degenerate saddles.
In full generality, due to the occurrence of (multiple) degenerate saddles,  we introduce new methods of handling functions with polynomial singularities over a full measure set of IETs.
\end{abstract}

\section{Introduction}

Let $M$ be a smooth compact connected orientable surface of genus $g\geq 1$.
We  consider a smooth flow $\psi_\R = (\psi_t)_{t\in\R}$ on $M$ preserving a smooth measure $\mu$ with positive density.
These flows are called \emph{locally Hamiltonian flows} in the sense that around any point on $M$, we can find local coordinates $(x,y)$
in which $d\mu=dx\wedge dy$ and $\psi_\R$ is a solution to the Hamiltonian equation
\[
\begin{cases}
\frac{dx}{dt} = \frac{\partial H}{\partial y}(x,y)\\
\frac{dy}{dt} = -\frac{\partial H}{\partial x}(x,y)
\end{cases}
\]
for a smooth real-valued function $H$.

For any smooth observable $f:M\to\R$ we are interested in understanding the asymptotics of the growth of \emph{ergodic integrals} (\emph{Birkhoff integrals})
 $\int_0^Tf(\psi_t x)\,dt$ as $T\to+\infty$.

\medskip

We always assume that all fixed points of the flow $\psi_\R$ are isolated, so the set of fixed points of flow $\psi_\R$, denoted  by $\mathrm{Fix}(\psi_\R)$, is finite. For $g \geq 2$, $\mathrm{Fix}(\psi_\R)$ is non-empty. As $\psi_\R$ is area-preserving, singularities are either centers, simple saddles or multi-saddles (saddles with $2k$ prongs with $k \geq 2$). We will deal only with saddles defined as follows:
a fixed point  $\sigma\in \mathrm{Fix}(\psi_\R)$ is a saddle of multiplicity $m=m_\sigma\geq 2$ if there exists a chart $(x,y)$ (called \emph{a singular chart}) in a neighborhood $U_\sigma$ of $\sigma$ such that
$d\mu=V(x,y)dx\wedge dy$ and $H(x,y)=\Im (x+iy)^m$ ($(0,0)$ are coordinates of $\sigma$). Then the corresponding local Hamiltonian equation in $U_\sigma$ is of the form
\[\frac{dx}{dt}=\frac{\frac{\partial H}{\partial y}(x,y)}{V(x,y)}=\frac{m\Re(x+iy)^{m-1}}{V(x,y)} \quad \text{and}\quad
\frac{dy}{dt}=-\frac{\frac{\partial H}{\partial x}(x,y)}{V(x,y)}=-\frac{m\Im(x+iy)^{m-1}}{V(x,y)}.
\]
If $m_\sigma=2$ then we say that the fixed point $\sigma$ is non-degenerate (the Hamiltonian $H$ is non-degenerate at $\sigma$) or it is a simple saddle.

The phenomenon of deviation spectrum and its relation with Lyapunov exponents were first observed by Zorich \cite{Zo}
in the context of studying deviations of Birkhoff (ergodic) sums for piecewise constant observables for almost all interval exchange translations.
Inspired by this result, Kontsevich and Zorich in \cite{Ko} and \cite{Ko-Zo} formulated the following conjecture:
for almost every locally Hamiltonian flow $\psi_\R$ with non-degenerate fixed points
there exist Lyapunov exponents $\nu_i$, $1\leq i\leq g$ so that
for every smooth map $f:M\to\R$ there exists $1\leq i\leq g$ such that
\[\limsup_{T\to+\infty}\frac{\log\left|\int_0^Tf(\psi_t(x))\,dt\right|}{\log T}=\nu_i\text{ for almost every }x\in M.\]
The exponents $\nu_i$, $1\leq i\leq g$ are the positive Lyapunov exponents of the Kotsevich-Zorich cocycle.
This conjecture was positively verified in the seminal paper by Forni \cite{Fo2} and later developed in Bufetov's paper \cite{Bu} for a certain family of observables $f$. More precisely,
for almost every locally Hamiltonian flow $\psi_\R$ without saddle connections (here we do not demand that all saddles are non-degenerate) there are $g$ cocycles $u_i:\R\times M\to\R$, $1\leq i\leq g$ (i.e.\ $u_i(t+s,x)=u_i(t,x)+u_i(s,\psi_t x)$ for all $t,s\in\R$), introduced by Bufetov,
such that for every observable $f:M\to\R$ from the weighted Sobolev space $H^1_W(M)$ defined in \cite{Fo1}\footnote{Recall that $f\in H^1_W(M)$ if $f/W$ belongs to the standard Sobolev space $H^1(M)$, where
$W:M\to\R_{\geq 0}$ is a smooth change of velocity which is positive outside $\mathrm{Fix}(\psi_\R)$ and $W(x,y)=(x^2+y^2)^{m_\sigma}/V(x,y)$ in singular coordinates on $U_\sigma$.}
 or satisfying a weak Lipschitz condition defined in \cite{Bu} we have
\begin{equation}\label{eq:sumbasic}
\int_0^T f(\psi_t(x))dt  =  \sum_{i=1}^g {D_i}(f)u_i(T,x)  + err(f,T,x),
\end{equation}
where
\begin{gather}\label{eq:growthbasic}
\limsup_{T\to+\infty}\frac{\log\big|\int_0^Tu_i(t,x)\,dt\big|}{\log T}=\nu_i,\
\lim_{T\to+\infty}\frac{\log|err(f,T,x)|}{\log T}=0\text{ for a.e. }x\in M,
\end{gather}
and the coefficients ${D_i}(f)$, $1\leq i\leq g$ are given by invariant distribution ${D_i}$, $1\leq i\leq g$.\footnote{The distributions ${D_i}$, $1\leq i\leq g$ are
called Forni's invariant distributions and play a crucial role in solving cohomological equations, see \cite{Fo1,Fo3}.}
However, even for non-degenerate saddles, this does not fully solve the original conjecture because every smooth observable belonging to $H^1_W(M)$ has to vanish at each simple saddle and its derivative is also zero at this point. More generally, if $f$ is smooth and belongs to $H^1_W(M)$ then for every $\sigma\in \mathrm{Fix}(\psi_\R)$ we have $f^{(j)}(\sigma)=0$ for all $0\leq j<m_\sigma$.
This impediment has recently been overcome, for non-degenerate fixed points, in the recent work of the first author and Ulcigrai \cite{Fr-Ul2} using techniques inspired by tools introduced by
Marmi-Moussa-Yoccoz paper \cite{Ma-Mo-Yo} on
solving the cohomological equation for (Roth-type) interval exchange transformations
(and the follow up article \cite{Ma-Yo} by Marmi and Yoccoz). In \cite{Fr-Ul2} the authors proved  \eqref{eq:sumbasic} with \eqref{eq:growthbasic} for a.e.\ non-degenerate locally Hamiltonian flow $\psi_\R$
restricted to any of its minimal component and any smooth observable $f:M\to\R$, moreover a somewhat deeper analysis of the asymptotics of the error term is provided.

\medskip

In the present paper we take a step beyond the original conjecture by considering the deviation problem in full generality.
The main result of this paper concerns the deviation of ergodic averages for a.a.\ locally Hamiltonian flows that allow saddles of any multiplicity
and for all smooth observables. More precisely, we deal with a locally Hamiltonian flow restricted to any of its minimal component. Recall that $M$ splits into
a finite number of components (surfaces with the boundary) such that the interior of each component is filled by periodic orbits or is minimal (every orbit is dense in the component).
The possibility of saddles $\sigma\in \mathrm{Fix}(\psi_\R)$ of higher multiplicity and the non-vanishing of an observable or its low-order derivative (of order $<m_\sigma-2$) at these saddles causes the deviation spectrum of the form \eqref{eq:sumbasic} not longer to occur. In this situation, a typical orbit stays around each fixed point long enough so that this affects the asymptotics of the ergodic integrals. This leads to a new phenomenon in the study of the deviation of ergodic integrals, which is displayed by a new family of \emph{singular cocycles}
and invariant distributions associated with each fixed point (partial derivatives of Dirac distributions).
They appear as new elements in the deviation spectrum and give rise to a new type of polynomial oscillation with exponents that are of a different nature than the Lyapunov exponents of the Kontsevich-Zorich cocycle. The new invariant distributions responsible for the intensity of the new type of oscillations are defined locally in a very simple way (in contrast to the distributions $D_i$, $1\leq i \leq g$ which are global objects without any explicit formula) using partial derivatives as follows:
if $\sigma\in \mathrm{Fix}(\psi_\R)$ is saddle and $(x,y)$ is a singular chart in the neighborhood $U_\sigma$
(i.e.\ $d\mu=V(x,y)dx\wedge dy$ and $H(x,y)=\Im (x+iy)^{m_\sigma}$) then for every $\alpha=(\alpha_1,\alpha_2)\in\Z_{\geq 0}\times\Z_{\geq 0}$ with $|\alpha|=\alpha_1+\alpha_2\leq m_\sigma$ and any $C^{m_\sigma}$-map $f$ we set $\partial_\sigma^\alpha(f):=\frac{\partial^{|\alpha|}(f\cdot V)}{\partial^{\alpha_1} x\partial^{\alpha_2} y}(0,0)$. Moreover, for every $\sigma\in \mathrm{Fix}(\psi_\R)$ and $0\leq k\leq m_\sigma-2$, let
\begin{equation}
b(\sigma,k) = \frac{m_\sigma-2 - k}{m_\sigma}.
\end{equation}
As it is shown in Lemma~\ref{lem;partial-invariance}, for every $\alpha\in\Z_{\geq 0}\times\Z_{\geq 0}$ with $|\alpha|=\alpha_1+\alpha_2\leq m_\sigma-2$ the distribution $\partial_\sigma^\alpha:C^{m_\sigma}(M)\to\R$ is $\psi_\R$-invariant.


\begin{theorem}\label{theorem;main}
 Let $\psi_\R$ be a locally Hamiltonian flow  on a compact surface $M$ and let $M'$ be its minimal component of genus $g \geq 1$. Let $m:=\max\{m_\sigma:\sigma\in \mathrm{Fix}(\psi_\R)\cap M'\}$. For almost every $\psi_\R$
there exist Lyapunov exponents
\[
1:= \nu_1 > \nu_2 > \cdots > \nu_g >0,
\]
invariant distributions ${D_i}:C^m(M)\to\R$, $1\leq i\leq g$, smooth cocycles $u_i(T,x) : \R \times M \rightarrow \R$, $1\leq i\leq g$ and smooth cocycles $c_{\sigma,\alpha}(T,x):  \R \times M \rightarrow \R$ for all $\sigma\in \mathrm{Fix}(\psi_\R)\cap M'$ and $\alpha\in\Z_{\geq 0}\times\Z_{\geq 0}$ with $|\alpha|<m_\sigma-2$ such that for every  $f \in C^{m}(M)$,
\begin{equation}\label{eqn;dev-complete}
\int_0^T f(\psi_t(x))dt  = \sum_{\sigma \in \mathrm{Fix}(\psi_\R)\cap M'}\sum_{\substack{\alpha\in\Z^2_{\geq 0}\\|\alpha| < m_\sigma-2}}
\partial_\sigma^\alpha (f) c_{\sigma,\alpha}(T,x)  + \sum_{i=1}^g {D_i}(f)u_i(T,x)  + err(f,T,x)
\end{equation}
with
\begin{align}
&\limsup_{T \rightarrow \infty} \frac{\log{|c_{\sigma,\alpha}(T,x)|}}{\log T}  \leq b(\sigma,|\alpha|)\text{ for a.e. }x\in M'; \label{res;dev1}\\
&\limsup_{T \rightarrow \infty} \frac{\log{\norm{c_{\sigma,\alpha}(T,\cdot)}_{L^1(M')}}}{\log T}  \leq b(\sigma,|\alpha|) \label{res;dev2}
\end{align}
for all $\sigma\in \mathrm{Fix}(\psi_\R)\cap M'$ and $\alpha\in\Z_{\geq 0}\times\Z_{\geq 0}$ with $|\alpha|<m_\sigma-2$,
\begin{align}
&\limsup_{T \rightarrow \infty} \frac{\log{| u_i(T,x)|}}{\log T}  =  \nu_i \text{ for a.e. }x\in M'; \label{res;dev3}\\
&\limsup_{T \rightarrow \infty} \frac{\log{\norm{ u_i(T,\cdot)}_{L^1(M')}}}{\log T}  =  \nu_i \label{res;dev4}
\end{align}
for every $0\leq i\leq g$ and if $err\neq 0$ then
\begin{align}
&\limsup_{T \rightarrow \infty} \frac{\log|err(f,T,x)|}{\log T}  = 0 \text{ for a.e. }x\in M'; \label{eqn;dev-remain}\\
& \limsup_{T\to+\infty}\frac{\log\|err(f,T,\cdot)\|_{L^{1}(M')}}{\log T} = 0. \label{eqn;dev-remain-lp}
\end{align}

Assume additionally that the flow $\psi_\R$ is  minimal on $M$ and let any $\sigma\in \mathrm{Fix}(\psi_\R)$ and $0\leq k< m_\sigma-2$.
Then for every  $f \in C^{m}(M)$
which vanishes on $\bigcup_{\sigma'\in \mathrm{Fix}(\psi_\R)\setminus\{\sigma\}} U_{\sigma'}$ and
such that $f^{(j)}(\sigma)=0$ for all $0\leq j<k$ and $f^{(k)}(\sigma)\neq 0$ we have
\begin{equation}\label{eqn;lowerbndb}
\limsup_{T \rightarrow \infty} \frac{\log{\left|\int_0^Tf(\psi_tx)\,dt\right|}}{\log T}  \geq b(\sigma,k) \text{ for a.e. }x\in M.
\end{equation}
In particular, for every $\sigma\in \mathrm{Fix}(\psi_\R)$ and $\alpha\in\Z^2_{\geq 0}$ with $|\alpha|<m_\sigma-2$ we have
\begin{align}\label{eqn;lowerbnc}
&\limsup_{T \rightarrow \infty} \frac{\log{|c_{\sigma,\alpha}(T,x)|}}{\log T}  = b(\sigma,|\alpha|)\text{ for a.e. }x\in M.
\end{align}
\end{theorem}

In the case of the locally Hamiltionian flows with only non-degenerate saddles, i.e.\ $m_\sigma=2$, new terms in the deviation spectrum do not appear,
this leads to the recent result of
the first author and Ulcigrai. We see the same effect when the smooth observable $f$ belongs to the Sobolev space $H^1_W(M)$.
Then $\partial_\sigma^\alpha(f)=0$ for every $\sigma\in \mathrm{Fix}(\psi_\R)\cap M'$ and $\alpha\in\Z_{\geq 0}^2$ with $|\alpha|<m_\sigma-2$.
This leads to an extension (we do not assume that $M'=M$) of the classical results by Forni and Bufetov in the smooth framework.

\subsection{Methods and outline} Let us introduce the main steps in the proof.
The general strategy starts by choosing a Poincar\'e map (first return map) for the locally Hamiltonian (area-preserving) flow $\psi_\R$.
Poincar\'e maps for the flow $\psi_\R$ are known to be interval exchange transformations (IET) $T:I\to I$ for $I = [0,1)$ (see \S\ref{sec;IET}).
Any minimal component of the locally Hamiltonian flow admits a representation called \emph{special flow} over an IET $T$.
The \emph{roof function} $g: I \rightarrow \R_{>0}\cup\{+\infty\}$ which arises from this representation is piecewise smooth and  has
\emph{singularities} at discontinuities $e \in I$ of $T$. In particular, (degenerate) multi-saddles of $\psi_\R$ are responsible
for the appearance of singularities of \emph{polynomial type}, specifically if $x \rightarrow e^{\pm}$, then the roof function $g(x)$
blows up  polynomially, i.e.\ $g(x)\sim C^\pm_e/|(x-e)^a|$ for some $0<a<1$ and the constants  $C^{\pm}_e$ are non-negative.
Simple saddles (non-degenerate) are responsible for the appearance of singularities of \emph{logarithmic type}, i.e.\ $g(x)\sim C^\pm_e|\log(x-e)|$.

Let $f:M \rightarrow \R$ be smooth observable. 
To study the deviation of ergodic integrals of $f$ for the flow $\psi_\R$ on $M$, we consider a cocycle $\varphi_f:I\to\R$
associated with the observable $f$.
The cocycle $\varphi_f(x)$ is defined as the integral of $f$ along the $\psi_\R$-orbit segment starting from $x$ until
its first return to $I$. This cocycle is also piecewise smooth and has polynomial and logarithmic singularities but
the corresponding constants $C_e^\pm(\varphi_f)$ can be positive, negative or zero.  Transition to the cocycle $\varphi_f$ enables
to reduce the deviation of Birkhoff integrals for $f$
to the deviation of Birkhoff sums of the cocycle $\varphi_f$ with {polynomial and logarithmic singularities} over IETs.

One of the new developments in this paper firstly appears in \S\ref{sec;polycocycle} and  involves the introduction of new Banach spaces,
$\pa$ (or $\wpa$) for $0\leq a<1$, containing functions with polynomial singularities of degree at most $a$.
We prove that there is a correspondence between smooth observables $f :M\to\R$ and  the cocycles $\varphi_f \in \pa$ for any $0 \leq a<1$
(see Theorem \ref{thm:ftophi}).
 In particular, $\po$ consists of functions with logarithmic singularities.
 This case was already extensively studied by  the first author and Ulcigrai in  \cite{Fr-Ul,Fr-Ul2}.
 In their study, functions with logarithmic singularities had a canonical decomposition into a purely logarithmic part and
 a piecewise absolutely continuous part.  
 In general ($0<a<1$),  cocycles in $ \pa$ cannot be decomposed as before, with purely polynomial part.
 It means that we cannot control the growth of cocycles directly by the norms of the same type as those provided in previous work.

In \S\ref{sec;DC}, we introduce a new Diophantine condition for IETs, called the \emph{Filtration Diophantine Condition} (\ref{FDC}).
This condition is characterized by  matrices of the (accelerated) Kontsevich-Zorich cocycle, and it imposes a growth of
 matrices of the cocycle and a uniformly hyperbolic behaviour of the matrix products. In particular, our condition requires a
 control over cocycles on each subspaces $U_j$ in the Oseledets filtration of the unstable space. This {effective} Oseledets
 control is closely related to the recent approach of the first author and Ulcigrai \cite{Fr-Ul2} and Ghazouani-Ulcigrai \cite{Gh-Ul}
 (we also refer a recent survey of Ulcigrai about Diophantine conditions \cite{Ul:ICM}).

Next ingredients are special Birkhoff sums $S(k)(\varphi)$ of cocycles $\varphi$ and \emph{correction operators} $\mathfrak{h}_j$ in
\S\ref{sect;renormalization} and \S\ref{sec;correction}. The special Birkhoff sum operators are mainly used in controlling bounds
of the sequence of all Birkhoff sums for $\varphi$ in \S\ref{sec;deviation}. In \S\ref{sect;renormalization}, we review the properties of $S(k)$ on the space $\pa$.

The main idea of the construction of correction operators $\mathfrak{h}_j$ in \S\ref{sec;correction} is inspired by the work of Marmi-Moussa-Yoccoz \cite{Ma-Mo-Yo}
in solving cohomological equations of IETs for Roth types. The correction operator was defined there by subtracting piecewise constant functions
from the piecewise absolutely continuous cocycles $\varphi$ in order to solve cohomological equations. In \cite{Fr-Ul,Fr-Ul2},
the  first author and Ulcigrai extended the construction of such correction operator to cocycles with logarithmic singularities to get
a better control of the growth of special Birkhoff sums $S(k)(\varphi)$ and all Birkhoff sums. They treated $L^1$-norm instead of the uniform norm used in the previous approach.
After correction of a cocycle $\varphi$ by a piecewise constant
function, the sequence $S(k)(\varphi)$ has subexponential growth and it is bounded along a subsequence, whenever logarithmic singularities of $\varphi$
are of symmetric type.

In our work, cocycles $\varphi$ from $\pa$ with $0<a<1$ cannot be corrected by piecewise constant functions to have the same conclusion.
In \S\ref{sec;correction}, we prove that for every $\varphi\in\pa$ after an optimal correction the corresponding sequence $S(k)(\varphi)$ has
exponential growth with the exponent $\lambda_1 a$, where $\lambda_1$ is the top Lyapunov exponent of the K-Z cocycle. The piecewise
constant correction is given by  $\mathfrak{h}_j(\varphi)\in U_j$ for some $j$ depending on $a$.
 This phenomenon plays a crucial role in showing the new form of deviation spectrum in  later sections.

Main novelty of this paper is the discovery of new phenomena in the deviation spectrum of Birkhoff integrals.
In our main Theorem~\ref{theorem;main}, we have two different kinds of cocycles: $u_i$ and $c_{\sigma,\alpha}$. The counterparts of the cocycles
$u_i$ were previously studied by Bufetov and Forni; the cocycles $u_i$ have polynomial growth determined by the Lyapunov exponents of the Kontsevich-Zorich cocycle.
The Bufetov-Forni deviation spectrum was recently improved by the first author and Ulcigrai for locally Hamiltonian flows $\psi_\R$ with simple saddles and observables $f$ non-vanishing on the set of fixed points.
Then the corresponding cocycle $\varphi_f$ has logarithmic singularities.


In our work,
the cocycles $c_{\sigma,\alpha}$ are constructed as the Birkhoff integrals of functions, which are locally supported on the neighborhood $U_\sigma$.
Each such function is corrected to remove the influence of the Lyapunov exponents of the K-Z cocycle, by using the correcting operators
defined in \S\ref{sec;reduction}.
The appearance of new phenomena in deviation spectrum is also related to the existence of new $\psi_\R$-invariant distributions
$\partial_{\sigma}^\alpha$ responsible for the behaviour of $f$ around singularities $\sigma$ and defined as partial derivatives of $f$ at $\sigma$.

One of the most important tools developed in this paper are the results of \S\ref{sec:locHam} where we prove some local relations between
``flatness'' of the observable $f$ around saddles and the types of singularity for the associated cocycle $\varphi_f$.
 These results  play a key role in proving Theorem~\ref{thm:ftophi} and its extension in \S\ref{sec:regularity}.
  In fact, we generalize the approach developed for simple saddles in \cite{Fr-Ul} (related to logarithmic singularity type) to multi-saddle type.
This part is also the key ingredient in the detailed proof of our new phenomena.

Finally, to give an idea of technical details for upper and lower bounds for Birkhoff integrals of cocycles,
we return to the {special flow} representation over IETs. In \S\ref{sec;deviation},
we present some results on the deviation of Birkhoff integrals for special flows built
over IETs satisfying \ref{FDC} and under roof functions in $\pa$. The upper bounds for such Birkhoff integrals
are related to the growth of the sequence $S(k)(\varphi_f)$. This provides the opportunity to use the analysis of growth  of $S(k)(\varphi_f)$
developed in \S\ref{sec;correction}.
Some of the results of \S\ref{sec;deviation} are creative extensions of ideas introduced in \cite{Fr-Ul2} for roof functions
with logarithmic singularities.
However, the existence of polynomial singularities significantly complicates the arguments.

Lower bounds rely on some Borel-Cantelli type argument. Here, we choose a sequence of return times for the special flows to a shrinking sequence of segments
originating from Rauzy-Veech induction.
 By ergodicity of the flow and  the Rokhlin tower condition on the base IETs (see Definition~\ref{def:RTC}) stated in \S\ref{sec;FDC-RTC},
 a Borel-Cantelli type argument can be applied.

\subsection{Structure of the paper}
Let us outline the structure of the paper.
In \S\ref{sec;IET&RV}, we recall the basic definition of locally Hamiltonian flows.
We summarize their relations with  special flows over IETs, Rauzy-Veech induction and accelerations of the Kontsevich-Zorich cocycle.
In \S\ref{sec;DC}, we introduce Oseledets filtration  of accelerated KZ-cocycles and formulate Diophantine Condition (\ref{FDC})
associated  with such filtrations. \ref{FDC} is used in constructing correction operators in \S\ref{sec;correction}.

In \S\ref{sec;polycocycle}, the spaces of cocycles ($\pa$ and $\wpa$) with polynomial singularities over IET are defined and
their basic properties are introduced.
In \S\ref{sect;renormalization}, we review renormalization operators and special Birkhoff sums.
In \S\ref{sec;correction}, correction operators for cocycles with polynomial singularities are constructed under \ref{FDC} conditions.

The asymptotic deviation spectrum  for special flows over IETs is studied in \S\ref{sec;deviation}.
In \S\ref{sec:locHam}, we relate some local properties of the observable $f$ around saddles $\sigma$
with types of singularity for $\varphi_f$.
In \S\ref{sec:regularity},  some global properties of the correspondence between an observable $f$ and the cocycle $\varphi_f$
with polynomial singularities are established.
Moreover, the concept of correcting operators for observables $f$ is introduced, later applied to the construction of cocycles $c_{\sigma,\alpha}$ in \S\ref{sec:pmt}.
 Finally, in \S\ref{sec:pmt},  the proof of the main Theorem~\ref{theorem;main} is presented. In particular, this section contains the proof of lower bounds.
 In Appendix \ref{sec:App1}, we prove that almost every IET satisfies \ref{FDC}.

\section{Preliminary materials}\label{sec;IET&RV}
In this section we give a review of basic tools and definitions concerning locally Hamiltonian flows and
its reduction to special flows, interval exchange transformation, and Rauzy-Veech induction.
We recall here some basic definitions and introduce the notation we used throughout the paper.
For comprehensive introduction to the subject we refer the reader to \cite{Fr-Ul2,Ul:ICM,Rav}.

\subsection{Locally Hamiltoniain flows on surfaces}\label{sec;LHF}
Let $M$ be a smooth compact connected orientable surface of genus $g\geq 1$ and fix
a smooth area form $\omega$ on $M$ (which in local coordinates is given by $V(x,y) d x \wedge d y$ for some smooth positive real-valued function  $V$). The corresponding area measure is denoted by $\mu$.  Let $X:M\to TM$ be a smooth tangent vector field with finitely many fixed points
and such that the corresponding flow $\psi_\R$ preserves the smooth area form $\omega$ (or equivalently the area measure $\mu$). 
These flows  are often called \emph{locally Hamiltonian flows} or \emph{multi-valued Hamiltonian flows}.

\subsubsection{Relation with closed $1$-forms}
Such flows are in $1-1$ correspondence with smooth closed real-valued $1$-forms. For every vector field $X$ preserving $\omega$ let $\eta:=\imath_X\omega=\omega(X, \,\cdot \,)$ be the corresponding $1$-form, where $\imath_X$ denotes the contraction operator. Since $\eta$ is a smooth closed $1$-form (i.e. $d\eta=0$), for any $p\in M$ and any simply connected neighbourhood $U$ of $p$ there exists a smooth (local Hamiltonian) map (unique up to additive constant)
such that $dH=\eta$ on $U$. It follows that the flow $\psi_\R$  is the (local) solution of the Hamiltonian equation $p'=X(p)$ with $dH=\omega(X, \,\cdot \,)$.
In local coordinates, if $\omega(x,y)=V(x,y) d x \wedge d y$ then the Hamiltonian equation is of the form $x'=\frac{\partial H}{\partial y}/V$,
$y'=-\frac{\partial H}{\partial x}/V$.

We denote by $\mathrm{Fix}(\psi_\R)$ the
set of fixed points of $\psi_\mathbb{R}$, i.e.\ the set of zeros of the form $\eta$.
If the form $\eta$ is Morse, locally differential of a \emph{Morse function} (i.e. the Hessian at every fixed point is non-zero), it corresponds to either a center  or a simple saddle. More precisely, by Morse lemma, for every fixed point $\sigma\in \mathrm{Fix}(\psi_\R)$ there exists a local chart $(x,y)$ in a neighborhood $U_\sigma$ of $\sigma$
such that $H(x,y)=x^2+y^2$ or $H(x,y)=2xy$. In this paper, we permit the appearance of  \emph{degenerate} fixed points, i.e.\ saddles of multiplicity $m>2$ such that the corresponding Hamiltonian function is of the form $H(x,y)=\Im (x+iy)^m$. More precise description of the local behaviour of $\psi_\R$ around multi-saddles is postponed to \S\ref{sec:locHam}.

\subsubsection{Minimality vs.\ minimal components} For more detailed explanation of the space of locally Hamiltonian flows and the partition of the surface into invariant components, we refer the reader to \cite{Rav}.
We call a \emph{saddle connection} an orbit of $\psi_\R$ running from a saddle to another saddle. A \emph{saddle loop} is a saddle connection
joining the same saddle.
If every fixed point in $\mathrm{Fix}(\psi_\R)$ is isolated, $M$ splits into a finite number of $\psi_\R$-invariant surfaces (with boundary) so that every such surface is a \emph{minimal component} of  $\psi_\R$ (every orbit, except of fixed points, is dense in the component) or is a periodic component (its interior is filled by periodic orbits and its boundary consists of saddle connections).

Denote by $\mathcal{F}$ the set of smooth locally Hamiltonian flows (or equivalently smooth closed $1$-forms) on $M$ with isolated zeros.
The set $\mathcal{F}$ is equipped with a topology by considering smooth perturbations
of closed smooth $1$-forms. For any vector $\overline{m}=(m_1,m_2,\ldots,m_s)$ of natural numbers $\geq 2$ and any $c\leq\sum_{i=1}^s(m_i-1)$, denote by  $\mathcal{F}_{\overline{m},c}$
the set of smooth locally Hamiltonian flows with $c$ centers and $s$ saddles of multiplicity $m_1,m_2,\ldots,m_s$. By the Poincar\'e-Hopf Theorem, $c-\sum_{i=1}^s(m_i-1)=2-2g$.
A measure-theoretical notion of typicality on $\mathcal{F}$ (on each  $\mathcal{F}_{\overline{m},c}$ separately) is defined by the cohomology class of the 1-form $\eta$,
 so called \emph{Katok fundamental class} (introduced by Katok in \cite{Ka0}).
Let $\gamma_1, \dots, \gamma_n$ be a base of $H_1(M, \mathrm{Fix}(\psi_\R), \mathbb{R})$, where $n=2g+s+c-1$.
Let us consider the period map
\[\Theta(\psi_\R)=\Big(\int_{\gamma_1}\eta,\ldots,\int_{\gamma_n}\eta\Big)\in\R^n,\]
which is well-defined in a neighbourhood of $\psi_\R \in \mathcal{F}_{\overline{m},c}$.
The $\Theta$-pullback of the Lebesgue measure class (i.e.\ class of sets with zero measure) gives the desired measure class on $\mathcal{F}_{\overline{m},c}$. When we use the expression \emph{a.e.\ locally Hamiltonian flow}  below  we mean full measure in each $\mathcal{F}_{\overline{m},c}$  with respect to the corresponding measure class.
We distinguish a subset $\mathcal{F}_{\min}=\bigcup_{\overline{m}}\mathcal{F}_{\overline{m},0}\subset \mathcal{F}$ and the corresponding measure class. Then a.e.\ flow $\psi_\R\in\mathcal{F}_{\min}$ is minimal.
On the other hand, every $\psi_\R\in\mathcal{F}\setminus\mathcal{F}_{\min}$ has a nontrivial splitting into minimal and periodic components.
Then we only deal this with the flow $\psi_\R$ restricted to any of its minimal component $M'\subset M$.

We should mention that mixing properties of a.e.\ non-degenerate locally Hamiltonian flow restricted to minimal components
are fully described in \cite{Ch-Wr,Rav,Sch,Ul:mix,Ul:wea,Ul:abs,Ul:ICM}. When degenerate saddles appear, mixing was proved by Kochergin in \cite{Ko:mix}.
On the other hand, understanding  more subtle spectral properties of locally Hamiltonian flows seems to be still in its infancy.
Only recently the first results have appeared for the singular spectrum in \cite{ChFKU} and the countable Lebesgue spectrum in \cite{Fa-Fo-Ka}.

\subsection{Interval exchange transformations (IET)}\label{sec;IET}
 To define an IET we adopt the notation
from \cite{ViB}. Let $\mathcal{A}$
be a $d$-element alphabet and let $\pi=(\pi_0,\pi_1)$ be a pair of
bijections $\pi_\vep:\mathcal{A}\to\{1,\ldots,d\}$ for $\vep=0,1$.
For every $\lambda=(\lambda_\alpha)_{\alpha\in\mathcal{A}}\in
\R_{>0}^{\mathcal{A}}$ let
$|\lambda|:=\sum_{\alpha\in\mathcal{A}}\lambda_\alpha$, $I:=\left[0,|\lambda|\right)$ and for every $\alpha\in\mathcal{A}$,
\begin{gather*}
 I_{\alpha}:=[l_\alpha,r_\alpha),\text{ where
}l_\alpha=\sum_{\pi_0(\beta)<\pi_0(\alpha)}\lambda_\beta,\;\;\;r_\alpha
=\sum_{\pi_0(\beta)\leq\pi_0(\alpha)}\lambda_\beta \\
 I'_{\alpha}:=[l'_\alpha,r'_\alpha),\text{ where
}l'_\alpha=\sum_{\pi_1(\beta)<\pi_1(\alpha)}\lambda_\beta,\;\;\;r'_\alpha
=\sum_{\pi_1(\beta)\leq\pi_1(\alpha)}\lambda_\beta.\
\end{gather*}
The \emph{interval exchange transformation} $T = T_{(\pi,\lambda)}$ given by the data $(\pi, \lambda )$ is the
orientation preserving piecewise isometry
$T_{(\pi,\lambda)}:I\to I$ which,
for each $\alpha \in \mathcal{A}$,  maps the interval $I_{\alpha}$
isometrically onto the interval $I'_{\alpha}$. Clearly $T$
preserves the Lebesgue measure on $I$. 

 Denote by
$\mathcal{S}^0_{\mathcal{A}}$ the subset of irreducible pairs,
i.e.\ such that
$\pi_1\circ\pi_0^{-1}\{1,\ldots,k\}\neq\{1,\ldots,k\}$ for $1\leq
k<d$.  We will always assume that $\pi \in
\mathcal{S}^0_{\mathcal{A}}$. The IET $T_{(\pi,\lambda)}$ is
explicitly given by $T(x)=x+w_\alpha$ for $x\in I_\alpha$, where
$w=\Omega_\pi\lambda$ and $\Omega_\pi$ is  the matrix
$[\Omega_{\alpha\,\beta}]_{\alpha,\beta\in\mathcal{A}}$ given by
\[\Omega_{\alpha\,\beta}=
\left\{\begin{array}{cl} +1 & \text{ if
}\pi_1(\alpha)>\pi_1(\beta)\text{ and
}\pi_0(\alpha)<\pi_0(\beta),\\
-1 & \text{ if }\pi_1(\alpha)<\pi_1(\beta)\text{ and
}\pi_0(\alpha)>\pi_0(\beta),\\
0& \text{ in all other cases.}
\end{array}\right.\]
We use also an alternative description of IET. Let  $\widehat{I}=(0,|I|]$ and denote by $\widehat{T}_{(\pi,\lambda)}:\widehat{I} \to \widehat{I}$ the
exchange of the intervals
$\widehat{I}_\alpha:=(l_\alpha,r_\alpha]$, $\alpha\in\mathcal{A}$,
i.e.\ $\widehat{T}_{(\pi,\lambda)}x=x+w_\alpha$  for $x\in \widehat{I}_\alpha$.

Let $End(T) = \{ l_\alpha, r_\alpha, \alpha \in \mathcal{A}\}$ stand for the set of end points of the intervals
$I_\alpha, \alpha\in\mathcal{A}$.
 A pair ${(\pi,\lambda)}$
satisfies the {\em Keane condition} (see \cite{Keane}) if
$T_{(\pi,\lambda)}^m l_{\alpha}\neq l_{\beta}$ for all $m\geq 1$
and for all $\alpha,\beta\in\mathcal{A}$ with $\pi_0(\beta)\neq 1$.

\subsection{Cocycles and special flows over IETs}\label{sec;specialflow}
Let $T:I\to I$ be an ergodic IET. Each measurable function $\varphi:I \rightarrow \R$ determines an
{\em additive cocycle} $\varphi^{(\,\cdot\,)}(\,\cdot\,)$ 
for $T$ so that
\[
\varphi^{(n)}(x) := \sum_{0\leq k<n}\varphi(T^k(x)) \quad \text{for} \quad n\geq 0\text{ and }x\in I.
\]
Let
$g:I\to\R_{>0}\cup\{+\infty\}$ be an integrable function such that
$\underline{g}=\inf_{x\in I}g(x)>0$. Denote by $(T^g_t)_{t\in\R}=T^g_\R$
 the \emph{special flow} built over the IET $T$ and under the \emph{roof function} $g$ acting on
\[I^g:=\{(x,r)\in I\times \R:0\leq r<g(x)\}\]
so that $T^g_t(x,r)=(x,r+t-g^{(n)}(x))$, where $n$ is the unique
integer number with $g^{(n)}(x)\leq r+t<g^{(n+1)}(x).$

\medskip

Locally Hamiltonian flows  are represented as special flows.
Let us consider a restriction of a locally Hamiltonian flow $\psi_\mathbb{R}$ on $M$ to its minimal component $M'\subset M$.
Let $I\subset M'$ be any transversal smooth curve with its standard parametrization $\gamma:[0,|I|]\to I$, i.e.\ $\int_{0}^{\gamma(s)}\eta=s$
for $s\in [0,|I|]$, where $\eta$
the closed $1$-form $\eta$ associated with the flow $\psi_\mathbb{R}$. By minimality, $I$ is a global transversal and the first return map $T:I\to I$ is an IET in standard coordinates on $I$.
Moreover, $\psi_\mathbb{R}$ restricted to $M'$ is isomorphic to the special flow $T^g_\R$, where  $g:I\to\R_{>0}\cup\{+\infty\}$ is the first return time map.
The roof function has logarithmic (polynomial) singularities derived from non-degenerate (degenerate) saddles. A detailed description of these relations is presented in Theorem~\ref{thm:ftophi}.

Let $f:M\to\R$ be an integrable map. We study the asymptotics of ergodic integrals $\int_0^T f(\psi_t x)\,dt$ using only the return times to the curve $I$. More precisely,
for every $x\in I$ we deal with $\varphi_f(x):=\int_0^{g(x)} f(\psi_t x)\,dt$\footnote{{In fact, the roof function $g$ can be obtained by choosing $f = 1$, i.e $g=\varphi_1$.}}. Then the asymptotic of the cocycle $\varphi^{(n)}(x)$ gives almost full information  about the growth of the ergodic integrals.
The cocycle $\varphi_f:I\to\R$ has also logarithmic and polynomial singularities depending on the multiplicity of saddles, which is described in Theorem~\ref{thm:ftophi}.

The same strategy can be applied to any special flow $T^g_\R$ and any integrable function $f:I^g\to\R$. Then $\varphi_f:I\to\R$ is
given by $\varphi_f(x)=\int_0^{g(x)}f(x,r)\,dr$.
By Fubini's theorem, $\varphi_f$ is well-defined for a.e.\ $x\in I$,
is integrable and
\[\int_I\varphi_f(x)\,dx=\int_{I^g}f(x,r)\,dx\,dr.\]

\begin{remark}
In the reduction of the locally Hamiltonian flow $\psi_\R$ to the special flow $T^g_\R$, one can see that the length of the interval
$I_\alpha$ exchanged by $T$ coincides with one of the coordinates of $\Theta(\psi_\R)$. Hence,
for every subset $A\subset \mathcal{F}_{\overline{m},c}$ of locally Hamiltonian flows, the set $\Theta(A)$ has full Lebesgue
measure if and only if a full measure set of IETs  appears in the base of special flows representations
of flows in $A$.
\end{remark}

\subsection{Rauzy-Veech induction}\label{sec;RVI}
The main tool to study the asymptotics of the cocycle $\varphi_f$ is a standard renormalization procedure called the Rauzy-Veech induction \cite{Ra} and its accelerations.
We refer the reader for some background to the lecture notes by Yoccoz \cite{Yo,Yoc} or Viana \cite{ViB}.

Let $T$ be an IET satisfying Keane's condition and let
$\widetilde I:= \big[0,\max (l_{\pi_0^{-1}(d)}, l_{\pi_1^{-1}(d)}) \big)$.
Denote by $\mathcal{R}(T) = \widetilde T : \widetilde I \to \widetilde I$ the first return map of $T$ to the interval $\widetilde I$.
Let
\begin{eqnarray*}
\epsilon =\epsilon(\pi,\lambda) = \left\{\begin{array}{cll}
0& \text{ if
}&\lambda_{\pi_0^{-1}(d)} > \lambda_{\pi_1^{-1}(d)},\\
1& \text{ if
}&\lambda_{\pi_0^{-1}(d)} < \lambda_{\pi_1^{-1}(d)}.\\ \end{array} \right.
\end{eqnarray*}
Let us consider a pair
$\widetilde{\pi}=(\widetilde{\pi}_0,\widetilde{\pi}_1)\in\mathcal{S}^0_{\mathcal{A}}$,
where
\begin{eqnarray}\label{def:pi}
\begin{split}
\widetilde{\pi}_\vep(\alpha)&=&\pi_\vep(\alpha)
\text{ for all }\alpha\in\mathcal{A}\text{ and }\\
\widetilde{\pi}_{1-\vep}(\alpha)&=&\left\{
\begin{array}{cll}
\pi_{1-\vep}(\alpha)& \text{ if
}&\pi_{1-\vep}(\alpha)\leq\pi_{1-\vep}\circ\pi^{-1}_\vep(d),\\
\pi_{1-\vep}(\alpha)+1& \text{ if
}&\pi_{1-\vep}\circ\pi^{-1}_\vep(d)<\pi_{1-\vep}(\alpha)<d,\\
\pi_{1-\vep}\circ\pi^{-1}_\vep(d)+1& \text{ if
}&\pi_{1-\vep}(\alpha)=d.\end{array} \right.
\end{split}
\end{eqnarray}
Then,  by Rauzy (see \cite{Ra}), $\widetilde T$ is also an IET on $d$-intervals and
$\widetilde T = T_{(\widetilde \pi,\widetilde \lambda)}$ with
\begin{equation*}
\tilde \lambda = A^{-1}(\pi,\lambda)\lambda\text{ and }A(T) = A(\pi,\lambda) = Id+E_{\pi_{\vep}^{-1}(d)\,\pi_{1-\vep}^{-1}(d)} \in SL_{\mathcal A}(\Z),
\end{equation*}
where $Id$ is the identity matrix and $(E_{ij})_{kl} = \delta_{ik}\delta_{jl}$, using the Kronecker delta notation.
Moreover, the renormalized version of the matrix $\Omega$ is of the form
\begin{equation}
\Omega_{\widetilde{\pi}}=A^t(\pi,\lambda)\cdot\Omega_{\pi}\cdot A(\pi,\lambda).
\end{equation}
Thus taking $H(\pi) = \Omega_\pi(\R^{\mathcal A})$, we have $H(\tilde \pi) = A^t(\pi,\lambda) H(\pi)$.

\subsection{Kontsevich-Zorich cocycle and its accelerations}\label{sec;setting}
If an IET $T$ satisfies Keane's condition, then $\widetilde T$ also satisfies Keane's condition and we can generate a sequence of IETs $(\mathcal{R}^n(T))_{n \geq 0}$. For every $n \geq 1$,
\[A^{(n)}(T) = A(T) \cdot A(\mathcal{R}(T)) \cdot \dotsc \cdot A(\mathcal{R}^{n-1}(T))\in  SL_{\mathcal A}(\Z).\]
This defines a multiplicative cocycle $A$ over the transformation $\mathcal{R}$ and taking values in $SL_{\mathcal A}(\Z)$, called the \emph{Kontsevich-Zorich cocycle}.
Let $(n_k)_{k\geq0}$ be an increasing sequence of integers with $n_0=0$ called an \emph{accelerating sequence}. For every $k \geq 0$, let $T^{(k)}:= \mathcal{R}^{n_k}(T) : I^{(k)} \to I^{(k)}$ be an
IET associated with the accelerating sequence. Denote by $(\pi^{(k)},\lambda^{(k)})$ the pair defining $T^{(k)}$ and $ \lambda^{(k)} = (\lambda_\alpha^{(k)})_{\alpha \in \mathcal A} = (|I_\alpha^{(k)}|)_{\alpha \in \mathcal A}$ is the
vector which determines $T^{(k)}$. Then
$T^{(k)}:I^{(k)}\to I^{(k)}$ is the first return map of
$T:I\to I$ to the interval $I^{(k)}\subset I$.

For every $k\geq 0$ let $Z(k+1):=A^{(n_{k+1}-n_{k})}(\mathcal{R}^{n_k}(T))^t$. We then have
\[\lambda^{(k)} = Z(k+1)^t\lambda^{(k+1)}, \quad k \geq 0.\]
By following notations from \cite{Ma-Mo-Yo}, for each $0 \leq k < l$ let
\[
Q(k,l) = Z(l)\cdot Z(l-1) \cdot \dotsc \cdot Z(k+2) \cdot Z(k+1) = A^{(n_{l}-n_{k})}(\mathcal{R}^{n_k}(T))^t.
\]
Then, $Q(k,l) \in SL_{\mathcal A}(\Z)$ and $ \lambda^{(k)} = Q(k,l)^t\lambda^{(l)}$. We write $Q(k) = Q(0,k)$.
In what follows, the norm of a vector is defined as the sum of
the absolute value of coefficients and for any matrix
$B=[B_{\alpha\beta}]_{\alpha,\beta\in\mathcal{A}}$ we set
$\|B\|=\max_{\alpha\in\mathcal{A}}\sum_{\beta\in\mathcal{A}}|B_{\alpha\beta}|$.
It follows that
\begin{equation}\label{eq:Ikl}
|I^{(k)}|  \leq |I^{(l)}|\norm{Q(k,l)}.
\end{equation}

\subsection{Rokhlin towers related to accelerations}
Recall that $Q_{\alpha\beta}(k)$ is the time spent by
any point of $I^{(k)}_{\alpha}$ in $I_{\beta}$ until it
returns to $I^{(k)}$.  It follows that
\[Q_{\alpha}(k)=\sum_{\beta\in\mathcal{A}}Q_{\alpha\beta}(k)\]
is the first return time of points of $I^{(k)}_{\alpha}$ to
$I^{(k)}$. Therefore, the IET $T:I\to I$ splits into a set of $d$ \emph{Rokhlin tower} of the form
\[\big\{ T^i (I^{(k)}_{\alpha}), \ 0\leq i < Q_{\alpha}(k)\big\},\quad \alpha\in \mathcal{A}.\]
Then the $Q_{\alpha}(k)$ floors of the tower  are disjoint intervals.

\section{Diophantine conditions on IETs}\label{sec;DC}
In this section we define a Diophantine Condition on IETs that will be crucial in the proof of the main Theorem~\ref{theorem;main}.
This condition is inspired by  Diophantine Conditions introduced in \cite{Fr-Ul2} and \cite{Gh-Ul}, and fits into the scheme presented by Ulcigrai in \cite{Ul:ICM}.
Since the present Diophantine Condition relates more specifically to the filtration occurring in Oseledets theorem, it is called the \emph{Filtration Diophantine Condition} (\ref{FDC}).
We also prove that a.e.\ IET satisfies \ref{FDC}, see Theorem \ref{thm;FDCRTC}. As the proof is quite standard, it is postponed to Appendix~\ref{sec:App1}.

\subsection{Oseledets filtration}
For each $k\geq 0$, let $\Gamma^{(k)} \subset L^1(I^{(k)})$ be the subspace of piecewise constant functions on $I_\alpha^{(k)} \subset I^{(k)}$ for each $\alpha \in \mathcal A$. Then, we identify every function $\sum_{\alpha \in \mathcal A} h_\alpha \chi_{I_\alpha^{(k)}} \in \Gamma^{(k)}$ with $h  = (h_\alpha)_{\alpha \in \mathcal A} \in \R^{\mathcal{A}}$. We also write $\Gamma = \Gamma^{(0)}$.

We deal with accelerations of the Kontsevich-Zorich cocycle for which the Oseledets ergodic theorem. By the symplecticity of the Kontsevich-Zorich cocycle (see \cite{Zor}) and the simplicity of its Lyapunov exponents
(see \cite{Fo2} and \cite{Av-Vi}),
there exist $\lambda_1>\ldots > \lambda_g>\lambda_{g+1}=0$ such that for a.e.\ IET $(\pi,\lambda)$ there exists a filtration of linear subspaces (Oseledets filtration)
\begin{gather}\label{eq:flagsp}
\begin{split}
\{0\}=E_{0}(\pi,\lambda)\subset E_{-1}(\pi,\lambda)\subset\ldots\subset E_{-g}(\pi,\lambda)\subset E_{cs}(\pi,\lambda)\\
=E_{g+1}(\pi,\lambda)\subset E_{g}(\pi,\lambda)\subset\ldots\subset E_{1}(\pi,\lambda)=\Gamma
\end{split}
\end{gather}
such that for every  $1\leq i\leq g$ we have
\begin{align}\label{eq:Oscond}
\begin{split}
&\lim_{n\to+\infty}\frac{\log\|Q(n)h\|}{n}=-\lambda_i\text{ for all } h\in E_{-i}(\pi,\lambda)\setminus E_{-i+1}(\pi,\lambda)\\
&\lim_{n\to+\infty}\frac{\log\|Q(n)h\|}{n}=0\text{ for all } h\in E_{cs}(\pi,\lambda)\setminus E_{-g}(\pi,\lambda)\\
&\lim_{n\to+\infty}\frac{\log\|Q(n)h\|}{n}=\lambda_i\text{ for all } h\in E_{i}(\pi,\lambda)\setminus E_{i+1}(\pi,\lambda)\\
&\dim E_{-i}(\pi,\lambda)-\dim E_{-i+1}(\pi,\lambda)=\dim E_{i}(\pi,\lambda)-\dim E_{i+1}(\pi,\lambda)=1.
\end{split}
\end{align}

\begin{remark}\label{rem:osel}
Let us consider a filtration of linear subspaces which is partially complementary to the Oseledets filtration \eqref{eq:flagsp}:
\begin{equation}\label{eq:osel}
\{0\}=U_1\subset U_2\subset\ldots \subset U_{g+1}\subset H(\pi)\text{ such that } E_j(\pi,\lambda)\oplus U_j = \Gamma\text{ for } 1\leq j\leq g+1.
\end{equation}
As $U_{j+1}=U_{j}\oplus (U_{j+1}\cap E_{j})$ and $\dim (U_{j+1}\cap E_{j})=1$, for every $1\leq j\leq g$ there exists $h_j\in E_j\setminus E_{j+1}$ such that
\[h_j\in H(\pi),\quad U_{j+1}=U_{j}\oplus\R h_j\text{ and }\lim_{n\to+\infty}\frac{\log\|Q(n)h_j\|}{n}=\lambda_j.\]
Then for every $2\leq j\leq g+1$ the linear subspace $U_j\subset \Gamma$ is  generated by $h_1,\ldots, h_{j-1}$
and
\begin{equation}\label{neq:posexph}
\text{if $0\neq h \in U_j$ then }\lim_{n\to+\infty}\frac{\log\|Q(n)h\|}{n}\geq \lambda_{j-1}\geq \lambda_g>0.
\end{equation}
\end{remark}

For every $k\geq 0$ and $1\leq j\leq g+1$, let $E_{j}^{(k)}:=Q(k)E_{j}$, $U_{j}^{(k)}:=Q(k)U_{j}$.
Then for all $0\leq k\leq l$ we have
\[E_{j}^{(l)}=Q(k,l)E_{j}^{(k)}, \quad U_{j}^{(l)}=Q(k,l)U_{j}^{(k)}.\]
For  any choice of the complementary filtration \eqref{eq:osel}, all $0\leq k\leq l$ and every $1\leq j\leq g+1$, we consider
the restrictions of the operator $Q(k,l):\Gamma^{(k)}\to \Gamma^{(l)}$ given by
\begin{gather*}
Q|_{E_j}(k,l):E_{j}^{(k)}\to E_{j}^{(l)},\quad
Q|_{U_j}(k,l):U_{j}^{(k)}\to U_{j}^{(l)}
\end{gather*}
and the corresponding projections
\begin{gather*}
P^{(k)}_{E_j}: \Gamma^{(k)} \to E_{j}^{(k)},\quad P^{(k)}_{U_j}: \Gamma^{(k)} \to U_{j}^{(k)},\text{ i.e. }P^{(k)}_{E_j}\oplus P^{(k)}_{U_j}=\operatorname{Id}.
\end{gather*}


\subsection{Rokhlin Tower Condition and Filtration Diophantine Condition}\label{sec;FDC-RTC}
First we remind the following Rokhlin Towers Condition (RTC) introduced in \cite{Fr-Ul2}.
\begin{definition}[RTC]\label{def:RTC}
An IET $T_{(\pi,\lambda)}$ together with  an acceleration   satisfy  RTC if there exists a constant $0<\delta<1$ such that
\begin{align}
 \tag{RT} \label{def:FDC-g}
\begin{split}
& \text{for any $k\geq 1$ there exists } \text{number $0<p_{k}\leq \min_{\alpha\in\mathcal{A}}Q_\alpha(k)$} \text{ such that}   \\
&  \{T^iI^{(k)}:0\leq i<p_k\} \text{ is a Rokhlin  of intervals with measure $\geq \delta|I|$.}
\end{split}
\end{align}
\end{definition}
For any  sequence $(r_n)_{n\geq 0}$ of real numbers and for all $0\leq k\leq l$, we will use the notation $r(k,l):=\sum_{k\leq j<l}r_j$.

\begin{definition}[\customlabel{FDC}{FDC}]\label{def;FDC}
An IET $T : I \to I$ satisfying Keane's condition and Oseledets generic (i.e.\ there is a filtration of linear subspaces \eqref{eq:flagsp} satisfying \eqref{eq:Oscond}), satisfies
the \emph{Filtration Diophantine Condition (FDC)} if for every $\tau > 0$ there exist constants $c>0$, $C,\kappa\geq 1$, an accelerating sequence $(n_k)_{k\geq0}$, a  sequence of natural numbers $(r_n)_{n\geq 0}$ with $r_0 = 0$ and
a complementary filtration $(U_j)_{1 \leq j \leq g+1}$ (satisfying \eqref{eq:osel}) such that \eqref{def:FDC-g} holds and
\begin{align}
\lim_{n\to+\infty}\frac{r(0,n)}{n}&\in(1,1+\tau)\label{def;sdc0}\\
\norm{Q|_{E_j}(k,l)} \leq Ce^{(\lambda_j+\tau)r(k,l)} &\text{ for all }0\leq k<l\text{ and }1\leq j\leq g+1 \label{def;sdc1}\\
\norm{Q|_{U_j}(k,l)^{-1}} \leq Ce^{(-\lambda_{j-1}+\tau)r(k,l)} &\text{ for all }0\leq k<l\text{ and }2\leq j\leq g+1 \label{def;sdc12}\\
\norm{Z(k+1)} \leq Ce^{\tau k} &\text{ for all } k\geq 0\label{def;sdc2}\\
ce^{\lambda_1k}\leq \norm{Q(k)} \leq Ce^{\lambda_1(1+\tau)k} &\text{ for all }k\geq 0 \label{def;sdc3}\\
\max_{\alpha\in\mathcal{A}}\frac{|I^{(k)}|}{|I^{(k)}_\alpha|}\leq \kappa &\text{ for all }k\geq 0 \label{def;sdc4}\\
\big|\sin \angle \big(E_j^{(k)}, U_j^{(k)}\big) \big| \geq c  \norm{Q(k)}^{-\tau} &\text{ for all }k\geq 0 \text{ and }2\leq j\leq g+1. \label{def;sdc5}
\end{align}
\end{definition}

\begin{theorem}\label{thm;FDCRTC}
 Almost every IET satisfies \ref{FDC}.
\end{theorem}
\noindent As we have mentioned before, the proof is rather standard and we postpone it to Appendix~\ref{sec:App1}.

\begin{remark}\label{rmk:acc}
By the proof of Theorem~\ref{thm;FDCRTC}, for every $\tau>0$ its corresponding accelerating sequence
$(n_k)_{k\geq0}$, which we call an FDC-acceleration, is a sequence of return times for the invertible Rauzy-Veech
renormalization to a subset such that its measure converges to $1$ as $\tau\to 0$. It follows that for any pair
of two distinct and small $\tau$ and $\tau'$ their corresponding accelerating sequences have most of the elements in common.
\end{remark}
%


\begin{remark}
By \eqref{def:FDC-g} and \eqref{def;sdc4},  for every $\alpha\in\mathcal A$ and $k\geq 1$ we have
\begin{equation}\label{eq:qlambda}
|I|\geq Q_\alpha(k)\lambda^{(k)}_\alpha\geq \min_{\beta\in\mathcal{A}}Q_\beta(k)\lambda^{(k)}_\alpha\geq \frac{1}{\kappa}p_k|I^{(k)}|\geq \frac{\delta}{\kappa}|I|.
\end{equation}
It follows that for all $\alpha,\beta\in\mathcal A$ we have
$\frac{\delta}{\kappa}\leq \frac{Q_\alpha(k)}{Q_\beta(k)}\leq\frac{\kappa}{\delta}$. Hence
\begin{equation}\label{eq:minmax}
\min_{\alpha\in\mathcal A}Q_\alpha(k)\leq\|Q(k)\|=\max_{\alpha\in\mathcal A}Q_\alpha(k)\leq \frac{\kappa}{\delta}\min_{\alpha\in\mathcal A}Q_\alpha(k),
\end{equation}
so
\begin{equation}\label{eq:Qalphaexp}
\lim_{k\to\infty}\frac{\log Q_\alpha(k)}{k}=\lambda_1\text{ for every }\alpha\in \mathcal A.
\end{equation}
\end{remark}


\begin{lemma}
For any $\tau>0$, $k\geq 0$ and $2\leq j\leq g+1$, the following holds:
\begin{align}\label{eqn;projbound}
\|P^{(k)}_{E_j}\|\leq C \norm{Q(k)}^\tau,\quad \|P^{(k)}_{U_j}\|\leq C \norm{Q(k)}^\tau.
\end{align}
\end{lemma}
\begin{proof}
Let us consider any $v\in\Gamma^{(k)}$ and set $e:=P^{(k)}_{E_j}v\in E_j^{(k)}$ and $u:=P^{(k)}_{U_j}v\in U_j^{(k)}$.
Then
\begin{align*}
\norm{v}^2 &\geq \norm{e}^2 + \norm{u}^2 - 2 |\cos \angle
(e, u)| \norm{e}\norm{u}
 \geq (\norm{e}^2 + \norm{u}^2 ) \left(1- | \cos \angle
(e, u)| \right) \\
& \geq \max\{\norm{u}^2,\norm{e}^2\} \frac{1}{2}\left(1-  \cos^2 \angle
(e,u) \right)=\max\{\norm{u}^2,\norm{e}^2\} \frac{1}{2} \sin^2 \angle
(e,u).
\end{align*}
It follows that
\[
\|P^{(k)}_{E_j}\| \leq \sqrt{2} \big|\sin \angle
\big(E_j^{(k)}, U_j^{(k)}\big)\big|^{-1}\quad\text{and}\quad\|P^{(k)}_{U_j}\| \leq \sqrt{2} \big|\sin \angle
\big(E_j^{(k)}, U_j^{(k)}\big)\big|^{-1}.\]
In view of \eqref{def;sdc5}, we obtain required  bounds for $\|P^{(k)}_{E_j}\|$ and $\|P^{(k)}_{U_j}\|$.
\end{proof}

\subsection{Diophantine series}\label{sec;DCS}
In the proof of our main results, certain sums and series (defined in Definition \ref{def;dio-series}) relying 
 on the matrices of the (accelerated) cocycle play a central role to control Birkhoff sums of $\varphi_f$. 
 We here show that these quantities, under \ref{FDC}, are well-defined and grow in a controlled way (see Proposition \ref{prop;FDCbound}).

\medskip

For every $a\geq 0$ and $s\geq 1$, let $\langle s\rangle^a= s^a$ if $a>0$ and $\langle s\rangle^a= 1+\log s$ if $a=0$.
\begin{definition}\label{def;dio-series}
For every IET $T : I \rightarrow I$ satisfying Keane's condition, any $0\leq a<1$, any $1\leq i\leq g$, any $\tau>0$ and any accelerating sequence we define sequences
$(K^{a,i,\tau}_{l})_{l\geq 0}, (C^{a,i,\tau}_{k})_{k\geq 0}$ so that
\begin{align*}
&K^{a,i,\tau}_{l}(T):=
\sum_{m\geq l}\norm{Q|_{U_i}(l,m+1)^{-1}}\|Z(m+1)\|\langle\|Q(m)\|\rangle^{a}\|Q(m+1)\|^\tau \text{ for } l \geq 0,
\\
&C^{a,i,\tau}_{k}(T):= \sum_{1 \leq l \leq k}\norm{Q|_{E_i}(l,k)}\|Z(l)\|\langle \|Q(l-1)\|\rangle^{a}\| \|Q(l)\|^\tau.
\end{align*}

\end{definition}

\begin{proposition}\label{prop;FDCbound}
Let $T : I \rightarrow I$ be an IET satisfying FDC and let $0\leq a<1$. Suppose that $2\leq i\leq g+1$ is chosen such that $\lambda_i\leq a\lambda_1<\lambda_{i-1}$.
Then for every $0 < \tau < \frac{\lambda_{i-1}-\lambda_{1}a}{3(1+\lambda_1)}$ the sequences $(K^{a,i,\tau}_l)_{l\geq 0}, (C^{a,i,\tau}_k)_{k\geq 0}$ are well defined and
\begin{align}
K^{a,i,\tau}_{l}(T) = O\big(e^{(\lambda_{1}a+5(1+\lambda_1)\tau)l}\big),\quad \label{prop;sdc-1}
C^{a,i,\tau}_{k}(T) = O\big(e^{(\lambda_{1}a + 5(1+\lambda_1)\tau)k}\big). 
\end{align}
\end{proposition}
\begin{proof}
By \eqref{def;sdc1} applied to $j=1$, there exists $C'\geq C$ such that for all $m \geq l+1$, we have
\begin{align}
\langle\|Q(l,m)\|\rangle^0 &=1+\log \|Q(l,m)\| \leq 1+\log C + (\lambda_1+\tau)r(l,m) \leq C'e^{\tau r(l,m)}\\
\langle\|Q(l,m)\|\rangle^a &\leq C'e^{(\lambda_1a+\tau) r(l,m)}\text{ for all }0\leq a<1. \label{eq:Qlma}
\end{align}



By  \eqref{def;sdc12}, \eqref{def;sdc2} and \eqref{def;sdc3}, it follows that for every $ l \geq 0$ we have
\begin{align*}
K^{a,i,\tau}_{l}(T) & \leq \sum_{ m \geq l} Ce^{(-\lambda_{i-1}+\tau)r(l,m+1)}C'e^{(\lambda_{1}a+\tau )r(0,m)} Ce^{\tau(m+1)}C^\tau e^{\lambda_1(1+\tau)(m+1)\tau}\\
& \leq  {(C')}^4\sum_{ m \geq l} Ce^{(-\lambda_{i-1}+\tau)r(l,m+1)}e^{(\lambda_{1}a+2\tau(1+\lambda_1) )r(0,m+1)} \\
& \leq  {(C')}^4e^{(\lambda_{1}a+2\tau(1+\lambda_1))r(0,l)}\sum_{ m \geq l} e^{(-\lambda_{i-1}+\lambda_{1}a+3\tau(1+\lambda_1))r(l,m+1)}\\
& \leq  {(C')}^4e^{(\lambda_{1}a+2\tau(1+\lambda_1))r(0,l)}\sum_{ m \geq l} e^{(-\lambda_{i-1}+\lambda_{1}a+3\tau(1+\lambda_1))(m+1-l)}\\
& \leq  {(C')}^4e^{(\lambda_{1}a+2\tau(1+\lambda_1))r(0,l)}\sum_{ j = l}^\infty e^{(-\lambda_{i-1}+\lambda_{1}a+3\tau(1+\lambda_1))j}.
\end{align*}
In view of \eqref{def;sdc0}, it follows that
\begin{align*}
\limsup_{l\to\infty}\frac{\log K^{a,i,\tau}_{l}(T)}{k} & \leq\lim_{l\to\infty}\Big( \frac{(\lambda_{1}a+2(1 +\lambda_1)\tau)r(0,l)}{l}\Big)\\
& < (\lambda_{1}a+2(1+\lambda_1)\tau)(1+\tau) < \lambda_{1}a + 5(1+\lambda_1)\tau,
\end{align*}
which gives the left part of \eqref{prop;sdc-1}.
\medskip

For the second bound, we apply \eqref{def;sdc1}, \eqref{def;sdc2}, \eqref{def;sdc3} and \eqref{eq:Qlma},
\begin{align*}
C^{a,i,\tau}_{k}(T) & \leq \sum_{0 \leq l \leq k} Ce^{(\lambda_{i}+\tau)r(l,k)}C'e^{(\lambda_1 a+\tau )r(0,l-1)}Ce^{\tau l} C^\tau e^{\lambda_1(1+\tau)l\tau} \\
& \leq {(C')}^4\sum_{0 \leq l \leq k} e^{(\lambda_{i}+\tau)r(l,k)}e^{(\lambda_1 a+2\tau(1+\lambda_1))r(0,l)} \\
& \leq {(C')}^4 e^{(\lambda_1 a+ 2\tau(1+\lambda_1))r(0,k)}\sum_{0 \leq l \leq k} e^{-(\lambda_i-\lambda_1 a+\tau)r(l,k)} \\
& \leq {(C')}^4e^{(\lambda_1 a+ 2\tau(1+\lambda_1))r(0,k)}\sum_{0 \leq l \leq k} e^{-(\lambda_i-\lambda_1 a+\tau)(k-l)} \\
& \leq {(C')}^4 e^{(\lambda_1 a+ 2\tau(1+\lambda_1))r(0,k)}\sum_{ l=0}^\infty e^{-(\lambda_i-\lambda_1 a+\tau)l}.
\end{align*}
In view of \eqref{def;sdc0}, it follows that
\begin{align*}
\limsup_{k\to\infty}\frac{\log C^{a,i,\tau}_{k}(T)}{k}& \leq\lim_{k\to\infty}\Big( \frac{(\lambda_{1}a+2(1 +\lambda_1)\tau)r(0,k)}{k}\Big)\\
& < (\lambda_{1}a+2(1+\lambda_1)\tau)(1+\tau) < \lambda_{1}a + 5(1+\lambda_1)\tau,
\end{align*}
which gives the right part of \eqref{prop;sdc-1}.

\end{proof}

\section{Functions with polynomial singularities}\label{sec;polycocycle}

In this section we first introduce  a one-parameter family of spaces of cocycles with polynomial singularities over a given IET.
Then we adopt the norms which make the spaces of cocycles a Banach space. These new Banach spaces are inspired by the notion of cocycles with logarithmic singularities and the corresponding space studied in \cite{Fr-Ul, Fr-Ul2}. We also prove several properties of them which will play a key role in the proofs of the main results.

\subsection{Spaces $\pag$ and $\wpa$}\label{sec;cocycle}
Fix $0\leq a<1$ and an IET $T=T_{\pi,\lambda}$. For every $\alpha\in\mathcal{A}$, denote by $m_\alpha$ the middle point of the interval $I_\alpha$, i.e.\ $m_\alpha = (l_\alpha + r_\alpha)/2$. Denote by $C^1(\sqcup_{\alpha \in \mathcal{A}}I_\alpha )$ the space of $C^1$-function $\varphi : \bigcup_{\alpha\in\mathcal{A}}(l_\alpha,r_\alpha) \rightarrow \R$. For every $\varphi\in C^1(\sqcup_{\alpha \in \mathcal{A}}I_\alpha )$ let us consider
\begin{align*}
p_a(\varphi):=& \sup\Big\{\min_{\bar x \in End(T)}|\varphi'(x)(x-\bar x)^{1+a}| : x \in I \backslash End(T)\Big\}\\
=& \max_{\alpha\in\mathcal{A}}\Big\{\sup_{x\in(l_\alpha,m_\alpha]}|\varphi'(x)(x-l_\alpha)^{1+a}|,\sup_{x\in[m_\alpha,r_\alpha)}|\varphi'(x)(r_\alpha-x)^{1+a}|\Big\}.
\end{align*}

\begin{definition}\label{def;pa}
We denote by $\pa(\sqcup_{\alpha \in \mathcal{A}}I_\alpha )$ the space of functions $\varphi\in C^1(\sqcup_{\alpha \in \mathcal{A}}I_\alpha )$ such that
$p_a(\varphi)<+\infty$ and for every $\alpha\in\mathcal{A}$ the limits
\[C_\alpha^+=C_\alpha^+(\varphi):=-\lim_{x\searrow l_\alpha}\varphi'(x)(x-l_\alpha)^{1+a}\text{ and }C_\alpha^-=C_\alpha^-(\varphi):=\lim_{x\nearrow r_\alpha}\varphi'(x)(r_\alpha-x)^{1+a}\]
exist.
Then we say that $\varphi\in \pa(\sqcup_{\alpha \in \mathcal{A}}I_\alpha )$ has \emph{polynomial singularities} of degree at most $a$.
If $\varphi\in \pa(\sqcup_{\alpha \in \mathcal{A}}I_\alpha )$ then
\begin{equation}\label{neq:Cpa}
|C_\alpha^+(\varphi)|\leq p_a(\varphi) \text{ and }|C_\alpha^-(\varphi)|\leq p_a(\varphi) \text{ for all }\alpha\in\mathcal{A}.
\end{equation}
We denote by $\pag(\sqcup_{\alpha \in \mathcal{A}}I_\alpha )\subset \pa(\sqcup_{\alpha \in \mathcal{A}}I_\alpha )$ the space of functions with polynomial singularities of \emph{geometric type}, i.e.\ such that  \[C_{\pi_0^{-1}(d)}^{-} \cdot C_{\pi_1^{-1}(d)}^{-}=0\quad\text{and}\quad C_{\pi_0^{-1}(1)}^{+} \cdot C_{\pi_1^{-1}(1)}^{+}=0.\]
\end{definition}

For every $0\leq a<1$, let us consider the space   $\wpa(\sqcup_{\alpha \in \mathcal{A}}I_\alpha)$ of Borel functions $\varphi:I\to\R$ such that
\begin{align*}
\max_{\alpha\in\mathcal{A}}\Big\{\sup_{x\in(l_\alpha,m_\alpha]}|\varphi(x)(x-l_\alpha)^{a}|,\sup_{x\in[m_\alpha,r_\alpha)}|\varphi(x)(r_\alpha-x)^{a}|
\Big\}<+\infty&\text{ if $0<a<1$, or}\\
\max_{\alpha\in\mathcal{A}}\Big\{\sup_{x\in(l_\alpha,m_\alpha]}\frac{|\varphi(x)|}{|\log(x-l_\alpha)|},
\sup_{x\in[m_\alpha,r_\alpha)}\frac{|\varphi(x)|}{|\log(r_\alpha-x)|}
\Big\}<+\infty&\text{ if }a=0.
\end{align*}
By Lemma~\ref{lemlogosc0}, $\pa(\sqcup_{\alpha \in \mathcal{A}}I_\alpha )\subset\wpa(\sqcup_{\alpha \in \mathcal{A}}I_\alpha )\subset L^1(I)$.

\medskip

The following result explains how both types of the spaces appear when considering the deviation of Birkhoff integrals for locally Hamiltonian flows.
The proof of theorem (and extended version (Theorem~\ref{thm:phifform})) is postponed to \S\ref{sec:regularity}.

\begin{theorem}\label{thm:ftophi}
Let $\psi_\R$ be a locally Hamiltonian flow, $M'\subset M$ its minimal component and $I\subset M'$ a transversal curve. Let
$m:=\max \{m_\sigma:\sigma\in \mathrm{Fix}(\psi_\R)\cap M'\}$.
\begin{itemize}
\item[(i)]
For every $f\in C^m(M)$ we have
\[\varphi_f \in \pag(\sqcup_{\alpha\in \mathcal{A}}
I_{\alpha})\text{ and  }\varphi_{|f|} \in \operatorname{\widehat{P}_{a}}(\sqcup_{\alpha\in \mathcal{A}}
I_{\alpha})\text{ with }a=\tfrac{m-2}{m}.\]
\item[(ii)]
Assume that $\sigma\in  \mathrm{Fix}(\psi_\R)\cap M'$ and $f:M\to\R$ is a $C^m$-map such that $f$ vanishes on an open neighbourhood of $\{\sigma'\in \mathrm{Fix}(\psi_\R):\sigma'\neq\sigma\}$.
 For every  $0\leq k\leq m_\sigma-2$ if $f^{(j)}(\sigma)=0$ for $0\leq j<k$ then
 \[\varphi_f \in \operatorname{P_{b(\sigma,k)}G}(\sqcup_{\alpha\in \mathcal{A}}
I_{\alpha})\text{  and }\varphi_{|f|} \in \operatorname{\widehat{P}_{b(\sigma,k)}}(\sqcup_{\alpha\in \mathcal{A}}
I_{\alpha})\text{ with }b(\sigma,k)=\tfrac{m_\sigma-2-k}{m_\sigma}.\]
\end{itemize}
\end{theorem}

For every $0\leq a<1$, let us consider the norm on $\pa(\sqcup_{\alpha \in \mathcal{A}}I_\alpha )$ given by
\[\|\varphi\|_a:= \|\varphi\|_{L^1(I)}+p_a(\varphi).\]

\begin{lemma}
For every $0\leq a<1$ the space $\pa(\sqcup_{\alpha \in \mathcal{A}}I_\alpha )$ equipped with the norm $\|\,\cdot\,\|_a$ is Banach. Moreover, $\pag(\sqcup_{\alpha \in \mathcal{A}}I_\alpha )$ is its closed subspace.
\end{lemma}

\begin{proof}
Suppose that $0<a<1$. In the case where $a=0$ the proof proceeds in the same way.
For every $\vep>0$ let $C^0_\vep(\sqcup_{\alpha \in \mathcal{A}}I_\alpha )$(or $C^1_\vep(\sqcup_{\alpha \in \mathcal{A}}I_\alpha )$) be the space of
$C^0$/$C^1$-maps on $\bigcup_{\alpha\in\mathcal{A}}[l_\alpha+\vep,r_\alpha-\vep]$ equipped with the standard norm denoted by
$\|\,\cdot\,\|_{C^0_\vep}$ (or $\|\,\cdot\,\|_{C^1_\vep}$). In view of \eqref{eqn;upperboundvarphi}, there exists $C\geq 1$ such that for every $\vep>0$
and for all $\varphi\in \pa(\sqcup_{\alpha \in \mathcal{A}}I_\alpha )$ we have
\begin{equation}\label{neq:c0}
\|\varphi\|_{C^0_\vep}\leq C\vep^{-a}\|\varphi\|_a\text{ and }\|\varphi'\|_{C^0_\vep}\leq \vep^{-a-1}\|\varphi\|_a.
\end{equation}
Hence
\begin{equation}\label{neq:c1l1}
\|\varphi\|_{C^1_\vep}\leq C\vep^{-a-1}\|\varphi\|_a\text{ and }\|\varphi\|_{L^1(I)}\leq \|\varphi\|_a.
\end{equation}
On the other hand, for every $\varphi\in C^1(\sqcup_{\alpha \in \mathcal{A}}I_\alpha)$ we have
\begin{equation}\label{neq:pac0}
p_a(\varphi)\leq\sup_{\vep>0}\vep^{1+a}\|\varphi'\|_{C^0_\vep}.
\end{equation}

Let $(\varphi_n)_{n\geq 1}$ be a Cauchy sequence in $\pa(\sqcup_{\alpha \in \mathcal{A}}I_\alpha )$. Then $\sup_{n\geq 1}\|\varphi_n\|_a<+\infty$.
Since $L^1(I)$ and $C^1_\vep(\sqcup_{\alpha \in \mathcal{A}}I_\alpha )$ for all $\vep>0$ are Banach spaces, in view of \eqref{neq:c1l1}, the sequence $(\varphi_n)_{n\geq 1}$  converges in
$L^1(I)$ and in $C^1_\vep(\sqcup_{\alpha \in \mathcal{A}}I_\alpha )$ for all $\vep>0$. Then its limit $\varphi$ belong to
$L^1(I)$ and to $C^1(\sqcup_{\alpha \in \mathcal{A}}I_\alpha )$.

By assumption, for every $\delta>0$ there exists $N\in\N$ such that $\|\varphi_n-\varphi_m\|_a<\delta/C$ if $m,n\geq N$.
In view of \eqref{neq:c0} and \eqref{neq:pac0}, for every $\vep>0$ and $n\geq N$ we have
\begin{align*}
\|\varphi'_n-\varphi'\|_{C^0_\vep}&= \lim_{m\to\infty}\|\varphi_n'-\varphi_m'\|_{C^0_\vep}\leq C\vep^{-a-1}\limsup_{m\to\infty}\|\varphi_n-\varphi_m\|_a\\
&\leq C \vep^{-a-1}\sup_{m\geq N}\|\varphi_n-\varphi_m\|_a\leq \vep^{-a-1}\delta.
\end{align*}
In view of \eqref{neq:pac0}, this gives
\begin{equation}\label{neq:padelta}
p_a(\varphi_n-\varphi)\leq\sup_{\vep>0}\vep^{1+a}\|\varphi'_n-\varphi'\|_{C^0_\vep}\leq \delta\text{ for all }n\geq N.
\end{equation}
It follows that $p_a(\varphi)<\infty$ and $\|\varphi_n-\varphi\|_a\to 0$ as $n\to\infty$. By \eqref{neq:Cpa},
\[\sup\{|C^+_\alpha(\varphi_n)|, |C^-_\alpha(\varphi_n)|:n\geq 1,\alpha\in\mathcal{A}\}\leq \sup_{n\geq 1}p_a(\varphi_n)\leq \sup_{n\geq 1}
\|\varphi_n\|_a<\infty.\]
Therefore, there exists a subsequence $(\varphi_{k_n})_{n\geq 1}$ such that
\begin{equation}\label{eq:C+-}
C^+_\alpha:=\lim_{n\to\infty}C^+_\alpha(\varphi_{k_n})\text{ and }C^-_\alpha:=\lim_{n\to\infty}C^-_\alpha(\varphi_{k_n})
\end{equation}
exist for all $\alpha\in\mathcal{A}$. Choose $m\geq 1$ such that $k_m\geq N$ and
\[|C^+_\alpha-C^+_\alpha(\varphi_{k_m})|<\delta\text{ and }|C^-_\alpha-C^-_\alpha(\varphi_{k_m})|<\delta\text{ for all }\alpha\in\mathcal{A}.\]
As $\varphi_{k_m}\in \pa(\sqcup_{\alpha \in \mathcal{A}}I_\alpha )$, there exists $\delta'>0$ such that
\begin{gather*}
|C^+_\alpha(\varphi_{k_m})+\varphi_{k_m}'(x)(x-l_\alpha)^{1+a}|<\delta \text{ if }x\in (l_\alpha,l_\alpha+\delta')\\
|C^-_\alpha(\varphi_{k_m})-\varphi_{k_m}'(x)(r_\alpha-x)^{1+a}|<\delta \text{ if }x\in (r_\alpha-\delta',r_\alpha).
\end{gather*}
In view of \eqref{neq:padelta}, it follows that for every $x\in (l_\alpha,l_\alpha+\delta')$ we have
\begin{align*}
|&C^+_\alpha+\varphi'(x)(x-l_\alpha)^{1+a}|\\
&\leq|C^+_\alpha-C^+_\alpha(\varphi_{k_m})|+|C^+_\alpha(\varphi_{k_m})+\varphi_{k_m}'(x)(x-l_\alpha)^{1+a}|\\
&\quad+
|\varphi'(x)(x-l_\alpha)^{1+a}-\varphi_{k_m}'(x)(x-l_\alpha)^{1+a}|\\
&\leq |C^+_\alpha-C^+_\alpha(\varphi_{k_m})|+|C^+_\alpha(\varphi_{k_m})+\varphi_{k_m}'(x)(x-l_\alpha)^{1+a}|+p_a(\varphi-\varphi_{k_m})< 3\delta
\end{align*}
for every $x\in (r_\alpha-\delta',r_\alpha)$ we have
\begin{align*}
|&C^-_\alpha-\varphi'(x)(r_\alpha-x)^{1+a}|\\
&\leq|C^-_\alpha-C^-_\alpha(\varphi_{k_m})|+|C^-_\alpha(\varphi_{k_m})-\varphi_{k_m}'(x)(r_\alpha-x)^{1+a}|\\
&\quad+
|\varphi'(x)(r_\alpha-x)^{1+a}-\varphi_{k_m}'(x)(r_\alpha-x)^{1+a}|\\
&\leq |C^-_\alpha-C^-_\alpha(\varphi_{k_m})|+|C^-_\alpha(\varphi_{k_m})-\varphi_{k_m}'(x)(r_\alpha-x)^{1+a}|+p_a(\varphi-\varphi_{k_m})< 3\delta.
\end{align*}
Therefore
\begin{equation}\label{eq:C+-varphi}
C_\alpha^+=-\lim_{x\searrow l_\alpha}\varphi'(x)(x-l_\alpha)^{1+a}\text{ and }C_\alpha^-=\lim_{x\nearrow r_\alpha}\varphi'(x)(r_\alpha-x)^{1+a},
\end{equation}
so $\varphi\in \pa(\sqcup_{\alpha \in \mathcal{A}}I_\alpha)$.

\medskip

Now additionally suppose that $\varphi_n\in \pag(\sqcup_{\alpha \in \mathcal{A}}I_\alpha)$ for every $n\geq 1$.
Hence
\[ C_{\pi_0^{-1}(d)}^{-}(\varphi_{k_n})\cdot C_{\pi_1^{-1}(d)}^{-}(\varphi_{k_n})=0\text{ and }
C_{\pi_0^{-1}(1)}^{+}(\varphi_{k_n})\cdot C_{\pi_1^{-1}(1)}^{+}(\varphi_{k_n})=0\]
for every $n\geq 1$. In view of \eqref{eq:C+-} and \eqref{eq:C+-varphi}, this gives
\[ C_{\pi_0^{-1}(d)}^{-}(\varphi)\cdot C_{\pi_1^{-1}(d)}^{-}(\varphi)=0\text{ and }
C_{\pi_0^{-1}(1)}^{+}(\varphi)\cdot C_{\pi_1^{-1}(1)}^{+}(\varphi)=0,\]
so $\varphi\in \pag(\sqcup_{\alpha \in \mathcal{A}}I_\alpha)$. It follows that $\pag(\sqcup_{\alpha \in \mathcal{A}}I_\alpha)$
is a Banach space as well.
\end{proof}

\subsection{Basic properties of functions with polynomial singularities.}
In this subsection we present basic properties of $\pag$-functions.
Most of them are general versions of inequalities from \cite{Fr-Ul, Fr-Ul2}.
\medskip

For every integrable function $f:I\to\R$ and a subinterval $J\subset I$, let $m(f,J)$ stand for the mean value of $f$ on $J$,
that is
\begin{equation}\label{def;m(f,J)}
m(f,J)=\frac{1}{|J|}\int_Jf(x)\,dx.
\end{equation}
For every IET $T$ we define the corresponding mean value \emph{projection operator} $\mathcal{M}:L^1(I)\to \Gamma$ by
\begin{equation}\label{eqn;projec}
\mathcal{M}(f)=\sum_{\alpha\in\mathcal{A}}m(f,I_\alpha)\chi_{I_\alpha}.
\end{equation}
This operator projects a function onto a piece-wise constant function, whose value is equal to the mean value of $f$ on the exchanged intervals $I_\alpha$, $\alpha\in\mathcal{A}$. 

\begin{lemma}\label{lemlogosc0}
Suppose that $0\leq a<1$.
Let us consider any $C^1$-map $f:J \rightarrow \R$ with $J:=(x_0,x_1]$ such that $|f'(x)(x-x_0)^{1+a}|\leq C$ for $x \in (x_0,x_1]$.
Then for every $s\in (x_0,x_1]$, we have
\begin{align}
\label{neq:fsa}
&|f(s)-m(f,J)| \leq C\left(\frac{1}{a(s-x_0)^{a}}+\frac{2a-1}{a(1-a)}\frac{1}{|J|^{a}}\right), \quad \text{ if } 0<a <1\\
\label{neq:fs0}
&|f(s)-m(f,J)| \leq  C\left(\log\frac{|J|}{s-x_0}+1 \right), \quad \text{ if } a = 0.
\end{align}
Moreover, for every $0\leq a<1$ we have
\begin{equation}\label{neq:intfs}
\frac{1}{|J|}\int_J\left|f(s) - m(f,J)\right|\,ds\leq \frac{2C}{(1-a)|J|^a}.
\end{equation}
\end{lemma}
\begin{proof}
For all $t,s \in (x_0,x_1]$ and $a \neq 0$, we have
\begin{align*}
|f(s) - f(t)| &= \left|\int_s^t f'(u)du\right| \leq C \left|\int_s^t \frac{1}{(u-x_0)^{1+a}}du \right | \\
& \leq C\left|-\frac {(t-x_0)^{-a}}{a} + \frac {(s-x_0)^{-a}}{a}\right|.
\end{align*}
It follows that
\begin{align*}
\left|f(s) - m(f,J)\right| & \leq \frac{C}{|J|}\int_{x_0}^{x_1}\left|-\frac {(t-x_0)^{-a}}{a} + \frac {(s-x_0)^{-a}}{a}\right|dt \\
&=C\left(\frac{2}{1-a}\frac{(s-x_0)^{1-a}}{|J|}+\frac{(s-x_0)^{-a}}{a}-\frac{(x_1-x_0)^{-a}}{a(1-a)}\right)\\
&\leq C\left(\frac{1}{a(s-x_0)^{a}}+\frac{2a-1}{a(1-a)}\frac{1}{|J|^{a}}\right),
\end{align*}
which gives \eqref{neq:fsa}.
Moreover, we have
\begin{align*}
\frac{1}{|J|}&\int_J\left|f(s) - m(f,J)\right|\,ds\\
&\leq \frac{C}{|J|}\int_{x_0}^{x_1}\left(\frac{2}{1-a}\frac{(s-x_0)^{1-a}}{|J|}+\frac{(s-x_0)^{-a}}{a}-\frac{(x_1-x_0)^{-a}}{a(1-a)}\right)\,ds\\
&=\frac{2C}{(1-a)(2-a)|J|^a}\leq \frac{2C}{(1-a)|J|^a},
\end{align*}
which gives \eqref{neq:intfs} when $0<a<1$.
\medskip

Similarly, if $a = 0$ then
\begin{align*}
|f(s) - f(t)| &= \left|\int_s^t f'(u)du\right| \leq  \left|\int_s^t \frac{C}{u-x_0}du\right| = C \left|\log \frac{t-x_0}{s-x_0}\right| .
\end{align*}
It follows that
\begin{align*}
\left|f(s) - m(f,J)\right| & \leq \frac{C}{|J|}\int_{x_0}^{x_1}\left|\log \frac{t-x_0}{s-x_0}\right|dt =C\left(\log\frac{|J|}{s-x_0}-1+2\frac{s-x_0}{|J|}\right)\\
& \leq C\left(\log\frac{|J|}{s-x_0}+1 \right),
\end{align*}
which gives \eqref{neq:fs0}.
Moreover, we have
\begin{align*}
\frac{1}{|J|}\int_J\left|f(s) - m(f,J)\right|\,ds&\leq \frac{C}{|J|}\int_{x_0}^{x_1}\left(\log\frac{|J|}{s-x_0}-1+2\frac{s-x_0}{|J|}\right)\,ds=C,
\end{align*}
which gives \eqref{neq:intfs} when $a=0$.
\end{proof}

\begin{remark}
By \eqref{neq:fsa} and \eqref{neq:fs0}, we also have
\begin{equation}\label{neq:fsx1}
\left|f(x_1) - m(f,J)\right|  \leq \frac{C}{(1-a)|J|^a}\text{ for every }0\leq a<1.
\end{equation}
\end{remark}

In particular, specific bounds for averaged function $\varphi$ on $I_\alpha$ are given as follows.
\begin{lemma}
Assume that $0\leq a<1$ and $\varphi \in \pa(\sqcup_{\alpha \in \mathcal{A}}I_\alpha )$.
Then for every $\alpha\in\mathcal{A}$ and $x \in Int I_\alpha$ we have
\begin{equation}\label{eqn;upperboundvarphi}
|\varphi(x)| \leq \norm{\mathcal{M}(\varphi)} +p_a(\varphi)\left(\frac{1}{a\min\{x-l_\alpha, r_\alpha-x\}^{a}}+
\frac{2^{a+2}}{a(1-a)|I_\alpha|^{a}}\right)\\
\end{equation}
if $0<a<1$ and
\begin{equation}\label{eqn;upperboundvarphi2}
|\varphi(x)| \leq \norm{\mathcal{M}(\varphi)} + p_a(\varphi)\left(\log\frac{|I_\alpha|}{2\min\{x-l_\alpha, r_\alpha-x\}}+2 \right)
\end{equation}
if $a=0$.
\end{lemma}

\begin{proof}
By \eqref{neq:fsx1} applied to $f = \varphi$ restricted to $J = (l_\alpha,m_\alpha]$ and $[m_\alpha,r_\alpha)$, we have
\begin{gather*}
|\varphi(m_\alpha)-m(\varphi,[l_\alpha, m_\alpha])| \leq \frac{2^{a}p_a(\varphi)}{1-a}\frac{1}{|I_\alpha|^a},\\
|\varphi(m_\alpha)-m(\varphi,[m_\alpha, r_\alpha])| \leq \frac{2^{a}p_a(\varphi)}{1-a}\frac{1}{|I_\alpha|^a}.
\end{gather*}
Therefore
\[|m(\varphi,[m_\alpha, r_\alpha])-m(\varphi,[l_\alpha, m_\alpha])| \leq \frac{2^{a+1}p_a(\varphi)}{1-a}\frac{1}{|I_\alpha|^a}.\]
As $m(\varphi, I_\alpha) = (m(\varphi, [l_\alpha,m_\alpha]) + m(\varphi, [m_\alpha,r_\alpha]))/2$, it follows that
\begin{align}\label{eq:mlm}
\begin{split}
|m(\varphi, [l_\alpha, m_\alpha])-m(\varphi, I_\alpha)|&=\frac{|m(\varphi, [l_\alpha, m_\alpha])-m(\varphi, [m_\alpha, r_\alpha])|}{2}
\leq \frac{p_a(\varphi)2^{a}}{(1-a)|I_\alpha|^{a}},\\
|m(\varphi, [m_\alpha,r_\alpha])-m(\varphi, I_\alpha)|&
\leq \frac{p_a(\varphi)2^{a}}{(1-a)|I_\alpha|^{a}}.
\end{split}
\end{align}

If $0<a<1$ then we can apply \eqref{neq:fsa} to $f = \varphi$ restricted to $J = (l_\alpha,m_\alpha]$ and $[m_\alpha,r_\alpha)$ and taking $C = p_a(\varphi)$. This gives
\begin{align*}
|\varphi(x) - m(\varphi, [l_\alpha, m_\alpha])| &\leq p_a(\varphi)\left(\frac{1}{a(x-l_\alpha)^{a}}+ \frac{2^{a+1}}{a(1-a)|I_\alpha|^{a}}\right), \quad  \text{ if } x \in (l_\alpha, m_\alpha],\\
|\varphi(x) - m(\varphi, [m_\alpha, r_\alpha])| &\leq p_a(\varphi)\left(\frac{1}{a(r_\alpha-x)^{a}}+ \frac{2^{a+1}}{a(1-a)|I_\alpha|^{a}}\right), \quad   \text{ if } x \in [m_\alpha, r_\alpha).
\end{align*}
Together with \eqref{eq:mlm} this yields \eqref{eqn;upperboundvarphi}.
\medskip

If $a=0$ then we can apply \eqref{neq:fs0} to $f = \varphi$ restricted to $J = (l_\alpha,m_\alpha]$ and $[m_\alpha,r_\alpha)$ and taking $C = p_a(\varphi)$. This gives
\begin{align*}
\begin{split}
&|\varphi(x) - m(\varphi, [l_\alpha, m_\alpha])| \leq p_a(\varphi)\left(\log\frac{|I_\alpha|}{2(x-l_\alpha)}+1 \right),\quad  \text{ if } x \in (l_\alpha, m_\alpha],\\
&|\varphi(x) - m(\varphi, [m_\alpha, r_\alpha])| \leq p_a(\varphi)\left(\log\frac{|I_\alpha|}{2(r_\alpha-x)}+1 \right),\quad \text{ if } x \in [m_\alpha, r_\alpha).
\end{split}
\end{align*}
Together with \eqref{eq:mlm} this yields \eqref{eqn;upperboundvarphi2}.
\end{proof}

\begin{lemma}
For every $0\leq a<1$, $\varphi \in \pa(\sqcup_{\alpha \in \mathcal{A}}I_\alpha )$ and $\alpha\in\mathcal{A}$ we have
\begin{equation}
\frac{1}{|I_\alpha|}\int_{I_\alpha}|\varphi(x)-m(\varphi,I_\alpha)|\,dx \leq \frac{2^{2+a}p_a(\varphi)}{1-a}\frac{1}{|I_\alpha|^a}.
\label{eqn;paosmean}
\end{equation}
\end{lemma}

\begin{proof}
By \eqref{neq:intfs} applied to $f = \varphi$ restricted to $J = (l_\alpha,m_\alpha]$ and $[m_\alpha,r_\alpha)$, we have
\begin{gather*}
\frac{1}{|I_\alpha|}\int_{[l_\alpha, m_\alpha]}|\varphi(x)-m(\varphi,[l_\alpha, m_\alpha])|\,dx \leq \frac{2^{a}p_a(\varphi)}{1-a}\frac{1}{|I_\alpha|^a},\\
\frac{1}{|I_\alpha|}\int_{[m_\alpha, r_\alpha]}|\varphi(x)-m(\varphi,[m_\alpha, r_\alpha])|\,dx \leq \frac{2^{a}p_a(\varphi)}{1-a}\frac{1}{|I_\alpha|^a}.
\end{gather*}
In view of \eqref{eq:mlm}, it follows that
\begin{gather*}
\frac{1}{|I_\alpha|}\int_{[l_\alpha, m_\alpha]}|\varphi(x)-m(\varphi,I_\alpha)|\,dx \leq \frac{2^{1+a}p_a(\varphi)}{1-a}\frac{1}{|I_\alpha|^a},\\
\frac{1}{|I_\alpha|}\int_{[m_\alpha, r_\alpha]}|\varphi(x)-m(\varphi,I_\alpha)|\,dx \leq \frac{2^{1+a}p_a(\varphi)}{1-a}\frac{1}{|I_\alpha|^a}.
\end{gather*}
Summing up, we obtain
\[\frac{1}{|I_\alpha|}\int_{I_\alpha}|\varphi(x)-m(\varphi,I_\alpha)|\,dx \leq \frac{2^{2+a}p_a(\varphi)}{1-a}\frac{1}{|I_\alpha|^a},
\]
which completes the proof.
\end{proof}

We finish the section by introducing a lower bound for function $\varphi$. This will be crucial in handling some lower bounds of renormalized cocycles and Birkhoff integrals. (See the  part V in the proof of Theorem~\ref{theorem;main} in \S\ref{sec:pmt}.)

\begin{lemma}\label{lem:belowest}
Suppose that $\varphi:J\to\R$ ($J=(x_0,x_1]$) is an integrable  $C^1$-function such that
\[0<c\leq|(x-x_0)^{1+a }\varphi'(x)|\text{ for all }x\in J.\]
Then there exists a sub-interval $\widehat{J}\subset J$ such that
\begin{equation}\label{eq:belowest}
|\widehat{J}|\geq |J|/4 \text{ and }|\varphi(x)|\geq \frac{c}{4|J|^a}\text{ for all }x\in \widehat{J}.
\end{equation}
\end{lemma}

\begin{proof}
We first show that for every $\xi>0$ we have
\begin{equation}\label{eq:lebest}
Leb\{x\in J:|\varphi(x)|\leq \xi\}\leq \frac{2\xi|J|^{1+a}}{c}.
\end{equation}
Note that $\varphi$ is strictly monotonic. We focus on the strictly decreasing case, i.e.\
$-(x-x_0)^{1+a }\varphi'(x)\geq c$ for all $x\in J$. Then $\varphi(x)\to+\infty$ as $x\searrow x_0$. The proof in the strictly increasing case follows by the same way.
Suppose that the set $\{x\in J:|\varphi(x)|\leq \xi\}$ is not empty. Then it is an interval $[y_1,y_2]\subset J$
such that $\varphi(y_1)=\xi$ and $\varphi(y_2)=\xi_2\geq -\xi$. It follows that
\begin{align*}
y_2-y_1&=\varphi^{-1}(\xi_2)-\varphi^{-1}(\xi)=
-\int_{\xi_2}^{\xi}(\varphi^{-1})'(x)\,dx
=-\int_{\xi_2}^{\xi}\frac{1}{\varphi'(\varphi^{-1}(x))}\,dx\\
&\leq\int_{\xi_2}^{\xi}\frac{(\varphi^{-1}(x)-x_0)^{1+a}}{c}\,dx
\leq \frac{|J|^{1+a}}{c}(\xi-\xi_2)\leq \frac{2\xi|J|^{1+a}}{c}.
\end{align*}
Applying \eqref{eq:lebest} to $\xi=\frac{c}{4|J|^a}$, we have that $\{x\in J:|\varphi(x)|\leq \frac{c}{4|J|^a}\}$ is an interval
whose length is at most
\[\frac{2\frac{c}{4|J|^a}|J|^{1+a}}{c}=\frac{|J|}{2}.\]
Since $\{x\in J:|\varphi(x)|\geq \frac{c}{4|J|^a}\}$ consists of at most two intervals and its measure is at least $|J|/2$,
one of these intervals satisfies \eqref{eq:belowest}.
\end{proof}

\section{Renormalization of cocycles}\label{sect;renormalization}
In this section we review a renormalization operator on cocycles over IETs derived from the renormalizations of an IETs given by accelerated Rauzy-Veech induction.

\subsection{Special Birkhoff sums}\label{SpecialBS}


Assume that an IET $T : I \rightarrow I$ satisfies Keane's condition.
 For any $0\leq k<l$ and any measurable cocycle $\varphi:I^{(k)}\to\R$ for the IET
$T^{(k)}:I^{(k)}\to I^{(k)}$, denote by
$S(k,l)\varphi:I^{(l)}\to\R$ the renormalized cocycle for
$T^{(l)}$ given by
\[S(k,l)\varphi(x)=\sum_{0\leq i<Q_{\beta}(k,l)}\varphi((T^{(k)})^ix)\text{ for }x\in I^{(l)}_\beta.\]
We write $S(k)\varphi$ for $ S(0,k)\varphi$ and we  adhere to  the
convention that $S(k,k)\varphi=\varphi$. Sums of this form are
usually called \emph{special Birkhoff sums}. If $\varphi$ is
integrable then
\begin{equation}\label{nase}
\| S(k,l)\varphi\|_{L^1(I^{(l)})}\leq
\|\varphi\|_{L^1(I^{(k)})}\quad\text{and}\quad
\int_{I^{(l)}}S(k,l)\varphi(x)\,dx=
\int_{I^{(k)}}\varphi(x)\,dx.
\end{equation}
Clearly, $S(k,l)\Gamma^{(k)} = \Gamma^{(l)}$ and $S(k,l)$ is the linear automorphism of $\R^{\mathcal A}$ whose matrix in the canonical basis is $Q(k,l)$.

Recalling the definition of $C^\pm_{\alpha}$ (see Definition \ref{def;pa}), the following lemma is simply mimics of Lemma~3.1~and~3.3 in \cite{Fr-Ul}. Since its proof
proceeds in the same way as the proofs of the mentioned lemmas in \cite{Fr-Ul}, we omit it.

\begin{lemma}\label{lem:CSk}
Suppose that $\varphi\in\pag(\sqcup_{\alpha\in \mathcal{A}}
I_{\alpha})$. For every $k>0$ we have $S(k)\varphi\in\pag(\sqcup_{\alpha\in \mathcal{A}}
I^{(k)}_{\alpha})$ and there exists a permutation $\chi:\mathcal{A}\to\mathcal{A}$ such that
\[C^+_{\alpha}(S(k)\varphi)=C^+_{\alpha}(\varphi)\text{ and }C^-_{\alpha}(S(k)\varphi)=C^-_{\chi(\alpha)}(\varphi)\]
for every $\alpha\in\mathcal{A}$. Moreover, there are distinct $\alpha_*$, $\alpha^*$ in $\mathcal{A}$ such that
\begin{gather*}
\big\{T^jl^{(k)}_\alpha:0\leq j<Q_\alpha(k)\big\}\cap End(T)=\{l_\alpha\}\text{ for every }\alpha\in\mathcal{A}\\
\big\{\widehat{T}^jr^{(k)}_\alpha:0\leq j<Q_\alpha(k)\big\}\cap End(T)=
\left\{
\begin{array}{cl}
\{r_{\chi(\alpha)}\}&\text{if }\alpha\in\mathcal{A}\setminus\{\alpha_*,\alpha^*\}\\
\big\{r_{\pi_0^{-1}(d)},r_{\pi_1^{-1}(d)}\big\}&\text{if }\alpha=\alpha^*\\
\emptyset&\text{if }\alpha=\alpha_*.
\end{array}
\right.
\end{gather*}
\end{lemma}

In view of the above Lemma, for all $0\leq k<l$ the operator $S(k,l)$
maps $\pag(\sqcup_{\alpha\in \mathcal{A}}
I^{(k)}_{\alpha})$ into $\pag(\sqcup_{\alpha\in
\mathcal{A}} I^{(l)}_{\alpha})$. The following three results give estimates for the increase $p_a$ along the  sequence of renormalized cocycles.

\begin{proposition}\label{renormpaos}
Suppose that $T=T_{(\pi,\lambda)}$ satisfies Keane's condition. Then for every $0<a<1$, $\varphi \in \pag(\sqcup_{\alpha \in \mathcal{A}}I_\alpha )$ and $k \geq 1$,
\begin{equation}\label{eqn;renormpaos}
p_a(S(k)\varphi) \leq \left(2+2d\zeta(1+a)\left(\frac{\max\{|I^{(k)}_\alpha|:\alpha\in\mathcal{A}\}}{\min\{|I^{(k)}_\alpha|:\alpha\in\mathcal{A}\}}\right)^{1+a}\right) p_a(\varphi).
\end{equation}
\end{proposition}
\begin{proof}
Suppose that $x\in (l_\beta^{(k)},m_\beta^{(k)}]$ ($m_\beta^{(k)}$ is the middle point of $I_\beta^{(k)}$) for some $\beta\in\mathcal{A}$.
By Lemma~\ref{lem:CSk}, there exists $0\leq j_\beta < Q_{\beta}(k)$ such that
$l_\beta=T^{j_\beta}l^{(k)}_\beta$. It follows that
\begin{align}\label{eq:jbeta}
\begin{split}
|\varphi'(T^{j_\beta}x)(x-l_\beta^{(k)})^{1+a}|&= |\varphi'(T^{j_\beta}x)(T^{j_\beta}x-T^{j_\beta}l_\beta^{(k)})^{1+a}|\\
&=|\varphi'(T^{j_\beta}x)(T^{j_\beta}x-l_\beta)^{1+a}| \leq p_a(\varphi).
\end{split}
\end{align}
For any $0\leq j < Q_{\beta}(k)$, $j\neq j_\beta$ denote by $e_{\alpha_j}$ an element of the set $End(T)$ that is the closest to $T^jx$,
more precisely $e_{\alpha_j}=l_{\alpha_j}$ if $T^jx\in(l_{\alpha_j},m_{\alpha_j}]$ and $e_{\alpha_j}=r_{\alpha_j}$ if $T^jx\in(m_{\alpha_j},r_{\alpha_j})$. Then
\begin{equation}\label{eq:jneqjbeta}
|\varphi'(T^{j}x)(T^{j}x-e_{\alpha_j})^{1+a}| \leq p_a(\varphi).
\end{equation}
Since $T^{j}x\in T^j(l_\beta^{(k)},m_\beta^{(k)}]$ with $j\neq j_\beta$ and the elements of $End(T)$ lie on the boundary of intervals
$T^iI^{(k)}_\gamma$ for $0\leq i<Q_\gamma(k)$, $\gamma\in\mathcal{A}$, we have
\begin{equation}\label{eq:jneqjbeta2}
|T^{j}x-e_{\alpha_j}|\geq \min\{\tfrac{1}{2}|I^{(k)}_{\beta}|,\min\{|I^{(k)}_\gamma|:\gamma\in\mathcal{A}\}\}\geq \tfrac{1}{2}\min\{|I^{(k)}_\gamma|:\gamma\in\mathcal{A}\}.
\end{equation}
Suppose that for some $j\neq j'$ we have $\alpha_{j}=\alpha_{j'}=\alpha$. Since $T^jx\in T^jI^{(k)}_\beta$ and $T^{j'}x\in T^{j'}I^{(k)}_\beta$ and they are equally distant from the ends of the intervals, we have
\begin{equation}\label{eq:jneqjbeta3}
|(T^{j}x-e_{\alpha_j})-(T^{j'}x-e_{\alpha_{j'}})|=|T^{j}x-T^{j'}x|\geq |I^{(k)}_{\beta}|.
\end{equation}
In view of \eqref{eq:jneqjbeta}, it follows that for every $\alpha\in\mathcal{A}$ and $e_\alpha\in I_\alpha$ we have
\begin{align*}
\sum_{\substack{0\leq j<Q_\beta(k)\\ j\neq j_\beta,
e_{\alpha_j}=e_\alpha}}|\varphi'(T^jx)|&\leq p_a(\varphi)\sum_{\substack{0\leq j<Q_\beta(k)\\ j\neq j_\beta,
e_{\alpha_j}=e_\alpha}}\frac{1}{|T^{j}x-e_{\alpha}|^{1+a}}\\
&\leq p_a(\varphi)\sum_{l\geq 0}\frac{1}{\big|\tfrac{1}{2}\min\{|I^{(k)}_\gamma|:\gamma\in\mathcal{A}\}+l|I^{(k)}_\beta|\big|^{1+a}}\\
&\leq \frac{2^{1+a}p_a(\varphi)}{\min\{|I^{(k)}_\gamma|:\gamma\in\mathcal{A}\}^{1+a}}\sum_{l\geq 1}\frac{1}{l^{1+a}}\\
&\leq \frac{\zeta(1+a)2^{1+a}p_a(\varphi)}{\min\{|I^{(k)}_\gamma|:\gamma\in\mathcal{A}\}^{1+a}}.
\end{align*}
Therefore, by \eqref{eq:jbeta}, we have
\begin{align*}
|S(k)&\varphi'(x)(x-l^{(k)}_\beta)^{1+a}|=\Big|\sum_{0\leq j<Q_\beta(k)}\varphi'(T^jx)(x-l^{(k)}_\beta)^{1+a}\Big|\\
&\leq |\varphi'(T^{j_\beta}x)(x-l^{(k)}_\beta)^{1+a}|+\sum_{\substack{\alpha\in\mathcal{A}\\e_\alpha\in \{l_\alpha,r_\alpha\}}}\sum_{\substack{0\leq j<Q_\beta(k)\\ j\neq j_\beta,
e_{\alpha_j}=e_\alpha}}|\varphi'(T^jx)|(|I^{(k)}_\beta|/2)^{1+a}\\
&\leq p_a(\varphi)\left(1+2d\zeta(1+a)\left(\frac{|I^{(k)}_\beta|}{\min\{|I^{(k)}_\gamma|:\gamma\in\mathcal{A}\}}\right)^{1+a}\right).
\end{align*}

For the other cases, we assume that $x\in (m_\beta^{(k)},r_\beta^{(k)})$  for some $\beta\in\mathcal{A}\setminus\{\alpha_*,\alpha^*\}$ or $\beta\in\{\alpha_*,\alpha^*\}$.
For the former case, by Lemma~\ref{lem:CSk}, there exists $0\leq j_\beta < Q_{\beta}(k)$ such that
$r_{\chi(\beta)}=\widehat{T}^{j_\beta}r^{(k)}_\beta$. It follows that
\begin{align}\label{eqn;case2pa}
\begin{split}
|\varphi'(T^{j_\beta}x)(r_\beta^{(k)}-x)^{1+a}|&= |\varphi'(T^{j_\beta}x)(\widehat{T}^{j_\beta}r_{\beta}-T^{j_\beta}x)^{1+a}|\\
&=|\varphi'(T^{j_\beta}x)(r_{\chi(\beta)}-T^{j_\beta}x)^{1+a}| \leq p_a(\varphi).
\end{split}
\end{align}
The estimate from above of $|\varphi'(T^{j}x)|$ for $j\neq j_\beta$ is the same as previous case. Therefore we obtain
\begin{equation}\label{eqn;analogous}
|S(k)\varphi'(x)(r^{(k)}_\beta-x)^{1+a}|\leq p_a(\varphi)\left(1+2d\zeta(1+a)\left(\frac{|I^{(k)}_\beta|}{\min\{|I^{(k)}_\gamma|:\gamma\in\mathcal{A}\}}\right)^{1+a}\right).
\end{equation}

For the latter cases, if $x\in (m_{\alpha^*}^{(k)},r^{(k)}_{\alpha^*})$, then
by Lemma~\ref{lem:CSk}, there exists $0\leq j_{\alpha^*} < Q_{{\alpha^*}}(k)-1$ such that
$r_{\pi_1^{-1}(d)}=\widehat{T}^{j_{\alpha^*}}r^{(k)}_{\alpha^*}$ and $r_{\pi_0^{-1}(d)}=\widehat{T}^{j_{\alpha^*}+1}r^{(k)}_{\alpha^*}$. It follows that
the estimate \eqref{eqn;case2pa} from above of $|\varphi'(T^{j}x)|$ for $j\neq j_{\alpha^*}, j_{\alpha^*}+1$ is the same as \eqref{eq:jbeta} in the first case.

Lastly, if $x\in (m_{\alpha_*}^{(k)},r^{(k)}_{\alpha_*})$, then by Lemma~\ref{lem:CSk} again, $\{\widehat{T}^jr^{(k)}_{\alpha_*}:0\leq j<Q_\alpha(k)\}\cap End(T)=\emptyset$.
Since  the elements of $End(T)$ lie on the boundary of intervals
$T^iI^{(k)}_\gamma$ for $0\leq i<Q_\gamma(k)$, $\gamma\in\mathcal{A}$, it follows that for every $0\leq j<Q_{\alpha_*}(k)$ we have
\begin{equation*}
\operatorname{dist}(T^jx,End(T))\geq \min\{\tfrac{1}{2}|I^{(k)}_{\beta}|,\min\{|I^{(k)}_\gamma|:\gamma\in\mathcal{A}\}\}\geq \tfrac{1}{2}\min\{|I^{(k)}_\gamma|:\gamma\in\mathcal{A}\}.
\end{equation*}
Repeating the same argument as in the first case, we obtain similar bound of \eqref{eqn;analogous}.
\end{proof}

For $a =0$,  the proof mimics the idea applied in the estimate for logarithmic singularity type from \cite[\S3]{Fr-Ul} and \cite[\S5.2-5.3]{Fr-Ul2}.
The proof is carried out in each sub-cases as in the proof of Proposition~\ref{renormpaos} and  the reasonings are repetitive.
\begin{proposition}
Suppose that $T_{(\pi,\lambda)}$ satisfies Keane's condition. Then for every $\varphi \in \pog(\sqcup_{\alpha \in \mathcal{A}}I_\alpha )$ and $k \geq 1$ we have
\begin{equation}\label{eqn;renormpaos0}
p_0(S(k)\varphi) \leq  p_0(\varphi)\left(2+2d(1+ \log \|Q(k)\|)\left(\frac{\max\{|I^{(k)}_\alpha|:\alpha\in\mathcal{A}\}}{\min\{|I^{(k)}_\alpha|:\alpha\in\mathcal{A}\}}\right)\right).
\end{equation}
\end{proposition}

\begin{corollary}\label{cor;renormpaos}
For every IET $T$ satisfying \ref{FDC} there exists $C\geq 1$ such that
for all $0 \leq k \leq l$ and for every function $\varphi \in \pag(\sqcup_{\alpha \in \mathcal{A}}I^{(k)}_\alpha )$ we have
\begin{align}\label{eqn;renormpaos2}
\begin{split}
&p_a(S(k,l)\varphi) \leq Cp_a(\varphi)\text{ if }0<a<1,\\
&p_a(S(k,l)\varphi) \leq C(1+\log\|Q(k,l)\|)p_a(\varphi)\text{ if }a=0.
\end{split}
\end{align}
\end{corollary}

We finish this section by giving a lower bound derived from the proof of Proposition~\ref{renormpaos}.
\begin{proposition}\label{prop:renormbelow}
Suppose that $T=T_{(\pi,\lambda)}$ satisfies Keane's condition and $0<a<1$. Assume that $\varphi \in \pag(\sqcup_{\alpha \in \mathcal{A}}I_\alpha)$ and $\alpha\in\mathcal{A} $ are such that  $C^+_\alpha(\varphi)\neq 0$ or $C^-_\alpha(\varphi)\neq 0$. Then choose $\delta>0$ so that
\begin{align*}
&|(x-l_\alpha)^{1+a}\varphi'(x)|\geq |C^+_\alpha(\varphi)|/2\text{ for all }x\in (l_\alpha,l_\alpha+\delta]\text{ or}\\
&|(r_\alpha-x)^{1+a}\varphi'(x)|\geq |C^-_\alpha(\varphi)|/2\text{ for all }x\in [r_\alpha-\delta,r_\alpha).
\end{align*}
If $k \geq 1$ is such that $|I^{(k)}|\leq \delta$ then for every $x\in(l^{(k)}_\alpha, m^{(k)}_\alpha]$ we have
\begin{equation*}\label{eqn;renormbelow1}
|S(k)\varphi'(x)| \geq \frac{|C^+_\alpha|}{2(x-l^{(k)}_\alpha)^{1+a}}-\frac{2^{2+a}d\zeta(1+a)p_a(\varphi)}{|I^{(k)}_\alpha|^{1+a}}
\left(\frac{\max\{|I^{(k)}_\beta|:\beta\in\mathcal{A}\}}{\min\{|I^{(k)}_\beta|:\beta\in\mathcal{A}\}}\right)^{1+a},
\end{equation*}
or for every $x\in[m^{(k)}_{\chi^{-1}(\alpha)}, r^{(k)}_{\chi^{-1}(\alpha)})$ we have
\begin{equation*}\label{eqn;renormbelow2}
|S(k)\varphi'(x)| \geq \frac{|C^-_\alpha|}{2(r^{(k)}_{\chi^{-1}(\alpha)}-x)^{1+a}}-\frac{2^{2+a}d\zeta(1+a)p_a(\varphi)}{|I^{(k)}_{\chi^{-1}(\alpha)}|^{1+a}}
\left(\frac{\max\{|I^{(k)}_\beta|:\beta\in\mathcal{A}\}}{\min\{|I^{(k)}_\beta|:\beta\in\mathcal{A}\}}\right)^{1+a}.
\end{equation*}
\end{proposition}

\begin{proof}
We focus only on the case $C^+_\alpha(\varphi)\neq 0$. The other case can be treated in the same way. By the proof of Proposition~\ref{renormpaos},
for every $x\in(l^{(k)}_\alpha, m^{(k)}_\alpha]$ we have
\[\Big|\sum_{\substack{0\leq j<Q_\beta(k)\\ j\neq j_\alpha}}\varphi(T^jx)\Big|\leq \frac{2^{2+a}d\zeta(1+a)p_a(\varphi)}{|I^{(k)}_\alpha|^{1+a}}
\left(\frac{\max\{|I^{(k)}_\beta|:\beta\in\mathcal{A}\}}{\min\{|I^{(k)}_\beta|:\beta\in\mathcal{A}\}}\right)^{1+a}\]
and
\[|\varphi'(T^{j_\alpha}x)(x-l_\alpha^{(k)})^{1+a}|
=|\varphi'(T^{j_\alpha}x)(T^{j_\alpha}x-l_\alpha)^{1+a}|\geq |C^{+}_\alpha|/2.\]
This yields our claim.
\end{proof}

\section{Correction operator}\label{sec;correction}
In this section we define the whole family of correction operators $\mathfrak{h}_j: \pag(\sqcup_{\alpha \in \mathcal{A}} I_\alpha) \rightarrow U_j\subset H(\pi)$ for $0\leq a<1$ with $2\leq j\leq g+1$
determined by  $\lambda_j\leq \lambda_1 a<\lambda_{j-1}$.
These operators are generalizations of the correction operator introduced in \cite{Fr-Ul2} for the cocycles with logarithmic singularities.
The correction operator allows us to correct a cocycle with polynomial singularities by a piecewise
constant function, so that we have better control over the growth of the special Birkhoff sums of the corrected cocycle.

\subsection{Correction operator for cocycles with polynomial singularities}\label{sec; COC}

Recall the projection operator $\M$ defined on $L^1(I)$ in \eqref{eqn;projec}. For every step of (accelerated) renormalization  let us consider the corresponding projection operators on the vector space, i.e.\ for every $k \geq 0$ let
\[
\M^{(k)}: \pag(\sqcup_{\alpha \in \mathcal{A}}I_\alpha ) \rightarrow \Gamma^{(k)} \quad\text{be given by}\quad
{\M^{(k)}(\varphi)= \sum_{\alpha \in \mathcal{A}}m(\varphi, I)\chi_{I^{(k)}_\alpha}}.
\]

\begin{theorem}\label{thm;correction}
Assume that $T$ satisfies \ref{FDC}. For any $0\leq a<1$ take $2\leq j\leq g+1$ so that $\lambda_j\leq \lambda_1 a<\lambda_{j-1}$. There exists a bounded linear operator
$\mathfrak{h}_j: \pag(\sqcup_{\alpha \in \mathcal{A}} I_\alpha) \rightarrow U_j $ such that for any $\tau>0$ there exists a constant $C=C_\tau\geq 1$ such that for every $\varphi \in \pag(\sqcup_{\alpha \in \mathcal{A}} I_\alpha) $ with $\mathfrak{h}_j(\varphi) = 0$ we have
\begin{equation}\label{eqn;maincorrection}
\|\mathcal{M}^{(k)}(S(k)\varphi)\| \leq C\left(\big(K^{a,j,\tau}_k+C^{a,j,\tau}_k\big) p_a(\varphi)+ \|Q_{E_j}(k)\| \frac{\norm{\varphi}_{L^1(I^{(0)})}}{|I^{(0)}|}\right).
\end{equation}
\end{theorem}

In view of Proposition~\ref{prop;FDCbound} and the fact that $\|Q_{E_j}(k)\|=O(e^{(\lambda_j+\tau)k}) \leq O(e^{(\lambda_1a+\tau)k})$ (see \eqref{def;sdc1} combined with \eqref{def;sdc0}), we have the following result.
\begin{corollary}\label{cor;correction}
Assume that $T$ satisfies \ref{FDC}. Then for every $\varphi \in \pag(\sqcup_{\alpha \in \mathcal{A}} I_\alpha) $ with $\mathfrak{h}_j(\varphi) = 0$, for every small enough $\tau>0$, there exists a FDC-acceleration such that
\begin{equation}\label{eqn;S(k)correction}
\|\mathcal{M}^{(k)}(S(k)\varphi)\|=O(e^{(\lambda_1 a+\tau)k}).
\end{equation}
\end{corollary}

Corollary \ref{cor;correction} plays a crucial role in determining the growth rate of ergodic integrals in \S\ref{sec;deviation}.
Before starting the proof of Theorem \ref{thm;correction}, we provide some basic inequalities regarding the operator $\M^{(k)}$.
\medskip

Note that, by definition and \eqref{eqn;paosmean},
\begin{gather}\label{eqn;normMk}
\big\|\M^{(k)}(\varphi)\big\|_{L^1(I^{(k)})} \leq 2 \norm{\varphi}_{L^1(I^{(k)})}\\
\label{eqn;Mkaverage}
\big\|\varphi - \M^{(k)}(\varphi)\big\|_{L^1(I^{(k)})} \leq  \frac{2^{2+a}d}{1-a}p_a(\varphi)|I^{(k)}|^{1-a}.
\end{gather}
As for every $h \in \Gamma^{(k)}$,
\begin{equation}\label{equiv-norm}
\frac{|I^{(k)}| \norm{h}}{\kappa} \leq \min_{\beta \in \mathcal{A}} |I_\beta^{(k)}|\norm{h} \leq \norm{h}_{L^1(I^{(k)})} \leq |I^{(k)}|\norm{h},
\end{equation}
this gives
\begin{equation}\label{eqn;Mknorm}
\|\M^{(k)}(\varphi)\| \leq \frac{2\kappa }{|I^{(k)}|}\norm{\varphi}_{L^1(I^{(k)})}.
\end{equation}
Let us consider a linear operator
$P^{(k)}_0: \pag(\sqcup_{\alpha \in \mathcal{A}} I_\alpha^{(k)}) \rightarrow \pag(\sqcup_{\alpha \in \mathcal{A}} I_\alpha^{(k)})$ given by
\[P^{(k)}_0(\varphi) = \varphi - \M^{(k)}(\varphi).\]
In view of \eqref{eqn;Mkaverage} and \eqref{eq:Ikl}, for all $ k\geq 0$ we have
\begin{equation}\label{eqn;P0k}
\frac{1}{|I^{(k)}|}\big\|P^{(k)}_0(\varphi)\big\|_{L^1(I^{(k)})}  \leq \frac{2^{2+a}d}{1-a}\frac{p_a(\varphi)}{|I^{(k)}|^{a}} \leq \frac{2^{2+a}d}{(1-a)|I|^a}\|Q(k)\|^ap_a(\varphi).
\end{equation}

\begin{proof}[Proof of Theorem \ref{thm;correction}]
Let us denote $v_k := \M^{(k)} \circ S(k)(\varphi)$.
Direct calculation shows that
\begin{align*}
\big(S(k,k+1)&\circ P^{(k)}_0 \circ S(k)(\varphi) -P_0^{(k+1)}\circ
S(k,k+1)\circ S(k)(\varphi)\big) \\
& =  -S(k,k+1)\circ \M^{(k)} \circ S(k)(\varphi) + \M^{(k+1)}\circ S(k+1)(\varphi)\\
& = -Z(k+1)v_k + v_{k+1}.
\end{align*}
By \eqref{nase},  
\eqref{eqn;P0k}, \eqref{eqn;renormpaos2} and  \eqref{eq:Ikl},
\begin{align*}
\frac{1}{|I^{(k+1)}|}\big\|S(k,k+1)\circ P^{(k)}_0 \circ S(k)(\varphi)\big\|_{L^1(I^{(k+1)})} & \leq \frac{1}{|I^{(k+1)}|}\big\|P^{(k)}_0 \circ S(k)(\varphi)\big\|_{L^1(I^{(k)})}\\
& \leq \frac{|I^{(k)}|}{|I^{(k+1)}|}\frac{2^{2+a}d}{(1-a)|I|^a}\norm{Q(k)}^ap_a(S(k)\varphi)\\
& \leq C_a \langle \|Q(k)\|\rangle^a \|Z(k+1)\|p_a(\varphi),
\end{align*}
and similarly,
\begin{align*}
\frac{1}{|I^{(k+1)}|}\big\|P_0^{(k+1)}\circ S(k+1)(\varphi)\big\|_{L^1(I^{(k+1)})} \leq C_a \langle \|Q(k+1)\|\rangle^ap_a(\varphi).
\end{align*}
Therefore,  for any $0\leq a < 1$, 
\begin{align*}
\frac{1}{|I^{(k+1)}|}\norm{Z(k+1)v_k - v_{k+1}}_{L^1(I^{(k+1)})} &\leq  C_a\left(\langle \|Q(k)\|\rangle^a \|Z(k+1)\| + \langle \|Q(k+1)\|\rangle^a\right)p_a(\varphi)\\
& \leq C'_a  \langle \|Q(k)\|\rangle^a \|Z(k+1)\|p_a(\varphi).
\end{align*}
By \eqref{equiv-norm}, this gives
\begin{equation}\label{eqn;anv}
\norm{Z(k+1)v_k-v_{k+1}}\leq \kappa C'_a \langle \|Q(k)\|\rangle^a \|Z(k+1)\|p_a(\varphi).
\end{equation}

\medskip

For any sequence $(x_k)_{k\geq 0}$ in $\R^{\mathcal{A}}$, let $\Delta x_{k+1} = x_{k+1}-Z(k+1)x_k$ for $k\geq 0$ and $\Delta x_0=x_0$.
Then, by telescoping,
\begin{align}\label{def;ehk}
x_{k} = \sum_{j=0}^{k} Q(j,k) \Delta x_j.
\end{align}
Letting $v_k = \M^{(k)} \circ S(k)(\varphi)$, by  \eqref{eqn;Mknorm} and \eqref{eqn;anv}, 
\begin{equation}\label{eqn;delta0}
\norm{\Delta v_0}  \leq  \frac{2\kappa}{|I^{(0)}|} \norm{\varphi}_{L^1(I^{(0)})}\text{ and }\norm{\Delta v_{k+1}}\leq \kappa C'_a \langle \|Q(k)\|\rangle^a \|Z(k+1)\|p_a(\varphi).
\end{equation}

Now we claim that for any $\tau>0$, there exists a vector $v\in U_j^{(0)}$ such that
\begin{equation}\label{eqn;anv2}
\norm{Q(k)v - v_{k}}\leq O\big((p_a(\varphi)+\|\varphi\|_{L^1(I)} )(K^{a,j,\tau}_k+C^{a,j,\tau}_k)\big).
\end{equation}
For every $k\geq 0$ let $e_k=P^{(k)}_{E_j}v_k\in E_j^{(k)}$ and $u_k=P^{(k)}_{U_j}v_k\in U_j^{(k)}$. Then $v_ k = u_k + e_k$. Since $Z(k+1)(E_j^{(k)})=E_j^{(k+1)}$ and $Z(k+1)(U_j^{(k)})=U_j^{(k+1)}$ we have
\begin{align*}
\Delta u_{k+1} = u_{k+1}-Z(k+1)u_k = P^{(k+1)}_{U_j}\Delta v_{k+1}&, \ \Delta e_{k+1} = e_{k+1}-Z(k+1)e_k=P^{(k+1)}_{E_j}\Delta v_{k+1},\\
\Delta u_0=u_0=P^{(0)}_{U_j}\Delta v_{0}&, \ \Delta e_0 = e_0=P^{(0)}_{E_j}\Delta v_{0}.
\end{align*}
In view of \eqref{eqn;delta0} and \eqref{eqn;projbound}, we have
\begin{align}\label{u0e0}
\norm{\Delta u_0}\leq C\norm{\Delta v_0}\leq \frac{2C\kappa}{|I^{(0)}|} \norm{\varphi}_{L^1(I^{(0)})}, \quad \norm{\Delta e_0}\leq C\norm{\Delta v_0}\leq \frac{2C\kappa}{|I^{(0)}|} \norm{\varphi}_{L^1(I^{(0)})},
\end{align}
and for every $k\geq 1$ we have
\begin{align}\label{ukek}
\begin{split}
\norm{\Delta u_{k}} &\leq C_{a,\tau} \|Z(k)\|\langle \|Q(k-1)\|\rangle^{a} \|Q(k)\|^\tau p_a(\varphi), \\
\norm{\Delta e_{k}} &\leq C_{a,\tau} \|Z(k)\|\langle \|Q(k-1)\|\rangle^{a} \|Q(k)\|^\tau p_a(\varphi).
\end{split}
\end{align}
Let us consider an infinite series $v :=  \sum_{l \geq 0} Q(0,l)^{-1}\Delta u_l$. For any $\tau>0$ small enough, by \eqref{ukek}, \eqref{u0e0}
and the definition of $K^{a,j,\tau}_l$,
\begin{align}\label{eqn;inftyv}
\begin{split}
\norm{v} &\leq \sum_{l \geq 0} \|Q|_{U_j}(0,l)^{-1}\|\|\Delta u_l\| \\
& \leq \|\Delta u_0\|+C'_{a,\tau}  p_a(\varphi)\sum_{l \geq 0} \|Q|_{U_j}(0,l+1)^{-1}\| \|Z(l+1)\| \langle \|Q(l)\|\rangle^{a}\| \|Q(l+1)\|^\tau \\
 &= C'_{a,\tau}\Big(\frac{2\kappa}{|I^{(0)}|} \norm{\varphi}_{L^1(I^{(0)})}+K^{a,j,\tau}_0p_a(\varphi)\Big) <+\infty.
\end{split}
\end{align}
Therefore, $v \in U_j^{(0)}$ is well-defined and for every $k\geq 0$ we have
\begin{align}\label{eqn;qedifference}
\begin{split}
 \|Q(k)v - u_{k}\|& = \Big\|\sum_{l\geq k} Q|_{U_j}(k,l+1)^{-1}\Delta u_{l+1}\Big\| \\
  &\leq C_{a,\tau}  p_a(\varphi)\sum_{l \geq k} \|Q|_{U_j}(k,l+1)^{-1}\| \|Z(l+1)\|\langle \|Q(l)\|\rangle^{a}\| \|Q(l+1)\|^\tau\\
&=  C_{a,\tau} K^{a,j,\tau}_k p_a(\varphi).
\end{split}
 \end{align}
To obtain the bound of norm of $e_k \in E_k$, we apply \eqref{def;ehk}, \eqref{ukek} and \eqref{u0e0},
\begin{align*}\label{eqn;hk}
\begin{split}
\|&e_k\| \leq \sum_{0\leq l\leq k}\|Q(l,k)\Delta e_{l}\|\leq\sum_{0\leq l\leq k}\|Q|_{E_j}(l,k)\|
\|\Delta e_{l}\|\\
&\leq C_{a,\tau}\Big( p_a(\varphi)\sum_{0< l\leq k} \|Q|_{E_j}(l,k)\| \|Z(l)\|\langle \|Q(l-1)\|\rangle^{a}\| \|Q(l)\|^\tau+ \|Q|_{E_j}(k)\|\frac{\|{\varphi}\|_{L^1(I^{(0)})}}{|I^{(0)}|} \Big)\\
& \leq O \Big(p_a(\varphi)C_{k}^{a,j,\tau} + \|Q|_{E_j}(k)\|\frac{\|{\varphi}\|_{L^1(I^{(0)})}}{|I^{(0)}|} \Big).
\end{split}
\end{align*}
Combining with \eqref{eqn;qedifference}, we conclude
\[ \norm{Q(k)v - v_{k}} \leq O \Big((K^{a,j,\tau}_k+C^{a,j,\tau}_k) p_a(\varphi)+ \|Q|_{E_j}(k)\|\frac{\|{\varphi}\|_{L^1(I^{(0)})}}{|I^{(0)}|} \Big).\]
Since $\mathcal{M}^{(k)}(S(k)(v))=Q(k)v$, this gives
\begin{align}\label{eqn;correction1}
\big\|\mathcal{M}^{(k)}(S(k)(\varphi-v))\big\|&\leq O \Big((K^{a,j,\tau}_k+C^{a,j,\tau}_k) p_a(\varphi)  + \|Q|_{E_j}(k)\| \frac{\norm{\varphi}_{L^1(I^{(0)})}}{|I^{(0)}|}\Big).
\end{align}

Let us consider the linear operator $\mathfrak{h}_j: \pag(\sqcup_{\alpha \in \mathcal{A}} I_\alpha) \rightarrow U_j^{(0)} \subset H(\pi)$, called the \emph{correction operator} given by
\begin{equation*}
\mathfrak{h}_j(\varphi) = v = \sum_{l \geq 0} Q(0,l)^{-1}\circ P^{(l)}_{U_j}\circ\big(\mathcal{M}^{(l)}\circ S(l)-Z(l)\circ\mathcal{M}^{(l-1)}\circ S(l-1)\big)(\varphi),
\end{equation*}
where $\mathcal{M}^{(-1)}=0$. Since $\mathfrak{h}_j(\varphi)=\lim_{l\to\infty}Q(0,l)^{-1}\circ P^{(l)}_{U_j}\circ\mathcal{M}^{(l)}\circ S(l)(\varphi)$, by Remark~\ref{rmk:acc},
the operator $\mathfrak{h}_j$ is independent of the choice of $\tau$ and its FDC-acceleration.
By \eqref{eqn;inftyv}, we have
\[\norm{\mathfrak{h}_j(\varphi)} \leq C'_{a,\tau}\Big(\frac{2\kappa}{|I^{(0)}|} \norm{\varphi}_{L^1(I^{(0)})}+K^{a,j,\tau}_0p_a(\varphi)\Big) \leq  C'_{a,\tau}\Big(\frac{2\kappa}{|I^{(0)}|} + K^{a,j,\tau}_0\Big)\norm{\varphi}_a,\]
thus, the operator $\mathfrak{h}_j$ is bounded.
Hence, if $\mathfrak{h}_j(\varphi)=0$, then, by \eqref{eqn;correction1}, we have \eqref{eqn;maincorrection}.
\end{proof}

The construction of the correction operator is similar to that of the correction operator for logarithmic singularity presented by the first author and C.~Ulcigrai in \cite{Fr-Ul,Fr-Ul2}.
We apply also some elements of  Bufetov's approach from \cite[\S2.6]{Bu}.

\section{Deviation of Birkhoff sums and integrals}\label{sec;deviation}
This section contains main results on the deviation of Birkhoff integrals for special flows built
over IETs satisfying \ref{FDC} and under roof functions in $\pag$. Because of the existence of such  special representations for almost every locally Hamiltonian flow (see \S\ref{sec;specialflow}),
the contents of this section will be directly used to prove the main Theorem~\ref{theorem;main} in \S\ref{sec:pmt}.
Some of the results in this section are creative extensions of ideas introduced in \cite{Fr-Ul2} for roof functions with logarithmic singularities.
However, the existence of polynomial singularities significantly complicates the arguments.

\subsection{Reduction of Birkhoff integrals to Birkhoff
sums}\label{sec:decompflow}
Given an IET $T:I\to I$ and an integrable function
$g:I\to\R_{>0}\cup\{+\infty\}$ with
$\underline{g}=\inf_{x\in I}g(x)>0$, we consider the special flow $(T^g_t)_{t\in\R}$ acting on
$I^g=\{(x,r)\in I\times \R:0\leq r<g(x)\}$.
For every integrable function $f:I^g\to\R$, let
\[\varphi_f:I\to\R\quad \text{be given by}\quad \varphi_f(x)=\int_0^{g(x)}f(x,r)\, dr\text{ for }x\in I.\]
For every $u\geq \underline{g}$ let $I_u$ be a subset of $I$ such that
$g(x)\leq u$ for every $x\in I_u$. Moreover, for every $s\geq 0$ let
\[A_u^s:=\{(x,r)\in I^g:x\in I_u\}\setminus\{T^g_{-t}(x,0):x\in
I\setminus I_u,\,0\leq t\leq s\}\subset I^g.\]

For every $(x,r)\in I^g$ and $s>0$, let $n(x,r,s)\geq 0$ be the number of crossing the interval $I$ by the orbit segment $\{T^g_t(x,r):t\in [0,s]\}$.
Then $0\leq n(x,r,s)\leq s/\underline{g}+1$.

\medskip

We recall the following elementary Lemma that relates the Birkhoff integrals of $f$
for the flow $(T_t^g)_{t\in \R}$ with the Birkhoff sums of
$\varphi_f$ for the IET $T$. This is a slightly modified version of Lemma~7.1 in \cite{Fr-Ul2} using the quantity $\chi(u)=\sup\{\varphi_{|f|}(x):x\in I_u\}$.

\begin{lemma}[Lemma~7.1 in \cite{Fr-Ul2}]\label{lem:int0s}
For every integrable map $f:I^g\to\R$,   $s>0$ and $u\geq \underline{g}$ if $(x,r)\in A^s_u$ then
\begin{equation}\label{eq:ia}
T^ix\in I_u\quad\text{for all}\quad 0\leq i\leq n(x,r,s),
\end{equation}
and
\begin{equation}\label{ineq:intchi}
\Big|\int_0^sf(T^g_t(x,r))\,dt\Big|
\leq |\varphi_f^{(n(x,r,s))}(x)|+2\chi(u).
\end{equation}
\end{lemma}


We next estimate $\varphi_f^{(n)}(x)$ using the decomposition
into special Birkhoff sums introduced by Zorich in \cite{Zo}.

\smallskip

Let $T:I\to I$ be an arbitrary IET satisfying Keane's condition.
For every $x\in I$ and $n\geq 0$ set
\[m(x,n)=m(x,n,T):=\max\{l\geq 0:\#\{0\leq k\leq n:T^kx\in
I^{(l)}\}\geq 2\}.\]
\begin{proposition}[see \cite{Zo} or \cite{ViB}] \label{relmn}
For every $x\in I$ and $n>0$ we have
\[\min_{\alpha\in\mathcal{A}}Q_{\alpha}(m)\leq n\leq
d\max_{\alpha\in\mathcal{A}}
Q_{\alpha}(m+1)=d\|Q(m+1)\|, \text{ where }m=m(x,n).\]
\end{proposition}
\noindent Since the sequence
$\big(\min_{\alpha\in\mathcal{A}}Q_{\alpha}(m)\big)_{m\geq 0}$
 increases to the infinity
\[m(n)=m(n,T):=\max\{m(x,n):x\in I\}\quad \text{is well-defined.}\]

\begin{lemma}\label{lem:phibounded}
Assume that $T$ satisfies \ref{FDC} and $\varphi:I \rightarrow \R$ is a bounded map such that
for some $\tau>0$ and the corresponding FDC-acceleration we have $\|S(k)\varphi\|_{\sup}=O(e^{\lambda k})$ for some $\lambda\geq 0$. Then
$\|\varphi^{(n)}\|=O(n^{(\lambda+\tau)/\lambda_1})$.
\end{lemma}

\begin{proof}
By Zorich's standard argument for  the decomposition
into special Birkhoff sums (cf. \cite{Ma-Mo-Yo}), for every $n>0$,
\[\|\varphi^{(n)}\|_{\sup} \leq 2\sum_{l=1}^{m(n)}\norm{Z(l+1)}\norm{S(l)\varphi}_{\sup}.  \]
By Proposition \ref{relmn}, if $T$ satisfies \ref{FDC}, by \eqref{def;sdc3} and \eqref{eq:minmax}, then for every $\tau>0$,
\[e^{\lambda_1 m(n)} \leq O\big(\min_{\alpha \in \mathcal{A}}Q_\alpha(m(n)) \big) \leq O(n).\]
Since $\norm{S(k)\varphi}_{\sup} = O(e^{\lambda k})$ and $\norm{Z(l+1)}=O(e^{\tau l})$,  we have
\[|\varphi^{(n)}(x)| = O\Big(\sum_{l=1}^{m(n)}e^{(\lambda + \tau)l}\Big) = O(e^{(\lambda + \tau)m(n)})=O(n^{(\lambda + \tau)/\lambda_1}).  \]
\end{proof}

If $\varphi_f$ is not bounded, we need a more subtle estimate coming from \cite{Fr-Ul2}.
\begin{proposition}[Proposition 7.3 in \cite{Fr-Ul2}]\label{twdevpre1}
For every integrable map $f:I^g\to\R$, $s> 0$ and $u\geq \underline{g}$
if $(x,r)\in A_u^s$ then
\begin{equation}\label{eq:twdevpre1}
|\varphi_f^{(n(x,r,s))}(x)|\leq
2\sum_{k=0}^{m(n(x,r,s))}\|Z(k+1)\|\|S(k)\varphi_f\|_{L^{\infty}(
I^{(k)}(u))},
\end{equation}
with
\[I^{(k)}(u):=\bigcup_{\alpha\in\mathcal A}\{x\in I^{(k)}_\alpha:
\forall_{0\leq j<Q_\alpha(k)}T^jx\in I_u\}.\]
\end{proposition}

Suppose that $0<b\leq a<1$, $g\in {\pag}(\sqcup_{\alpha\in \mathcal{A}}
I_{\alpha})$ and $f:I^g\to\R$ is an integrable map such that $\varphi_f\in {\pbg}(\sqcup_{\alpha\in \mathcal{A}}
I_{\alpha})$ and $\varphi_{|f|}\in {\wpb}(\sqcup_{\alpha\in \mathcal{A}}
I_{\alpha})$.
As $g\in {\wpa}(\sqcup_{\alpha\in \mathcal{A}} I_{\alpha})$ and $\varphi_{|f|}\in {\wpb}(\sqcup_{\alpha\in \mathcal{A}}
I_{\alpha})$, there exist positive constants $c,C> 0$ such that
\[g(x) \leq u \text{ for all }  x\in I_u:=\bigcup_{\alpha \in \mathcal{A}} [l_\alpha + cu^{-1/a}, r_\alpha - cu^{-1/a}] \]
and $\chi(u)\leq Cu^{b/a}$.

\begin{lemma}\label{lem:birkonA}
Suppose that $T$ satisfies \ref{FDC} and for every small enough $\tau >0$,
\begin{equation}\label{neq:sko}
\|\mathcal{M}^{(k)}(S(k)\varphi_f)\|=O(e^{b\lambda_1(1+\tau) k}).
\end{equation}
Then for every $\tau>0$ there exists a constant $C=C_\tau\geq 1$ such that for every $u\geq s^a$ and $(x,r)\in A^s_{u}$ we have
\begin{equation}\label{neq:birA}
\left|\int_{0}^sf(T^g_t(x,r))dt \right| \leq Cu^{(1+\tau)\frac{b}{a}}.
\end{equation}
\end{lemma}

\begin{proof}
Let $u\geq s^a$. Since  $[0, c{u}^{-1/a}] \subset I \setminus I_{u}$, if  $|I^{(k)}| \leq c{u}^{-1/a}$,
then $I^{(k)}(u) = \emptyset$. By \eqref{def;sdc3} and \eqref{def;sdc4}, we have
\[e^{\lambda_1k}|I^{(k)}|= O(\|Q(k)\||I^{(k)}|)=O(\sum_{\alpha\in\mathcal
A}Q_\alpha(k)|I_\alpha^{(k)}|)=O(1).\]
It follows that
\[I^{(k)}(u) \neq \emptyset \Rightarrow  |I^{(k)}|>cu^{-1/a}\Rightarrow  k \leq \frac{1}{a\lambda_1}\log(C'u).\]
Moreover, if $x \in I^{(k)}(u) \cap I_\alpha^{(k)}$, then $x \in [ l^{(k)}_\alpha + cu^{-1/a}, r^{(k)}_\alpha - cu^{-1/a}]$.
In view of \eqref{eqn;upperboundvarphi}, \eqref{eqn;renormpaos}, \eqref{neq:sko}, \eqref{eq:Ikl} and \eqref{def;sdc3}, it follows that
for every $x \in I^{(k)}(u)$,
\begin{align*}
|(S(k)\varphi_f)(x)| & \leq \|\mathcal{M}^{(k)}(S(k)\varphi_f)\| +
p_b(S(k)\varphi_f)\Big(\frac{1}{b{(cu^{-1/a})}^b} + \frac{2^{b+2}}{b(1-b)|I^{(k)}|^{b}}\Big) \\
& \leq O(e^{b\lambda_1(1+\tau) k}) + O(p_b(\varphi_f))\Big(\frac{u^{b/a}}{b{c}^b} + \frac{2^{b+2}}{b(1-b)}O\big(e^{b\lambda_1(1+\tau)k}\big)\Big)\\
&\leq O\big(u^{(1+\tau)\frac{b}{a}}\big)+O\big(u^{\frac{b}{a}}\big)=O\big(u^{(1+\tau)\frac{b}{a}}\big).
\end{align*}
Therefore, by \eqref{ineq:intchi} and \eqref{eq:twdevpre1}, for every $(x,r) \in A^s_{u}$ we have
\begin{align*}
\left|\int_{0}^sf(T^g_t(x,r))dt\right|&\leq 2\chi(u)+2\sum_{k\geq 0,\,I^{(k)}(u)\neq\emptyset}\|Z(k+1)\|\|S(k)\varphi_f\|_{L^{\infty}(
I^{(k)}(u))}\\
&=O(u^{\frac{b}{a}})+O\Big(u^{(1+\tau)\frac{b}{a}}\sum_{0\leq k\leq
\frac{\log(C'u)}{a\lambda_1}}e^{\tau k}\Big)\\
&=O(u^{\frac{b}{a}})+O\Big(u^{(1+\tau)\frac{b}{a}}u^{\frac{\tau}{\lambda_1}\frac{1}{a}}\Big)
=O\Big(u^{(1+\tau(1+\frac{1}{b\lambda_1}))\frac{b}{a}}\Big).
\end{align*}
\end{proof}

\subsection{Estimates of ergodic integrals}\label{sec;DPS}
In this section, we show  $L^1$ and a.e.\ bounds  of the Birkhoff integrals for special flows. These bounds are related to the growth of the sequence $\|\mathcal{M}^{(k)}(S(k)\varphi)\|$.
It is the first step toward complete deviation formula in \S \ref{sec:pmt}. The main tools used in this section come from Lemma~\ref{lem:birkonA}.

\medskip

Recall that $0<b\leq a<1$, $g\in {\pag}(\sqcup_{\alpha\in \mathcal{A}}
I_{\alpha})$ and $f:I^g\to\R$ is an integrable map such that $\varphi_f\in {\pbg}(\sqcup_{\alpha\in \mathcal{A}}
I_{\alpha})$ and $\varphi_{|f|}\in {\wpb}(\sqcup_{\alpha\in \mathcal{A}}
I_{\alpha})$. Moreover, $c> 0$ is chosen  so that
\[  x\in I_u=\bigcup_{\alpha \in \mathcal{A}} [l_\alpha + cu^{-1/a}, r_\alpha - cu^{-1/a}]\Longrightarrow g(x) \leq u. \]
For every $s>0$ let
\[B_{s^a}:=\{(x,r)\in I^g:x\in I\setminus I_{s^a},0\leq r< g(x)-s\}.\]

\begin{lemma}
There exists $C\geq 1$ such that for all $s^a\leq u_1<u_2$ we have
\begin{equation}\label{neq:au1u2}
Leb\big((A_{u_2}^s\setminus A_{u_1}^s)\setminus B_{s^a}\big)\leq Cs(u_1^{-1/a}-u_2^{-1/a}).
\end{equation}
\end{lemma}

\begin{proof}
First note that
\begin{align*}
(A_{u_2}^s\setminus A_{u_1}^s)\setminus B_{s^a}&\subset \bigcup_{0\leq t\leq s}T^g_{-t}\big((I_{u_2}\setminus I_{u_1})\times\{0\}\big)\\
&\quad\cup\big\{(x,r)\in I^g:x\in I_{u_2}\setminus I_{u_1},g(x)-s\leq r<g(x)\big\}.
\end{align*}
Therefore
\[Leb\big((A_{u_2}^s\setminus A_{u_1}^s)\setminus B_{s^a}\big)\leq 2s\, Leb(I_{u_2}\setminus I_{u_1}).\]
Since $I_{u_2}\setminus I_{u_1}$ is the union of $2d$ disjoint intervals of the same length $c(u_1^{-1/a}-u_2^{-1/a})$,
we have
\[2s\,Leb(I_{u_2}\setminus I_{u_1})=4dcs(u_1^{-1/a}-u_2^{-1/a}).\]
This gives \eqref{neq:au1u2}.
\end{proof}

\begin{lemma}
There exists $C\geq 1$ such that
\begin{equation}\label{neq:birkBs}
\Big\|\int_0^sf\circ T^g_t\,dt\Big\|_{L^1(B_{s^a})}\leq C s^{b}\text{ for all }s\geq 1.
\end{equation}
\end{lemma}

\begin{proof}
First note that for every $x\in I\setminus I_{s^a}$ we have
\begin{align*}
\int_{0}^{g(x)-s}&\Big|\int_0^s f(T^g_t(x,r))\,dt\Big|\,dr\leq \int_{0}^{g(x)-s}\int_r^{r+s} |f(x,t)|\,dt\,dr\\
&=\int_{0}^{g(x)}\Big(\int_{\max\{0,t-s\}}^{\min\{t,g(x)-s\}} |f(x,t)|\,dr\Big)\,dt\\
&=\int_{0}^{g(x)}(\min\{t,g(x)-s\}-\max\{0,t-s\})|f(x,t)|\,dt.
\end{align*}
As $\min\{t,g(x)-s\}-\max\{0,t-s\}\leq s$, we have
\[\int_{0}^{g(x)-s}\Big|\int_0^s f(T^g_t(x,r))\,dt\Big|\,dr\leq s\int_{0}^{g(x)}|f(x,t)|\,dt=s\varphi_{|f|}(x).\]
It follows that
\[\Big\|\int_0^sf\circ T^g_t\,dt\Big\|_{L^1(B_{s^a})}=\int_{I\setminus I_{s^{a}}}\int_{0}^{g(x)-s}\Big|\int_0^s f(T^g_t(x,r))\,dt\Big|\,dr\,dx
\leq s\int_{I\setminus I_{s^a}}\varphi_{|f|}(x)\,dx.\]
As $\varphi_{|f|}\in \wpb$, there exists $C>0$ such that for every $\alpha\in\mathcal{A}$ and every $x\in I_{\alpha}$ we have
\begin{equation*}
\varphi_{|f|}(x) \leq \frac{C}{\min\{x-l_\alpha, r_\alpha-x\}^{b}}.
\end{equation*}
Since $I\setminus I_{s^a}$ is the union of $2d$ intervals of length $cs^{-1}$ with ends at $l_\alpha$, $r_\alpha$ for $\alpha\in \mathcal{A}$, we have
\begin{align*}
\int_{I\setminus I_{s^a}}\varphi_{|f|}(x)\,dx&\leq 2dC\int_0^{cs^{-1}}x^{-b}\,dx=\frac{2dCc^{1-b}}{1-b}s^{b-1}.
\end{align*}
Therefore
\[\Big\|\int_0^sf\circ T^g_t\,dt\Big\|_{L^1(B_{s^a})}\leq \frac{2dCc^{1-b}}{1-b}s^{b}.\]
\end{proof}

Now we obtain the $L^1$-bound for special flows under the condition for the growth rate of the renormalized and projected  cocycles $\mathcal{M}^{(k)}(S(k)\varphi)$.
\begin{theorem}\label{thm:expbL1}
Suppose that $T$ satisfies \ref{FDC}, $0<b\leq a<1$, $g\in {\pag}(\sqcup_{\alpha\in \mathcal{A}}
I_{\alpha})$, and $f:I^g\to\R$ is an integrable map such that $\varphi_f\in {\pbg}(\sqcup_{\alpha\in \mathcal{A}}
I_{\alpha})$ and $\varphi_{|f|}\in {\wpb}(\sqcup_{\alpha\in \mathcal{A}}
I_{\alpha})$.
Assume that for every $\tau >0$,
\[
\|\mathcal{M}^{(k)}(S(k)\varphi_f)\|=O(e^{b\lambda_1(1+\tau) k}).
\]
Then for every $\tau>0$ we have
\begin{equation}\label{eqn;devL1bound}
\Big\|\int_{0}^sf\circ T^g_tdt\Big\|_{L^1(I^g)}=O\big(s^{(1+\tau)b}).
\end{equation}
\end{theorem}

\begin{proof}
Let $(u_n)_{n\geq 0}$ be any strictly increasing sequence diverging to $+\infty$ and such that $u_0=s^a$. Then
\begin{align*}
\Big\|\int_{0}^sf\circ T^g_tdt\Big\|_{L^1(I^g)}&\leq \Big\|\int_{0}^sf\circ T^g_tdt\Big\|_{L^\infty(A^s_{u_0})}Leb(A^s_{u_0})
+\Big\|\int_{0}^sf\circ T^g_tdt\Big\|_{L^1(B_{u_0})}\\
&+\sum_{j=1}^{\infty}\Big\|\int_{0}^sf\circ T^g_tdt\Big\|_{L^\infty((A^s_{u_j}\setminus A^s_{u_{j-1}})\setminus B_{u_0})}
Leb\big((A^s_{u_j}\setminus A^s_{u_{j-1}})\setminus B_{u_0}\big).
\end{align*}
In view of \eqref{neq:birA}, \eqref{neq:birkBs} and \eqref{neq:au1u2}, we have
\begin{gather*}
\Big\|\int_{0}^sf\circ T^g_tdt\Big\|_{L^\infty(A^s_{u_0})}\leq Cu_0^{(1+\tau)\frac{b}{a}}=Cs^{(1+\tau)b};\\
\Big\|\int_{0}^sf\circ T^g_tdt\Big\|_{L^\infty((A^s_{u_j}\setminus A^s_{u_{j-1}})\setminus B_{u_0})}\leq C u_j^{(1+\tau)\frac{b}{a}};\\
\Big\|\int_0^sf\circ T^g_t\,dt\Big\|_{L^1(B_{s^a})}\leq C s^{b};\\
Leb\big((A^s_{u_j}\setminus A^s_{u_{j-1}})\setminus B_{u_0})\big)\leq Cs(u_{j-1}^{-1/a}-u_{j}^{-1/a}).
\end{gather*}
Hence
\begin{align*}
\Big\|\int_{0}^sf\circ T^g_tdt\Big\|_{L^1(I^g)}&\leq Cs^{(1+\tau)b}+Cs^{b}+sC^2\sum_{j=1}^\infty u_j^{(1+\tau)\frac{b}{a}}(u_{j-1}^{-1/a}-u_{j}^{-1/a}).
\end{align*}
Let us consider a strictly decreasing sequence of positive numbers $(v_n)_{n\geq 0}$  given by $v_n:=u_{n}^{-1/a}$.
Then $v_0=s^{-1}$ and $v_n\to 0$ as $n\to+\infty$ and
\begin{align*}
\Big\|\int_{0}^sf\circ T^g_tdt\Big\|_{L^1(I^g)}&\leq 2Cs^{(1+\tau)b}+sC^2\sum_{j=1}^\infty v_j^{-(1+\tau)b}(v_{j-1}-v_{j}).
\end{align*}
Passing through all possible sequences $(u_n)_{n\geq 0}$ and taking the infimum of values  standing on the right hand side, this gives
\begin{align*}
\Big\|\int_{0}^sf\circ T^g_tdt\Big\|_{L^1(I^g)}&\leq 2Cs^{(1+\tau)b}+sC^2\int_{0}^{s^{-1}}x^{-(1+\tau)b}\,dx\\
&=2Cs^{(1+\tau)b}+\frac{sC^2}{1-(1+\tau)b}s^{(1+\tau)b-1}=O\big(s^{(1+\tau)b}\big).
\end{align*}
\end{proof}

The following measure estimations are key ingredients for Borel-Cantelli argument applied to prove a.e.\ pointwise upper bound for Birkhoff integrals.
\begin{lemma}\label{lem:lebA}
There exists $C>0$ such that for all $1\leq s<s'$ and $\tau>0$ we have
\begin{gather}
\label{neq:Astau}
Leb\big((A^s_{s^{(1+\tau)a}})^c\big)\leq Cs^{-\tau}+Cs^{-(1+\tau)(1-a)};\\
\label{neq:Asdiff}
Leb\big((A^{s'}_{{s'}^{(1+\tau)a}})^c\setminus(A^s_{s^{(1+\tau)a}})^c\big)\leq C(s')^{-(1+\tau)}(s'-s).
\end{gather}
\end{lemma}

\begin{proof}
Recall that
\[(A^s_{s^{(1+\tau)a}})^c=\bigcup_{0\leq t\leq s}T^g_{-t}((I\setminus I_{s^{(1+\tau)a}})\times\{0\})\cup
\{(x,r)\in I^g:x\in I\setminus I_{s^{(1+\tau)a}}\},\]
where
\[I\setminus I_{s^{(1+\tau)a}}=\bigcup_{\alpha\in\mathcal{A}}[l_\alpha,l_\alpha+cs^{-(1+\tau)})\cup(r_\alpha-cs^{-(1+\tau)},r_\alpha].\]
As $I\setminus I_{{s'}^{(1+\tau)a}}\subset I\setminus I_{s^{(1+\tau)a}}$, it follows that
\[(A^{s'}_{{s'}^{(1+\tau)a}})^c\setminus(A^s_{s^{(1+\tau)a}})^c\subset \bigcup_{s< t\leq s'}T^g_{-t}((I\setminus I_{{s'}^{(1+\tau)a}})\times\{0\}).\]
Therefore
\[Leb\big((A^{s'}_{{s'}^{(1+\tau)a}})^c\setminus(A^s_{s^{(1+\tau)a}})^c\big)\leq(s'-s)Leb(I\setminus I_{{s'}^{(1+\tau)a}})=(s'-s)2dc(s')^{-(1+\tau)},\]
which gives \eqref{neq:Asdiff}.
Moreover, we have
\[Leb\big((A^s_{s^{(1+\tau)a}})^c\big)\leq sLeb(I\setminus I_{s^{(1+\tau)a}})+\int_{I\setminus I_{s^{(1+\tau)a}}}g(x)\,dx.\]
As $g\in\wpb$, there exists $C>0$ such that for every $\alpha\in\mathcal{A}$ and every $x\in I_{\alpha}$ we have
\begin{equation*}
g(x) \leq \frac{C}{\min\{x-l_\alpha, r_\alpha-x\}^{a}}.
\end{equation*}
Hence
\begin{align*}
\int_{I\setminus I_{s^{(1+\tau)a}}}g(x)\,dx\leq 2dC\int_0^{cs^{-(1+\tau)}}x^{-a}\,dx
&\leq \frac{2dCc^{1-a}}{1-a}s^{-(1+\tau)(1-a)}.
\end{align*}
As $s Leb(I\setminus I_{s^{(1+\tau)a}})=s2dcs^{-(1+\tau)}=2dcs^{-\tau}$, this gives \eqref{neq:Astau}.
\end{proof}

Now we prove a.e. pointwise upper bound of Birkhoff integrals for special flows under some restriction on the growth rate of renormalized and projected cocycles $\mathcal{M}^{(k)}(S(k)\varphi_f)$.

\begin{theorem}\label{thm;L1bound}
Suppose that $T$ satisfies \ref{FDC}, $0<b\leq a<1$, $g\in {\pag}(\sqcup_{\alpha\in \mathcal{A}}
I_{\alpha})$ and $f:I^g\to\R$ is an integrable map such that $\varphi_f\in {\pbg}(\sqcup_{\alpha\in \mathcal{A}}
I_{\alpha})$ and $\varphi_{|f|}\in {\wpb}(\sqcup_{\alpha\in \mathcal{A}}
I_{\alpha})$.
Assume that for every $\tau >0$,
\[
\|\mathcal{M}^{(k)}(S(k)\varphi_f)\|=O(e^{b\lambda_1(1+\tau) k}).
\]
Then for a.e.\ $(x,r)\in I^g$ we have
\begin{equation}\label{eqn;devaebound}
\limsup_{s\to+\infty}\frac{\log |\int_{0}^sf\circ T^g_t(x,r)\,dt|}{\log s} \leq b.
\end{equation}
\end{theorem}

\begin{proof}
Fix $\tau>0$ small enough. In view of Lemma~\ref{lem:birkonA}, there exists $C_\tau\geq 1$ such that for every $m\in\N$ we have
\begin{equation}\label{neq:birlogAm}
\frac{\log|\int_{0}^mf(T^g_t(x,r))dt|}{\log m}\leq (1+\tau)^2b+\frac{C_\tau}{\log m}\text{ for every }(x,r)\in A^m_{m^{(1+\tau)a}}.
\end{equation}
Moreover, in view of Lemma~\ref{lem:lebA}, there exists $C>0$ such that for every $m\in\N$ we have
\begin{align*}
Leb\big(\bigcup_{n\geq m}(A^n_{n^{(1+\tau)a}})^c\big)&=Leb\big((A^m_{m^{(1+\tau)a}})^c\big)+
\sum_{n\geq m}Leb\big((A^{n+1}_{(n+1)^{(1+\tau)a}})^c\setminus (A^n_{n^{(1+\tau)a}})^c\big)\\
&\leq Cm^{-\tau}+Cm^{-(1+\tau)(1-a)}+\sum_{n\geq m} C(n+1)^{-(1+\tau)}((n+1)-n)\\
&\leq Cm^{-\tau}+Cm^{-(1+\tau)(1-a)}+\frac{C}{\tau}m^{-\tau}.
\end{align*}
Hence
\[Leb\big(\bigcup_{m\geq 1}\bigcap_{n\geq m}A^n_{n^{(1+\tau)a}}\big)=1.\]
In view of \eqref{neq:birlogAm}, it follows that for a.e.\ $(x,r)\in I^g$ we have
\[\limsup_{m\to+\infty}\frac{\log|\int_{0}^mf(T^g_t(x,r))dt|}{\log m}\leq (1+\tau)^2b.\]
As we can choose $\tau>0$ arbitrary small,  this gives
\[\limsup_{m\to+\infty}\frac{\log|\int_{0}^mf(T^g_t(x,r))dt|}{\log m}\leq b\text{ for a.e. }(x,r)\in I^g.\]
Hence
\[\limsup_{s\to+\infty}\frac{\log|\int_{0}^sf(T^g_t(x,r))dt|}{\log s}
\leq \limsup_{s\to+\infty}\frac{\log\big(|\int_{0}^{[s]}f(T^g_t(x,r))dt|+\|f\|_{L^\infty(I^g)}\big)}{\log [s]}
\leq b\]
for a.e.\ $(x,r)\in I^g$.
\end{proof}

The following corollary will be essential in proving a $L^1$-lower bound of Birkhoff integrals in the next section.
\begin{corollary}\label{cor:aeexp}
Suppose $T$ satisfies \ref{FDC},  $0<b<1$ and  $\varphi\in {\pbg}(\sqcup_{\alpha\in \mathcal{A}}
I_{\alpha})$ such that
\[
\|\mathcal{M}^{(k)}(S(k)\varphi)\|=O(e^{b\lambda_1(1+\tau) k})\text{ for every }\tau >0.
\]
Then for a.e.\ $x\in I$ we have
\begin{equation}\label{eq:aeexp}
\limsup_{n\to+\infty}\frac{\log |\varphi^{(n)}(x)|}{\log n} \leq b.
\end{equation}
In particular, for every $\varphi\in {\pbg}(\sqcup_{\alpha\in \mathcal{A}}
I_{\alpha})$ for a.e.\ $x\in I$ we have
\begin{equation}\label{eq:aeexp1}
\limsup_{n\to+\infty}\frac{\log |\varphi^{(n)}(x)-n\mu(\varphi)|}{\log n} \leq\max\left\{b,\frac{\lambda_2}{\lambda_1}\right\},
\end{equation}
where $\mu(\varphi)=\int_I\varphi(y)\,dy$.
\end{corollary}

\begin{proof}
Let us consider the constant roof function $g:I\to\R$, $g\equiv 1$. Choose an integrable map $f:I^g\to\R$ such that $\varphi_f=\varphi$ and $\varphi_{|f|}=|\varphi_{f}|\in {\pbg}(\sqcup_{\alpha\in \mathcal{A}}
I_{\alpha})$ (we can set $f(x,r)=\varphi(x)$ for $(x,r)\in I^g$). Since $g\in {\pbg}(\sqcup_{\alpha\in \mathcal{A}} I_{\alpha})$, by Theorem~\ref{thm;L1bound} applied to $s=g^{(n)}(x)=n$, we have
 \[\limsup_{n\to+\infty}\frac{\log |\varphi^{(n)}(x)|}{\log n}=\limsup_{n\to+\infty}\frac{\log \big|\int_{0}^{g^{(n)}(x)}f\circ T^g_t(x,r)\,dt\big|}{\log g^{(n)}(x)} \leq b.
\]

Denote by $\Gamma_0\subset\Gamma$ the space of zero mean piecewise constant maps, i.e.\ $h\in\Gamma_0$ if $\langle h,\lambda\rangle=\sum_{i=1}^dh_i\lambda_i=0$.
As $h_1\not\in\Gamma_0$, $h_2,\ldots,h_g\in \Gamma_0$, we have
\begin{equation}\label{eq:gamma0}
\lim_{k\to\infty}\frac{\log\|Q(k)h\|}{k}\leq \lambda_2\text{  for every }h\in \Gamma_0.
\end{equation}
In view of Corollary~\ref{cor;correction}, there exists $\sum_{1\leq i<j}\beta_ih_i\in U_j$ such that
\begin{equation}\label{eq:zeromeanexp}
\Big\|\mathcal{M}^{(k)}\Big(S(k)\Big(\varphi-\sum_{1\leq i<j}\beta_ih_i)\Big)\Big)\Big\|=O(e^{b\lambda_1(1+\tau) k})\text{ for every }\tau >0.
\end{equation}
By  \eqref{eq:aeexp}, for every $\tau>0$ we have $\big(\varphi-\sum_{1\leq i<j}\beta_ih_i\big)^{(n)}(x)=O(n^{(b+\tau)})$ for a.e.\ $x\in I$.
Taking $\tau>0$ such that $b+\tau<1$, by the ergodicity of $T$, we obtain
\[\int_I\Big(\varphi-\sum_{1\leq i<j}\beta_ih_i\Big)(x)\,dx=0.\]
It follows that $\mu(\varphi)-\sum_{1\leq i<j}\beta_ih_i\in\Gamma_0$. In view of \eqref{eq:gamma0}, this gives
\[\lim_{k\to\infty}\frac{\log\|Q(k)(\mu(\varphi)-\sum_{1\leq i<j}\beta_ih_i)\|}{k}\leq \lambda_2.\]
By \eqref{eq:zeromeanexp}, it follows that for every $\tau >0$,
\[
\|\mathcal{M}^{(k)}(S(k)(\varphi-\mu(\varphi)))\|= \|S(k)(\varphi-\mu(\varphi))\|=O(e^{(\max\{\lambda_2,b\lambda_1\}+\tau) k}).
\]
Applying again \eqref{eq:aeexp}, we get  \eqref{eq:aeexp1}.
\end{proof}

\subsection{Pure power deviation}\label{sec:puredeviation}
In this section we relate the growth  of the accelerated KZ cocycle with the deviation of special flows for functions associated with homology elements.
In particular, it is mostly devoted to prove the lower bound of Birkhoff integrals in a.e.\ and $L^1$-norm.
In the following proposition, we reprove Bufetov's result about pure deviations a.e. Here the $L^1$-estimates account for the novelty of this result.
\begin{proposition}\label{thm:specflownonzero}
Suppose that the IET $T:I\to I$ satisfies \ref{FDC}.
Assume that $0<a<1$, $g\in \pag(\sqcup_{\alpha\in \mathcal{A}}
I_{\alpha})$ and $f:I^g\to \R$ is a bounded function such that there exists $K>0$ for which $f(x,r)=0$ for $r\geq K$ and $\varphi_f=
h=(h_\alpha)_{\alpha\in\mathcal{A}}\in H(\pi)$.
Suppose that there exists $\lambda>0$ such that
\[\lim_{k\to+\infty}\frac{\log\|Q(k)h\|}{k}=\lambda.\]
Then
\begin{align}\label{eq:Birkh1}
&\limsup_{s\to+\infty}\frac{\log|\int_{0}^sf(T^g_t(x,r))\,dt|}{\log s}= \frac{\lambda}{\lambda_1}\text{ for a.e. }(x,r)\in I^g,\\
\label{eq:Birkh2}
&\limsup_{s\to+\infty}\frac{\log\|\int_{0}^sf\circ T^g_t\,dt\|_{L^1}}{\log s}= \frac{\lambda}{\lambda_1}.
\end{align}
\end{proposition}
\begin{proof}
First note that for every $\tau>0$,
\begin{equation}\label{eq:inttokoc}
\|\mathcal{M}^{(k)}(S(k)\varphi_f)\| = \|\mathcal{M}^{(k)}(S(k)h)\|\leq \|Q(k)h\|=O(e^{k(\lambda+\tau)}).
\end{equation}
As $\varphi_{|f|}$ is bounded, Theorem~\ref{thm:expbL1} and \ref{thm;L1bound} yield the inequalities $\leq$ in \eqref{eq:Birkh1} and \eqref{eq:Birkh2}.

To show the reverse inequalities, note that for every $n\geq 1$ and  $(x,r)\in I^g$ we have
\begin{equation}\label{eq:inttokocgn}
\int_0^{g^{(n)}(x)-r}f(T^g_t(x,r))\,dt=\int_r^{g^{(n)}(x)}f(T^g_t(x,0))\,dt=\varphi^{(n)}_f(x)-\int_0^{r}f(T^g_t(x,0))\,dt.
\end{equation}
Moreover, for every $\alpha\in\mathcal A$ we have
\begin{equation}\label{eq:koc}
\varphi^{(Q_\alpha(k))}_f(x)=(Q(k)h)_\alpha\text{ if }x\in\bigcup_{0\leq i<p_k}T^iI^{(k)}_\alpha.
\end{equation}
Indeed, if $x\in I^{(k)}_\alpha$ then
\[\varphi^{(Q_\alpha(k))}_f(x)=S(k)\varphi_f(x)=(Q(k)h)_\alpha.\]
Now suppose that $x=T^iy$ so that $y\in I^{(k)}_\alpha$ and $0\leq i<p_k$.
Then
\[\varphi^{(Q_\alpha(k))}_f(x)-\varphi^{(Q_\alpha(k))}_f(y)=\varphi^{(Q_\alpha(k))}_f(T^iy)-\varphi^{(Q_\alpha(k))}_f(y)
=\sum_{0\leq l<i}(\varphi_f(T^lT^{Q_\alpha(k)}y)-\varphi_f(T^ly)).\]
Since $y,T^{Q_\alpha(k)}y\in I_\alpha^{(k)}$ and, by \eqref{def:FDC-g}, $\{T^jI^{(k)}_\alpha:{0\leq j<p_k}\}$ is a Rokhlin tower of intervals,
for every $0\leq i<p_k$ the points  $T^iy$ and $T^iT^{Q_\alpha(k)}y$ belong to the same interval $I_\beta$.
As $\varphi_f=h$ is constant
on each interval $I_\beta$, $\beta\in\mathcal A$, it follows
that
\[\varphi^{(Q_\alpha(k))}_f(x)=\varphi^{(Q_\alpha(k))}_f(y)=(Q(k)h)_\alpha.\]

Choose $\alpha\in \mathcal A$ and a subsequence $(k_n)_{n\geq 1}$ such that
\begin{equation}\label{eq:choicealpha}
\lim_{n\to\infty}\frac{\log|(Q(k_n)h)_\alpha|}{k_n}=\lambda.
\end{equation}
In view of \eqref{def:FDC-g},
\begin{equation}\label{neq:lebrt}
Leb\Big(\bigcup_{0\leq i<p_{k_n}}T^iI^{(k_n)}_\alpha\Big)\geq \frac{\delta|I|}{\kappa}\text{ for all }n\geq 1.
\end{equation}
By the ergodicity of $T$, for a.e.\ $x\in I$, passing to a further subsequence, we have
\[x\in \bigcup_{0\leq i<p_{k_{l_n}}}T^iI^{(k_{l_n})}_\alpha\text{ for all }n\geq 1\]
and
\begin{equation}\label{eq:meang}
\lim_{n\to\infty}\frac{g^{(Q_\alpha(k_{l_n}))}(x)}{Q_\alpha(k_{l_n})}=\mu(g):=\int_Ig(y)\,dy.
\end{equation}
In view of \eqref{eq:koc}, this gives
\begin{equation}\label{eq:valphifkoc}
\varphi^{(Q_\alpha(k_{l_n}))}_f(x)=(Q(k_{l_n})h)_\alpha\text{ for every }n\geq 1.
\end{equation}

For every $r\geq 0$ with $(x,r)\in I^g$, let us consider a sequence $(\tau_n)_{n\geq 1}$ given by
\[\tau_n=g^{(Q_\alpha(k_{l_n}))}(x)-r.\]
Then, by \eqref{eq:inttokocgn} and \eqref{eq:valphifkoc}, we have
\begin{align*}
&\frac{\log\left|\int_0^{\tau_n}f(T^g_t(x,r))\,dt\right|}{\log \tau_n}=
\frac{\log\left|(Q(k_{l_n})h)_\alpha-\int_0^{r}f(T^g_t(x,0))\,dt\right|}{\log (g^{(Q_\alpha(k_{l_n}))}(x)-r)}\\
&=\frac{\log\left|(Q(k_{l_n})h)_\alpha-\int_0^{r}f(T^g_t(x,0))\,dt\right|}{k_{l_n}}
\frac{\log(Q_\alpha(k_{l_n}))}{\log (g^{(Q_\alpha(k_{l_n}))}(x)-r)}\frac{k_{l_n}}{\log(Q_\alpha(k_{l_n}))}.
\end{align*}
Moreover, by \eqref{eq:Qalphaexp}, we have
\[\lim_{n\to\infty}\frac{\log(Q_\alpha(k_{l_n}))}{k_{l_n}}=\lambda_1;\]
by \eqref{eq:meang}, we have
\[\lim_{n\to\infty}\frac{\log(Q_\alpha(k_{l_n}))}{\log (g^{(Q_\alpha(k_{l_n}))}(x)-r)}=1;\]
by \eqref{eq:choicealpha}, we have
\[\lim_{n\to\infty}\frac{\log\left|(Q(k_{l_n})h)_\alpha-\int_0^{r}f(T^g_t(x,0))\,dt\right|}{k_{l_n}}=\lambda.\]
It follows that
\[\lim_{n\to\infty}\frac{\log\left|\int_0^{\tau_n}f(T^g_t(x,r))\,dt\right|}{\log \tau_n}=\frac{\lambda}{\lambda_1},\]
which completes the proof of \eqref{eq:Birkh1}.
\medskip

We now turn to $L^1$-estimate of the lower bound. Let us consider a new sequence $(\tau_n)_{n\geq 1}$ given by $\tau_n=Q_\alpha(k_{n})\mu(g)$.
Take any $\max\{\lambda_2/\lambda_1,a\}<\zeta<1$ and $\tau>0$ such that $\zeta(\lambda/\lambda_1+2\tau)\lambda_1/\lambda<1$. Since
\[\limsup_{n\to+\infty}\frac{\log |g^{(n)}(x)-n\mu(g)|}{\log n} <\zeta
\]
(see Corollary~\ref{cor:aeexp}) and by the ergodic theorem, there exist $N\in\N$ and $J\subset I$ such that $Leb(I\setminus J)<|I|{\delta}/(2\kappa)$
and for every $n\geq N$ and $x\in J$ we have
\begin{equation}\label{eqn;zetabound}
|g^{(n)}(x)-n\mu(g)|<n^{\zeta}\text{ and }n\mu(g)/2<g^{(n)}(x)<2n\mu(g).
\end{equation}

Suppose that
\[(x,r)\in D_n:=\Big(J\cap\bigcup_{0\leq i<p_{k_n}}T^iI^{(k_n)}_\alpha\Big)\times [0,\underline{g}]\]
and $Q_\alpha(k_n)\geq 2N+1$. As $Leb(I\setminus J)<|I|{\delta}/(2\kappa)$, by \eqref{neq:lebrt}, we have $Leb(D_n)>\underline{g}|I|\delta/(2\kappa)=:\bar{\delta}>0$.
Moreover
\begin{align*}
&\left|\int_0^{\tau_n}f(T^g_t(x,r))\,dt-(Q(k_{n})h)_\alpha\right|\\
&
=\left|\int_{r}^{\tau_n+r}f(T^g_t(x,0))\,dt-\int_0^{g^{(Q_\alpha(k_{n}))}(x)}f(T^g_t(x,0))\,dt\right|\\
&\leq 2r\|f\|_{C^0}+\left|\int_{\tau_n}^{g^{(Q_\alpha(k_{n}))}(x)}f(T^g_t(x,0))\,dt\right|.
\end{align*}
Denote by $l_n(x)$ the unique natural number such that
\[g^{(l_n(x))}(x)\leq \tau_n<g^{(l_n(x)+1)}(x).\]
By assumption,
\[\left|\int_{\tau_n}^{g^{(l_{n}(x))}(x)}f(T^g_t(x,0))\,dt\right|\leq K\|f\|_{\sup}.\]
It follows that
\begin{equation*}
\left|\int_0^{\tau_n}f(T^g_t(x,r))\,dt-(Q(k_{n})h)_\alpha\right|\leq (2r+K)\|f\|_{C^0}+|\varphi_f^{(Q_\alpha(k_n))}(x)-\varphi_f^{(l_n(x))}(x)|.
\end{equation*}
Since $\|S(k)\varphi_f\|_{\sup}=\|Q(k)h\| = O(e^{(\lambda+\tau)k})$, by
Lemma~\ref{lem:phibounded}, there exists $C>0$ such that $\|\varphi_f^{(n)}\|_{sup}\leq Cn^{\lambda/\lambda_1+2\tau}$.
Therefore,
\begin{equation}\label{eqn;displacement}
\left|\int_0^{\tau_n}f(T^g_t(x,r))\,dt-(Q(k_{n})h)_\alpha\right|\leq (2r+K)\|f\|_{C^0}+C|Q_\alpha(k_n)-l_n(x)|^{\lambda/\lambda_1+2\tau}.
\end{equation}
As $g^{([Q_\alpha(k_n)/2])}(x)<Q_\alpha(k_n)\mu(g)<g^{(l_n(x)+1)}(x)$, we have $N\leq Q_\alpha(k_n)/2-1\leq l_n(x)$.
Hence by \eqref{eqn;zetabound},
\begin{align*}
-l_n(x)^{\zeta}&<g^{(l_n(x))}(x)-l_n(x)\mu(g)\leq Q_\alpha(k_{n})\mu(g)-l_n(x)\mu(g)\\
&<g^{(l_n(x)+1)}(x)-l_n(x)\mu(g)<(l_n(x)+1)^{\zeta}+\mu(g).
\end{align*}
It follows that there exists $C>0$ such that for every $(x,r)\in D_n$ (with $Q_\alpha(k_n)\geq 2N+1$) we have
\[|Q_\alpha(k_n)-l_n(x)|\leq C (Q_\alpha(k_n))^{\zeta}.\]
Then for every $(x,r)\in D_n $, by \eqref{eqn;displacement}, we have 
\begin{align*}
\left|\int_0^{\tau_n}f(T^g_t(x,r))\,dt \right|&\geq |(Q(k_{n})h)_\alpha|-C' (Q_\alpha(k_n))^{\zeta(\lambda/\lambda_1+2\tau)}.
\end{align*}
As $Leb(D_n)>\bar{\delta}$, it follows that
\[\frac{\log\|\int_0^{\tau_n}f\circ T^g_t\,dt\|_{L^1(I^g)}}{\log \tau_n}\geq \frac{\log\bar{\delta}\big(|(Q(k_{n})h)_\alpha|-C' (Q_\alpha(k_n))^{\zeta(\lambda/\lambda_1+2\tau)}\big)}{\log \mu(g)Q_\alpha(k_n)}.\]
By \eqref{eq:choicealpha} and \eqref{eq:Qalphaexp},
\[\lim_{n\to\infty}\frac{\log (Q_\alpha(k_n))^{\zeta(\lambda/\lambda_1+2\tau)}}{\log |(Q(k_{n})h)_\alpha|}=\zeta(\lambda/\lambda_1+2\tau)\frac{\lambda_1}{\lambda}<1.\]
Therefore,
\[\lim_{n\to\infty}\frac{\log\big(|(Q(k_{n})h)_\alpha|-C' (Q_\alpha(k_n))^{\zeta(\lambda/\lambda_1+2\tau)}\big)}{\log |(Q(k_{n})h)_\alpha|}=1.\]
Hence, by \eqref{eq:choicealpha} and \eqref{eq:Qalphaexp} again,
\[\liminf_{n\to\infty}\frac{\log\|\int_0^{\tau_n}f\circ T^g_t\,dt\|_{L^1(I^g)}}{\log \tau_n}\geq
\lim_{n\to\infty}\frac{\log\bar{\delta}|(Q(k_{n})h)_\alpha|}{\log \mu(g)Q_\alpha(k_n)}=\frac{\lambda}{\lambda_1}.\]
\end{proof}

We finish the section by checking sub-polynomial bounds for $\pog$.
\begin{proposition}[cf.\ Theorem 2 in \cite{Fr-Ul2}]\label{thm:specflowzeroexp}
Suppose that the IET $T:I\to I$ satisfies the \ref{FDC}.
Assume that $\varphi_f\in \pog(\sqcup_{\alpha\in \mathcal{A}}
I_{\alpha})$, $\varphi_{|f|}\in \operatorname{\widehat{P}_0}(\sqcup_{\alpha\in \mathcal{A}}
I_{\alpha})$, $g\in \pag(\sqcup_{\alpha\in \mathcal{A}}
I_{\alpha})$ for some $0\leq a<1$ and every $\tau >0$,
\[
\|\mathcal{M}^{(k)}(S(k)\varphi_f)\|=O(e^{\tau k}).
\]
Then
\begin{align}\label{eq:Birk0}
&\limsup_{s\to+\infty}\frac{\log|\int_{0}^sf(T^g_t(x,r))dt|}{\log s}\leq 0\text{ for a.e.\ $(x,r)\in I^g$};\\
\label{eq:BirkL1}
&\limsup_{s\to+\infty}\frac{\log\|\int_{0}^sf\circ T^g_tdt\|_{L^1(I^g)}}{\log s}\leq 0.
\end{align}
\end{proposition}
\begin{proof}
For every $0<\vep<1-a$ we apply Theorem \ref{thm:expbL1} and \ref{thm;L1bound}  to $b:=\vep$  and $a:=a+\vep$. Since $\varphi_f\in \operatorname{{P}_\vep G}(\sqcup_{\alpha\in \mathcal{A}}
I_{\alpha})$, $\varphi_{|f|}\in \operatorname{\widehat{P}_\vep}(\sqcup_{\alpha\in \mathcal{A}}
I_{\alpha})$, $g\in \operatorname{{P}_{a+\vep} G}(\sqcup_{\alpha\in \mathcal{A}}
I_{\alpha})$, we obtain the upper bounds by taking arbitrary small $\vep>0$.
\end{proof}

\begin{remark}\label{rem:dodeq}
If additionally $f\neq 0$ then, using arguments from the proof of Theorem~1.4 in \cite[\S7.2.4]{Fr-Ul2}, we have equalities in both  \eqref{eq:Birk0} and \eqref{eq:BirkL1}.
\end{remark}

\section{Local analysis around saddles}\label{sec:locHam}
In this section we mainly prove some local relations between regularity of the function $f:M\to\R$ around fixed
points and the types of singularity for the associated cocycle $\varphi_f:I\to\R$. This analysis is one of
 our main result of this paper that was not studied in the related literatures beforehand in such a comprehensive way.
 The main results of this chapter play a key role in proving Theorem~\ref{thm:ftophi} and its extension in \S\ref{sec:regularity}.
  In fact, we generalize the approach developed for simple saddles in \cite{Fr-Ul} (related to logarithmic singularity type) to multi-saddle type.

\medskip
Let $M'\subset M$ be a minimal component of a locally Hamiltonian flow $\psi_\R$ associated with a closed 1-form $\eta$.
Let $I\subset M'$ be its transversal curve equipped with a standard parametrization. Recall that a parametrization of curve $\gamma: [a,b] \to M$ is \emph{standard} if $\gamma: I \to M$ if $\eta(d\gamma ) = 1$.
In the standard coordinates, the first return map $T:I\to I$ is an IET.

For every saddle $\sigma\in \mathrm{Fix}(\psi_\R)$  of multiplicity $m=m_\sigma\geq 2$ let $(x,y)$  be a singular chart in a neighborhood $U_\sigma$ of $\sigma$. Then the corresponding local Hamiltonian is of the form $H(x,y)=\Im (x+iy)^m$. If the $\psi_\R$-invariant area form $\omega=V(x,y)dx\wedge dy$, then the corresponding local Hamiltonian equation in $U_\sigma$ is of the form
\[\frac{dx}{dt}=\frac{\frac{\partial H}{\partial y}(x,y)}{V(x,y)}=\frac{m\Re(x+iy)^{m-1}}{V(x,y)} \quad \text{and}\quad
\frac{dy}{dt}=-\frac{\frac{\partial H}{\partial x}(x,y)}{V(x,y)}=-\frac{m\Im(x+iy)^{m-1}}{V(x,y)},
\]
so
\begin{equation}\label{eq:X}
X(x,y)= X_1(x,y) + iX_2(x,y) =\frac{m\overline{(x+iy)^{m-1}}}{V(x,y)}
\end{equation}
and
\[
\eta=m\Im(x+iy)^{m-1}\,dx+m\Re(x+iy)^{m-1}\,dy.
\]
Therefore,  a $C^1$-curve $\gamma:[a,b]\to U_\sigma$ is standard if and only if
\begin{align}
\label{eq:eta}
\begin{split}
1 = \eta_{\gamma(t)}\gamma'(t)&=m\Im(\gamma(t))^{m-1}\Re \gamma'(t)+m\Re(\gamma(t))^{m-1}\Im \gamma'(t)\\
&=\Im\big(m(\gamma(t))^{m-1}\gamma'(t)\big)=
\Im\left(\frac{d}{dt}(\gamma(t))^m\right).
\end{split}
\end{align}

For every $f\in C^m(M)$ and any $\alpha=(\alpha_1,\alpha_2)\in \Z_{\geq 0}^2$ with $|\alpha|=\alpha_1+\alpha_2\leq m$ let
$\partial_\sigma^\alpha(f)=\frac{\partial^{|\alpha|}(f\cdot V)}{\partial^{\alpha_1} x\partial^{\alpha_2} y}(0,0)$.

\begin{lemma}\label{lem;partial-invariance}
For every $f\in C^m(M)$ and any $\alpha\in \Z_{\geq 0}^2$ with $|\alpha|\leq m-2$ we have
\[\partial_\sigma^\alpha(f)=\partial_\sigma^\alpha(f\circ \psi_t)\text{ for every }t\in \R.\]
\end{lemma}

\begin{proof}
First note that for every $(x,y)\in U_\sigma\cap\psi_{-t}(U_\sigma)$ we have
\begin{align*}
\frac{d}{dt}((f\cdot V)\circ\psi_t)(x,y)&
=\frac{\frac{\partial(f\cdot V)}{\partial x}(\psi_t(x,y))}{V(\psi_t(x,y))}(V\cdot X_1)(\psi_t(x,y))\\
&\quad+
\frac{\frac{\partial (f\cdot V)}{\partial y}(\psi_t(x,y))}{V(\psi_t(x,y))}(V\cdot X_2)(\psi_t(x,y)).
\end{align*}
Therefore, by induction
\begin{gather*}
\frac{d}{dt}\frac{\partial^{|\alpha|}}{\partial^{\alpha_1}x\partial^{\alpha_2}y}((f\cdot V)\circ\psi_t)(x,y)=
\frac{\partial^{|\alpha|}}{\partial^{\alpha_1}x\partial^{\alpha_2}y}\frac{d}{dt}((f\cdot V)\circ\psi_t)(x,y)\\=
\sum_{|\beta|\leq |\alpha|}W_{\beta,1}(t,x,y)\frac{\partial^{|\beta|}}{\partial^{\beta_1}x\partial^{\beta_2}y}(V\cdot X_1)(\psi_t(x,y))\\
\qquad+
\sum_{|\beta|\leq |\alpha|}W_{\beta,2}(t,x,y)\frac{\partial^{|\beta|}}{\partial^{\beta_1}x\partial^{\beta_2}y}(V\cdot X_2)(\psi_t(x,y)).
\end{gather*}
As $V\cdot X_1$ and $V\cdot X_2$ are homogenous polynomials of degree $m-1$, we have
\[\frac{\partial^{|\beta|}}{\partial^{\beta_1}x\partial^{\beta_2}y}(V\cdot X_1)(0,0)=
\frac{\partial^{|\beta|}}{\partial^{\beta_1}x\partial^{\beta_2}y}(V\cdot X_2)(0,0)=0\text{ if }|\beta|\leq m-2.\]
It follows that
\[\frac{d}{dt}\frac{\partial^{|\alpha|}}{\partial^{\alpha_1}x\partial^{\alpha_2}y}((f\cdot V)\circ\psi_t)(0,0)=0\text{ for all }t\in\R\text{ and }|\alpha|\leq m-2.\]
Hence
\[\partial_\sigma^\alpha(f\circ \psi_t)=\frac{\partial^{|\alpha|}}{\partial^{\alpha_1}x\partial^{\alpha_2}y}((f\cdot V)\circ\psi_t)(0,0)=
\frac{\partial^{|\alpha|}}{\partial^{\alpha_1}x\partial^{\alpha_2}y}(f\cdot V)(0,0)=\partial_\sigma^\alpha(f).\]
\end{proof}

Let $G_0:\C\to\C$ be the principal $m$-th root map, i.e.\ $G_0(re^{is})=r^{1/m}e^{is/m}$ if $s\in[0,2\pi)$, and let $\omega\in \C$ be the principal $2m$-th root of unity.
\begin{definition}
For every $\vep>0$ denote by $D_\vep$ the pre-image of the square $[-\vep,\vep]\times[-\vep,\vep]$
by the map $z\mapsto z^m$. Given a neighborhood $U_\sigma$ of $\sigma$, choose $\vep>0$ such that $D_\vep = D_{\sigma,\vep} \subset U_\sigma$.
\end{definition}

Let us consider four curves that parametrize some incoming and outgoing segments of the boundary of $D_\vep$:
$\gamma_+^{in},\gamma_+^{out}:(0, \vep)\to \partial D_\vep$, $\gamma_-^{in},\gamma_-^{out}:(- \vep,0)\to \partial D_\vep$ are given by
\[\gamma_{\pm}^{in}(s)=  G_0(-\vep+ is),\quad \gamma_{\pm}^{out}(s)=  G_0(\vep+ is).\]

For every interval $J\subset [0,2\pi)$ denote by $\mathcal{S}(J)$ the corresponding angular sector
$\{z\in \C:\operatorname{Arg}(z)\in J\}$.

\begin{lemma}\label{lem:segm} The following statements hold:

(i) The maps $\gamma_{\pm}^{in}$/$\gamma_{\pm}^{out}$ are standard parametrizations of incoming/outgoing segments of $D_\vep\cap \mathcal{S}([0,2\pi/m))$ for the flow $\psi_\R$.

(ii) The orbit segments entering  $D_\vep$ at $\gamma_{\pm}^{in}(s)$
leave it at $\gamma_{\pm}^{out}(s)$.

 Denote by $\tau(s)$ the time spent by this orbit in the set $D_\vep$. Then

 (iii)  for every $f\in C^m(M)$
 we have
\begin{equation}\label{eq:inttau}
\int_0^{\tau(s)}f(\psi_t(\gamma_{\pm}^{in}(s)))\,dt=\frac{1}{m^2}\int_{- \vep}^{ \vep}
\frac{(f\cdot V)(G_0(u,s))}{(u^2+s^2)^{\frac{m-1}{m}}}du.
\end{equation}
\end{lemma}

\begin{proof}
As
\[\Im(\gamma_{\pm}^{in}(s)^m)= \Im (G_0(-\vep+ is )^m)=s,\quad \Im(\gamma_{\pm}^{out}(s)^m)= \Im (G_0(\vep+i s )^m)= s,\]
in view of \eqref{eq:eta}, the parametrizations $\gamma_{\pm}^{in}$, $\gamma_{\pm}^{out}$ are standard.

Since the map $z\mapsto z^m$ is a bijection between $D_\vep\cap \mathcal{S}([0,2\pi/m))$ and $[-\vep,\vep]\times[-\vep,\vep]$, and
$G_0$ is its inverse, let us consider a local flow $\tilde{\psi}_\R$ on   $[-\vep,\vep]\times[-\vep,\vep]$ conjugated to the flow
$\psi_\R$ restricted to $D_\vep\cap \mathcal{S}([0,2\pi/m))$, i.e.\ $\tilde{\psi}_t(z)=\psi_t(G_0(z))^m$. By \eqref{eq:X},
\begin{align*}
\frac{d}{dt}\tilde{\psi}_t(z)&=m\,\psi_t(G_0(z))^{m-1}\frac{d}{dt}\psi_t(G_0(z))=m\,\psi_t(G_0(z))^{m-1}X(\psi_t(G_0(z)))\\
&=m^2\frac{|\psi_t(G_0(z))|^{2(m-1)}}{V(\psi_t(G_0(z)))}= m^2\frac{|\tilde{\psi}_t(z)|^{\frac{2(m-1)}{m}}}{V\circ G_0(\tilde{\psi}_t(z))}.
\end{align*}
Hence
\[\frac{d}{dt}\Re\tilde{\psi}_t(z)=m^2\frac{|\tilde{\psi}_t(z)|^{\frac{2(m-1)}{m}}}{V\circ G_0(\tilde{\psi}_t(z))}>0
\text{ and }\frac{d}{dt}\Im\tilde{\psi}_t(z)=0.\]
It follows that the interval $\{(-\vep,s ):s\in(- \vep, \vep)\}$ is the incoming and $\{(\vep,s ):s\in(- \vep, \vep)\}$ is the outgoing
segment  of  $[-\vep,\vep]\times[-\vep,\vep]$ for the local flow  $\tilde{\psi}_\R$.
Moreover,
the orbit segments entering  $[-\vep,\vep]\times[-\vep,\vep]$ at $(-\vep,s )$
leave it at $(\vep,s )$.
Passing via $G_0$ to the flow $\psi_\R$, we obtain the first claim of the lemma.

Recall that $\tau(s)$ is the time spent by $\tilde{\psi}_\R$-orbit starting at $(-\vep,s )$ in the set $[-\vep,\vep]\times[-\vep,\vep]$.
Then
\begin{align*}
\int_0^{\tau(s)}f(\psi_t(\gamma_{\pm}^{in}(s)))\,dt&=\int_0^{\tau(s)}f\circ G_0\big(\tilde{\psi}_t(-\vep,s )\big)\,dt\\
&=
\int_0^{\tau(s)}f\circ G_0\big(\Re\tilde{\psi}_t(-\vep,s ),s \big)\,dt.
\end{align*}
Next we integrate by substituting $u(t)=\Re\tilde{\psi}_t(-\vep,s )$.
As
$$- \vep=\Re\tilde{\psi}_0(-\vep,s ), \quad \vep=\Re\tilde{\psi}_{\tau(s)}(-\vep,s )$$
and
\begin{align*}
\frac{du}{dt}&=\frac{d}{dt}\Re\tilde{\psi}_t(-\vep,s )
=m^2\frac{|\tilde{\psi}_t(-\vep,s )|^{\frac{2(m-1)}{m}}}{V\circ G_0(\tilde{\psi}_t(-\vep,s ))}\\
&= m^2\frac{((\Re\tilde{\psi}_t(-\vep,s ))^2+s^2)^{\frac{m-1}{m}}}{V\circ G_0(\Re\tilde{\psi}_t(-\vep,s ),s )}
= m^2\frac{(u^2+s^2)^{\frac{m-1}{m}}}{V\circ G_0(u ,s )},
\end{align*}
by change of variables, we have
\begin{align*}
\int_0^{\tau(s)}f\circ G_0\big(\Re\tilde{\psi}_t(-\vep,s ),s \big)\,dt& =\int_{- \vep}^{ \vep}f\circ G_0(u,s)
\frac{V\circ G_0(u,s)}{m^2(u^2+s^2)^{\frac{m-1}{m}}}du\\
&=\frac{1}{m^2}\int_{- \vep}^{ \vep}
\frac{(f\cdot V)(G_0(u,s))}{(u^2+s^2)^{\frac{m-1}{m}}}du.
\end{align*}
This gives \eqref{eq:inttau}.
\end{proof}

\begin{remark}\label{rem:eeD}
Lemma~\ref{lem:segm} describes incoming and outgoing segments on the boundary of $D_\vep$ but only in the angular sector
$\mathcal{S}([0,2\pi/m))$. The same arguments apply  to the flow $\psi_\R$ restricted to $\mathcal{S}([2\pi k/m,2\pi (k+1)/m))$ for
$0\leq k< m$. As $\omega\in \C$ is the principal $2m$-th root of unity, the incoming/outgoing segments of $D_\vep\cap \mathcal{S}([2\pi k/m,2\pi (k+1)/m))$ are given by $\omega^{2k}\gamma_{\pm}^{in}$ and
 $\omega^{2k}\gamma_{\pm}^{out}$ respectively. Moreover, if $\tau_k(s)$ is the time spent by ${\psi}_\R$-orbit starting at
$\omega^{2k}\gamma_{\pm}^{in}(s)$ in the set $D_\vep$, then
\begin{equation}\label{eq:inttaugen}
\varphi_f^{\sigma,k}(s):=\int_0^{\tau_k(s)}f(\psi_t(\omega^{2k}\gamma_{\pm}^{in}(s)))\,dt=\frac{1}{m^2}\int_{- \vep}^{ \vep}
\frac{(f\cdot V)(\omega^{2k}G_0(u,s))}{(u^2+s^2)^{\frac{m-1}{m}}}du.
\end{equation}
Note that for $(u,s)\in\R^2_{\geq 0}$ we have
\[G_0(-u,-s)=\omega G_0(u,s),\ G_0(u,-s)= \omega^2 \overline{G_0}(u,s),\ G_0(-u,s)=\omega\overline{G_0}(u,s).\]
It follows that for every $s\in(0,\vep)$ we have
\begin{align*}
m^2\varphi_f^{\sigma,k}(s)&=\int_{0}^{ \vep}
\frac{(f\cdot V)(\omega^{2k}G_0(u,s))}{(u^2+s^2)^{\frac{m-1}{m}}}du+\int_{0}^{\vep}
\frac{(f\cdot V)(\omega^{2k}G_0(-u,s))}{(u^2+s^2)^{\frac{m-1}{m}}}du\\
&=\int_{0}^{ \vep}
\frac{(f\cdot V)(\omega^{2k}G_0(u,s))}{(u^2+s^2)^{\frac{m-1}{m}}}du+\int_{0}^{\vep}
\frac{(f\cdot V)(\omega^{2k+1}\overline{G_0}(u,s))}{(u^2+s^2)^{\frac{m-1}{m}}}du
\end{align*}
and
\begin{align*}
m^2\varphi_f^{\sigma,k}(-s)&=\int_{0}^{ \vep}
\frac{(f\cdot V)(\omega^{2k}G_0(u,-s))}{(u^2+s^2)^{\frac{m-1}{m}}}du+\int_{0}^{\vep}
\frac{(f\cdot V)(\omega^{2k}G_0(-u,-s))}{(u^2+s^2)^{\frac{m-1}{m}}}du\\
&=\int_{0}^{ \vep}
\frac{(f\cdot V)(\omega^{2k+2}\overline{G_0}(u,s))}{(u^2+s^2)^{\frac{m-1}{m}}}du+\int_{0}^{\vep}
\frac{(f\cdot V)(\omega^{2k+1}{G_0}(u,s))}{(u^2+s^2)^{\frac{m-1}{m}}}du.
\end{align*}
\end{remark}

\subsection{Singularities of $\varphi_f^{\sigma,k}$}
The purpose of this section is to understand the type of singularity of functions $\varphi_f^{\sigma,k}$.
These functions are responsible for reading the singularities of $\varphi_f$ and provide the tools to prove Theorem~\ref{thm:phifform} in \S\ref{sec:regularity}.
\medskip

For every $m\geq 2$ let $G:\R^2_{\geq 0}\to\C$ be a continuous inverse of one of the maps $z\mapsto z^m$, $z\mapsto \ov{z}^m$, $z\mapsto -z^m$ or $z\mapsto -\ov{z}^m$. Then $G$ is homogenous of degree $1/m$ and analytic on $\R^2_{> 0}$. If $G_0:\R^2_{\geq 0}\to\C$ is the principal $m$-th root and $\omega$ is the principal $2m$-th root of unity, then $G$ is either $\omega^lG_0\text{ or }\omega^l\ov{G}_0 \text{ for some }0\leq l<2m.$
\medskip

  Let $f:{D}\to\R$ be a bounded Borel map where $D$ is the pre-image of the square $[-1,1]\times[-1,1]$
by the map $z\mapsto z^m$. For every $ a\geq 1/2$ let us consider the map $\varphi=\varphi_{f,a}:(0,1]\to\R$ given by
\begin{equation}\label{eqn;fgus}
\varphi(s)=\int_0^1\frac{f(G(u,s))}{(u^2+s^2)^a}\,du.
\end{equation}

\begin{remark}\label{rem:passvep}
Note that for every $\vep>0$ we have
\begin{align*}
\varphi_{f,a,\vep}(s)&
=\int_0^\vep\frac{f(G(u,s))}{(u^2+s^2)^a}\,du=\frac{1}{\vep}\int_0^1\frac{f(G(u/\vep,s))}{((u/\vep)^2+s^2)^a}\,du\\
&=
\vep^{2a-1}\int_0^1\frac{f(\vep^{-1/m}G(u,\vep s))}{(u^2+(\vep s)^2)^a}\,du=\vep^{2a-1}\varphi_{f\circ \vep^{-1/m},a}(\vep s).
\end{align*}
Therefore
\[s^{2a}\varphi'_{f,a,\vep}(s)=(\vep s)^{2a}\varphi'_{f\circ \vep^{-1/m},a}(\vep s).\]
\end{remark}
Notice that
\[s^{2a-1}\int_0^1\frac{1}{(u^2+s^2)^a}\,du=\int_0^1\frac{1}{((\frac{u}{s})^2+1)^a}\frac{du}{s}=\int_0^{1/s}\frac{dx}{(x^2+1)^a}.\]
Let us recall the definitions of Gamma function $\Gamma(z)$ and Beta function $B(x,y)$
\[B(x,y) := \int_0^1t^{x-1}{(1-t)}^{y-1}dt, \quad \Gamma(z) := \int_0^\infty x^{z-1}e^{-x}dx,\]
and let us denote
\[\Gamma_{a}:=\int_0^{+\infty}\frac{dx}{(x^2+1)^a}=\frac{1}{2}B(\frac{1}{2},a-\frac{1}{2})
=\frac{1}{2}\frac{\Gamma(\frac{1}{2})\Gamma(a-\frac{1}{2})}{\Gamma(a)}=\frac{\sqrt{\pi}}{2}\frac{\Gamma(a-\frac{1}{2})}{\Gamma(a)}.\]
Then,  for every $a>1/2$,
\begin{gather}
s^{2a-1}\int_0^1\frac{1}{(u^2+s^2)^a}\,du\leq \Gamma_a\text{ for all }s\in(0,1]\label{eq:esta>12}, \text{ and}\\
\lim_{s\to 0}s^{2a-1}\int_0^1\frac{1}{(u^2+s^2)^a}\,du =\Gamma_a.\label{eq:lima>12}
\end{gather}
If $a=1/2$ then
\begin{equation}\label{eq:esta=12}
\int_0^1\frac{1}{(u^2+s^2)^a}\,du=\int_0^{1/s}\frac{dx}{\sqrt{x^2+1}}=\log\Big(\frac{1}{s}+\sqrt{\frac{1}{s^2}+1}\Big)\leq \log\frac{3}{s}\leq 2+|\log s|.
\end{equation}
In view of \eqref{eq:esta>12} and \eqref{eq:esta=12}, for every $s\in(0,1]$,
\begin{eqnarray}\label{eq:hatP}
\begin{aligned}
s^{2a-1}\varphi_{|f|,a}(s)&\leq \|f\|_{\sup}\Gamma_{a}\ & \text{ if }\ a\geq 1/2,\\
\varphi_{|f|,a}(s)&\leq \|f\|_{\sup}(2+|\log s|)\ & \text{ if }\ a=1/2.
\end{aligned}
\end{eqnarray}

\medskip

In the following lemmas, we find the upper bound of $\varphi'$ by $C^k$-norms of the function $f$ and the order of vanishing at the saddle.

\begin{lemma}\label{lem:pag}
Suppose that $f:D\to\R$ is a $C^1$-map. For every $1/2\leq a\leq 1$ we have
\[|s^{2a}\varphi'(s)|\leq 2\|f\|_{C^1}\Gamma_{a+\frac{m-1}{2m}}.\]
Moreover,
\[\lim_{s\to 0}s^{2a}\varphi'(s)=-2af(0,0)\Gamma_{a+1}.\]
\end{lemma}

\begin{proof}
First note that $\varphi$ is a $C^1$-function on $(0,1]$ with
\begin{align}\label{eq:phiprim}
\begin{split}
\varphi'(s)&=\int_0^1\frac{\frac{\partial f}{\partial x}(G(u,s))\frac{\partial G_1}{\partial s}(u,s)+\frac{\partial f}{\partial y}(G(u,s))\frac{\partial G_2}{\partial s}(u,s)}{(u^2+s^2)^a}\,du\\
&\quad-2a\int_0^1\frac{s f(G(u,s))}{(u^2+s^2)^{a+1}}\,du.
\end{split}
\end{align}
As $(G_1+iG_2)^m=\pm u \pm is$, we have
\begin{equation}\label{eq:GG}
m(G_1+iG_2)^{m-1}\Big(\frac{\partial G_1}{\partial s}+i\frac{\partial G_2}{\partial s}\Big)=\pm i.
\end{equation}
Hence
\[\Big|\frac{\partial G_1}{\partial s}+i\frac{\partial G_2}{\partial s}\Big|=\frac{1}{m|G_1+iG_2|^{m-1}}=\frac{1}{m(u^2+s^2)^{\frac{m-1}{2m}}}.\]
It follows that
\begin{gather}\label{eq:parg}
\begin{split}
\Big|&\int_0^1\frac{\frac{\partial f}{\partial x}(G(u,s))\frac{\partial G_1}{\partial s}(u,s)+\frac{\partial f}{\partial y}(G(u,s))\frac{\partial G_2}{\partial s}(u,s)}{(u^2+s^2)^a}\,du\Big|\\
&\leq \frac{\|f'\|_{C^0}}{m}\int_0^1\frac{1}{(u^2+s^2)^{a+\frac{m-1}{2m}}}\,du \leq \frac{\|f'\|_{C^0}}{m}\frac{\Gamma_{a+\frac{m-1}{2m}}}{s^{2a-\frac{1}{m}}}
 \leq
\frac{\|f'\|_{C^0}}{m}\frac{\Gamma_{a+\frac{m-1}{2m}}}{s^{2a}}
\end{split}
\end{gather}
and
\[\Big|\int_0^1\frac{s f(G(u,s))}{(u^2+s^2)^{a+1}}\,du\Big|\leq \|f\|_{C^0}\int_0^1\frac{s}{(u^2+s^2)^{a+1}}\,du\leq
{\|f\|_{C^0}}\frac{\Gamma_{a+1}}{s^{2a}}.\]
It follows that
\[|\varphi'(s)|\leq \Big(2a\|f\|_{C^0}\Gamma_{a+1}+\frac{\|f'\|_{C^0}}{m}\Gamma_{a+\frac{m-1}{2m}}\Big)\frac{1}{s^{2a}}\leq
2\|f\|_{C^1}\Gamma_{a+\frac{m-1}{2m}}\frac{1}{s^{2a}}.\]

Since $f$ is of class $C^1$, we have
\[|f(G(u,s))-f(0,0)|\leq \|f\|_{C^1}\|G(u,s)\|\leq\|f\|_{C^1}(u^2+s^2)^{\frac{1}{2m}}.\]
Moreover, by \eqref{eq:esta>12},
\[\int_0^1\frac{(u^2+s^2)^{\frac{1}{2m}}}{(u^2+s^2)^{a+1}}\,du \leq \frac{\Gamma_{a+1-\frac{1}{2m}}}{s^{2a+1-\frac{1}{m}}}.\]
Therefore, in view of \eqref{eq:phiprim}, \eqref{eq:parg}, we have
\begin{align*}
\Big|& s^{2a}\varphi'(s)+2af(0,0)s^{2a+1}\int_0^1\frac{du}{(u^2+s^2)^{a+1}}\Big|\\
&\leq s^{2a}\left|\int_0^1\frac{\frac{\partial f}{\partial x}(G(u,s))\frac{\partial G_1}{\partial s}(u,s)+\frac{\partial f}{\partial y}(G(u,s))\frac{\partial G_2}{\partial s}(u,s)}{(u^2+s^2)^a}\,du\right|\\
&\quad + 2as^{2a+1}\int_0^1\frac{|f(0,0)-f(G(u,s))|}{(u^2+s^2)^{a+1}}du\\
&=s^{2a}\|f'\|_{C^0}\frac{\Gamma_{a+\frac{m-1}{2m}}}{s^{2a-\frac{1}{m}}}+
2a s^{2a+1}\|f\|_{C^1}\int_0^1\frac{(u^2+s^2)^{\frac{1}{2m}}}{(u^2+s^2)^{a+1}}\,du\\
&\leq \|f\|_{C^1}\Gamma_{a+\frac{m-1}{2m}}s^{\frac{1}{m}}+2a s^{2a+1}\|f\|_{C^1}\frac{\Gamma_{a+1-\frac{1}{2m}}}{s^{2a+1-\frac{1}{m}}}=O(s^{\frac{1}{m}}).
\end{align*}
Hence
\[\lim_{s\to 0}s^{2a}\varphi'(s)=-\lim_{s\to 0}2af(0,0)s^{2a+1}\int_0^1\frac{du}{(u^2+s^2)^{a+1}}=-2af(0,0)\Gamma_{a+1}.\]
\end{proof}

\begin{lemma}\label{lemma;s2a}
Assume that $f:D\to\R$ is a $C^m$-map, $1/2\leq a\leq 1$ and let $k$ be a natural number such that $k\leq m(2a-1)$.
Suppose that $f^{(j)}(0,0)=0$ for $0\leq j<k$. Then
\[|s^{2a-\frac{k}{m}}\varphi'(s)|\leq \frac{\|f\|_{C^k}\Gamma_{1}}{(k-1)!}\text{ for all }s\in(0,1].\]
Moreover, if $k< m(2a-1)$ then
\[s^{2a-\frac{k}{m}-1}\varphi_{|f|,a}(s)\leq \frac{\|f\|_{C^k}\Gamma_{{a-\frac{k}{2m}}}}{k!}\text{ for all }s\in(0,1]\]
 and if $k=m(2a-1)$, then
\[\varphi_{|f|,a}(s)\leq \frac{\|f\|_{C^k}}{k!}(2+|\log s|)\text{ for all }s\in(0,1].\]
\end{lemma}

\begin{proof}
By assumption,
\begin{align*}
|f(G(u,s))|&\leq\frac{\|f^{(k)}\|_{C^0}}{k!}\|G(u,s)\|^k\leq\frac{\|f^{(k)}\|_{C^0}}{k!}(u^2+s^2)^{\frac{k}{2m}},\\
\left |\frac{\partial f}{\partial x}(G(u,s)) \right|&\leq\frac{\|(\partial f/\partial x)^{(k-1)}\|_{C^0}}{(k-1)!}\|G(u,s)\|^{k-1}\leq\frac{\|(\partial f/\partial x)^{(k-1)}\|_{C^0}}{(k-1)!}(u^2+s^2)^{\frac{k-1}{2m}},\\
\left |\frac{\partial f}{\partial y}(G(u,s)) \right|&\leq\frac{\|(\partial f/\partial y)^{(k-1)}\|_{C^0}}{(k-1)!}\|G(u,s)\|^{k-1}\leq\frac{\|(\partial f/\partial y)^{(k-1)}\|_{C^0}}{(k-1)!}(u^2+s^2)^{\frac{k-1}{2m}}.
 \end{align*}
Therefore
\[\varphi_{|f|,a}(s)\leq \frac{\|f^{(k)}\|_{C^0}}{k!}\int_0^1\frac{(u^2+s^2)^{\frac{k}{2m}}}{(u^2+s^2)^a}du
=\frac{\|f^{(k)}\|_{C^0}}{k!}\int_0^1\frac{1}{(u^2+s^2)^{a-\frac{k}{2m}}}du\]
and, by \eqref{eq:phiprim},
\begin{align*}
|\varphi'(s)|&\leq\frac{\|f^{(k)}\|_{C^0}}{m(k-1)!}\int_0^1\frac{(u^2+s^2)^{\frac{k-1}{2m}}}{(u^2+s^2)^{\frac{m-1}{2m}}(u^2+s^2)^a}du\\
&\quad+
2as\frac{\|f^{(k)}\|_{C^0}}{k!}\int_0^1\frac{(u^2+s^2)^{\frac{k}{2m}}}{(u^2+s^2)^{a+1}}du\\
&\leq\frac{\|f^{(k)}\|_{C^0}}{m(k-1)!}\int_0^1\frac{1}{(u^2+s^2)^{a+\frac{m-k}{2m}}}du+
2as\frac{\|f^{(k)}\|_{C^0}}{k!}\int_0^1\frac{1}{(u^2+s^2)^{a+1-\frac{k}{2m}}}du.
\end{align*}
As $k\leq m(2a-1)$, we have $a+1-\frac{k}{2m}>a+\frac{m-k}{2m}=1+\frac{(2a-1)m-k}{2m}\geq 1$.
By \eqref{eq:esta>12}, this gives
\begin{align*}
|\varphi'(s)|\leq\frac{\|f^{(k)}\|_{C^0}}{m(k-1)!}\frac{\Gamma_{a+\frac{m-k}{2m}}}{s^{2a-\frac{k}{m}}}+
2as\frac{\|f^{(k)}\|_{C^0}}{k!}\frac{\Gamma_{a+1-\frac{k}{2m}}}{s^{2a+1-\frac{k}{m}}}\leq\frac{\|f^{(k)}\|_{C^0}}{(k-1)!}\frac{\Gamma_{1}}{s^{2a-\frac{k}{m}}}.
\end{align*}
If additionally $k< m(2a-1)$ then $a-\frac{k}{2m}>\frac{1}{2}$. By \eqref{eq:esta>12} again,
\begin{align*}
\varphi_{|f|,a}(s)\leq \frac{\|f^{(k)}\|_{C^0}}{k!}\int_0^1\frac{1}{(u^2+s^2)^{a-\frac{k}{2m}}}du\leq\frac{\|f^{(k)}\|_{C^0}}{k!}\frac{\Gamma_{a-\frac{k}{2m}}}{s^{2a-1-\frac{k}{m}}}.
\end{align*}
On the other hand,  if $k=m(2a-1)$ then $a-\tfrac{k}{2m}=\tfrac{1}{2}$ and, by  \eqref{eq:esta=12},
\begin{align*}
\varphi_{|f|,a}(s)\leq \frac{\|f^{(k)}\|_{C^0}}{k!}\int_0^1\frac{1}{(u^2+s^2)^{a-\frac{k}{2m}}}du\leq \frac{\|f^{(k)}\|_{C^0}}{k!}(2+|\log s|).
\end{align*}
\end{proof}

\subsection{Quantities $C^\pm_\alpha(\varphi_f)$}
In this section, we develop some tools that help compute the non-vanishing quantities $C^\pm_\alpha(\varphi_f)$.

\begin{definition}
For every real $\beta$ let $C_{\beta}(\R_{>0}^2)$ be the space of
continuous homogenous functions $H:\R_{>0}\times \R_{>0}\to\C$ of degree $ \beta$ such that $H(u,s)=O(\|(u,s)\|^\beta)$.

For every real $a$ such that  $2a-\beta>1$ and  $H\in C_{\beta}(\R_{>0}^2)$ let
\[\Gamma_a(H):=\int_0^{+\infty}\frac{H(x,1)}{(x^2+1)^a}\,dx.\]
\end{definition}
As ${H(x,1)}/{(x^2+1)^a}=O(1/(1+x)^{2a-\beta})$, the quantity $\Gamma_a(H)$ is well-defined and
\begin{equation}\label{def;gammaH}
s^{2a-\beta-1}\int_0^1\frac{H(u,s)}{(u^2+s^2)^a}\,du=\int_0^1\frac{H(u/s,1)}{((u/s)^2+1)^a}\,\frac{du}{s}=\int_0^{1/s}\frac{H(x,1)}{(x^2+1)^a}\,dx\to \Gamma_a(H).
\end{equation}

Note that if $H(u,s)=\tilde{H}(u,s)/\|(u,s)\|^{\max\{-\beta,0\}}$ and $\tilde{H}:\R_{\geq 0}\times \R_{\geq 0}\to\C$
is continuous homogenous of degree $\max\{\beta,0\}$ then $H\in C_{\beta}(\R_{>0}^2)$.

\begin{lemma}
Assume that $2a-\beta>1$,  $H\in C_{\beta}(\R_{>0}^2)$ is of class $C^1$ and $\tfrac{\partial H}{\partial y}\in C_{\beta-1}(\R_{>0}^2)$. Then
\begin{gather}
\label{eq:parH}
\Gamma_{a}(\tfrac{\partial H}{\partial y})=2a\Gamma_{a+1}(H)-(2a-1-\beta)\Gamma_a(H).
\end{gather}
\end{lemma}
\begin{proof}
Note that for every $y>0$,
\begin{equation*}
\Gamma_a(H)=\int_0^{+\infty}\frac{H(x,1)}{(x^2+1)^a}\,dx=\int_0^{+\infty}\frac{H(x/y,1)}{(x^2/y^2+1)^a}\,\frac{dx}{y}=
y^{2a-1-\beta}\int_0^{+\infty}\frac{H(x,y)}{(x^2+y^2)^a}\,dx.
\end{equation*}
As $H\in C_{\beta}(\R_{>0}^2)$  and $\tfrac{\partial H}{\partial y}\in C_{\beta-1}(\R_{>0}^2)$, by differentiating with respect to $y$, we get
\begin{align*}
-(2a-1-\beta)y^{\beta-2a}\Gamma_a(H)&=\frac{d}{dy}y^{1+\beta-2a}\Gamma_a(H)=\frac{d}{dy}\int_0^{+\infty}\frac{H(x,y)}{(x^2+y^2)^a}\,dx\\
&=\int_0^{+\infty}\frac{\frac{\partial H}{\partial y}(x,y)}{(x^2+y^2)^a}\,dx-2ay\int_0^{+\infty}\frac{H(x,y)}{(x^2+y^2)^{a+1}}\,dx.
\end{align*}
Taking $y=1,$ this yields \eqref{eq:parH}.
\end{proof}

Recall that, by Lemma~\ref{lem:pag}
for any $C^1$-map $f:D\to\R$ with $f(0,0)\neq 0$ and for every $1/2\leq a\leq 1$ we have
\[\lim_{s\to 0}s^{2a}\varphi'(s)=-2af(0,0)\Gamma_{a+1}.\]
In the next preliminary lemmas, we find precise asymptotics of $\varphi'$ at zero also when the function $f$ and some of its derivatives vanish at the saddle.
This plays a crucial role in calculating the quantity $C^\pm_\alpha(\varphi_f)$ explicitly in  \S\ref{sec:regularity}.

\begin{lemma}\label{lem:Cpm}
Let $k$ be an integer number such that $0\leq k\leq m(2a-1)$.
Suppose that $f:D\to\R$ is of class $C^{k+1}$ and  $f^{(j)}(0,0)=0$ for $0\leq j<k$. Then
\begin{equation}\label{eq:stepk}
\lim_{s\to 0}s^{2a-\frac{k}{m}}\varphi'(s)
=\sum_{j=0}^k\binom{k}{j}\frac{\partial^k f}{\partial x^j \partial y^{k-j}}(0,0)\Gamma_a^{k,j}(G),
\end{equation}
where
\begin{align*}
\Gamma_a^{k,j}(G)=-\frac{2a-1-\frac{k}{m}}{k}\Gamma_{a}(G_1^jG_2^{k-j})-2a\frac{k-1}{k}\Gamma_{a+1}(G_1^jG_2^{k-j})\text{ if }k\geq 1
\end{align*}
and
$\Gamma_a^{0,0}(G)=-2a\Gamma_{a+1}$.
\end{lemma}

\begin{proof}
If $k=0$, then \eqref{eq:stepk} follows directly from Lemma~\ref{lem:pag}.

Assume that $k\geq 1$. Then,
by assumptions,
\begin{align*}
f(G(u,s))&=\sum_{j=0}^k\binom{k}{j}\frac{\partial^k f}{\partial x^j \partial y^{k-j}}(0,0)G_1^j(u,s)G_2^{k-j}(u,s) + O(\|G(u,s)\|^{k+1})\\
&=\sum_{j=0}^k\binom{k}{j}\frac{\partial^k f}{\partial x^j \partial y^{k-j}}(0,0)G_1^j(u,s)G_2^{k-j}(u,s) + O((u^2+s^2)^{\frac{k+1}{2m}}).
\end{align*}
Moreover,
\begin{align*}
\frac{\partial f}{\partial x}(G(u,s))
&=\sum_{j=0}^{k-1}\binom{k-1}{j}\frac{\partial^k f}{\partial x^{j+1} \partial y^{k-1-j}}(0,0)G_1^j(u,s)G_2^{k-1-j}(u,s) + O((u^2+s^2)^{\frac{k}{2m}}),\\
\frac{\partial f}{\partial y}(G(u,s))
&=\sum_{j=0}^{k-1}\binom{k-1}{j}\frac{\partial^k f}{\partial x^{j} \partial y^{k-j}}(0,0)G_1^j(u,s)G_2^{k-1-j}(u,s) + O((u^2+s^2)^{\frac{k}{2m}}).
\end{align*}
As $\|G(u,s)\|=\|(u,s)\|^{\frac{1}{m}}=(u^2+s^2)^{\frac{1}{2m}}$ and $\|\frac{\partial G}{\partial s}(u,s)\|=\frac{\|(u,s)\|^{-\frac{m-1}{m}}}{m}=\frac{(u^2+s^2)^{-\frac{m-1}{2m}}}{m}$, it follows that
\begin{align}\label{eqn;deriv}
\begin{split}
\frac{d}{ds}\frac{f(G(u,s))}{(u^2+s^2)^a}&=
\frac{\frac{\partial f}{\partial x}(G(u,s))\frac{\partial G_1}{\partial s}(u,s)
+\frac{\partial f}{\partial y}(G(u,s))\frac{\partial G_2}{\partial s}(u,s)}{(u^2+s^2)^a}
-2a\frac{s f(G(u,s))}{(u^2+s^2)^{a+1}}\\
&=\sum_{j=0}^k\binom{k}{j}\frac{\partial^k f}{\partial x^{j} \partial y^{k-j}}(0,0)
\frac{\big(\frac{j}{k}G_1^{j-1}G_2^{k-j}\frac{\partial G_1}{\partial s}+
\frac{k-j}{k}G_1^jG_2^{k-1-j}\frac{\partial G_2}{\partial s}\big)(u,s)
}{(u^2+s^2)^{a}}\\
&\quad-2as\sum_{j=0}^k\binom{k}{j}\frac{\partial^k f}{\partial x^{j} \partial y^{k-j}}(0,0)
\frac{(G_1^{j}G_2^{k-j})(u,s)}{(u^2+s^2)^{a+1}}\\
&\quad + O\Big(\frac{1}{(u^2+s^2)^{a+\frac{m-k-1}{2m}}}\Big)+O\Big(\frac{s}{(u^2+s^2)^{a+1-\frac{k+1}{2m}}}\Big).
\end{split}
\end{align}
Since $G_1^{j-1}G_2^{k-j}\frac{\partial G_1}{\partial s}$, $G_1^{j}G_2^{k-j-1}\frac{\partial G_2}{\partial s}$ are homogenous of degree $\frac{k-m}{m}<2a-1$ and $G_1^{j}G_2^{k-j}$ is homogenous of degree $\frac{k}{m}<2(a+1)-1$, by \eqref{def;gammaH} we have
\begin{gather}\label{eqn;G1G2}
\begin{split}
\lim_{s\to 0}s^{2a-\frac{k}{m}}\int_0^1\frac{(G_1^{j-1}G_2^{k-j}\frac{\partial G_1}{\partial s})(u,s)}{(u^2+s^2)^{a}}\,du=
\Gamma_{a}(G_1^{j-1}G_2^{k-j}\tfrac{\partial G_1}{\partial s}),\\
\lim_{s\to 0}s^{2a-\frac{k}{m}}\int_0^1\frac{(G_1^{j}G_2^{k-j-1}\frac{\partial G_2}{\partial s})(u,s)}{(u^2+s^2)^{a}}\,du=
\Gamma_{a}(G_1^{j}G_2^{k-j-1}\tfrac{\partial G_2}{\partial s}),\\
\lim_{s\to 0}s^{2a-\frac{k}{m}}\int_0^1\frac{s (G_1^{j}G_2^{k-j})(u,s)}{(u^2+s^2)^{a+1}}\,du=\Gamma_{a+1}(G_1^{j}G_2^{k-j}).
\end{split}
\end{gather}
Furthermore,
\begin{gather*}
\lim_{s\to 0}s^{2a-\frac{k+1}{m}}\int_0^1\frac{1}{(u^2+s^2)^{a+\frac{m-k-1}{2m}}}\,du=\Gamma_{a+\frac{m-k-1}{2m}},\\
\lim_{s\to 0}s^{2a-\frac{k+1}{m}}\int_0^1\frac{s}{(u^2+s^2)^{a+1-\frac{k+1}{2m}}}=\Gamma_{a+1-\frac{k+1}{2m}}.
\end{gather*}
In view of \eqref{eqn;deriv} and \eqref{eqn;G1G2}, this gives the statement \eqref{eq:stepk} with
\[\Gamma_a^{k,j}(G)=\frac{j}{k}\Gamma_{a}(G_1^{j-1}G_2^{k-j}\tfrac{\partial G_1}{\partial s})+\frac{k-j}{k}\Gamma_{a}(G_1^jG_2^{k-1-j}\tfrac{\partial G_2}{\partial s})-2a\Gamma_{a+1}(G_1^jG_2^{k-j}).\]
In view of \eqref{eq:parH} applied to $H=G_1^jG_2^{k-j}$ (which is homogenous of degree $k/m$), we get
\begin{align*}
\Gamma_a^{k,j}(G)&=\frac{1}{k}\Gamma_{a}\big(\tfrac{\partial}{\partial s}(G_1^{j}G_2^{k-j})\big)-2a\Gamma_{a+1}(G_1^jG_2^{k-j})\\
&=\frac{2a}{k}\Gamma_{a+1}(G_1^{j}G_2^{k-j})-\frac{2a-1-\frac{k}{m}}{k}\Gamma_{a}(G_1^jG_2^{k-j})-2a\Gamma_{a+1}(G_1^jG_2^{k-j})\\
&=-\frac{2a-1-\frac{k}{m}}{k}\Gamma_{a}(G_1^jG_2^{k-j})-2a\frac{k-1}{k}\Gamma_{a+1}(G_1^jG_2^{k-j}).
\end{align*}
\end{proof}

Let $G_0:\R^2_{\geq 0}\to\C$ be the principal branch of the $m$-th root and let $\omega$ and $\omega_0$ be the principal $2m$-th and $4m$-th root of unity respectively.
\begin{lemma}\label{lem:nonnegsum}
Let $1/2<a\leq (m-1)/m$, $1\leq k< (2a-1)m$ and $a_0,a_1,\ldots,a_k$ are real numbers not all equal to $0$. Then there exists $0\leq l<2m$
such that
\[\sum_{j=0}^ka_j\big(\Gamma_a^{k,j}(\omega^lG_0)+\Gamma_a^{k,j}(\omega^{l+1}\ov{G_0})\big)\neq 0.\]
\end{lemma}

\begin{proof}
Let $\mathfrak{G}:C_{k/m}(\R_{>0}^2)\to\C$ be the linear operator given by
\[\mathfrak{G}(H):= -\frac{2a-1-\frac{k}{m}}{k}\Gamma_{a}(H)-2a\frac{k-1}{k}\Gamma_{a+1}(H).\]
Then $\ov{\mathfrak{G}(H)}=\mathfrak{G}(\ov{H})$ and $\Gamma_a^{k,j}(G)=\mathfrak{G}(G_1^jG_2^{k-j})$.

Suppose that, contrary to our claim,
\begin{equation}\label{eq:contr}
\sum_{j=0}^ka_j
\big(\Gamma_a^{k,j}(\omega^lG_0)+\Gamma_a^{k,j}(\omega^{l+1}\ov{G_0})\big)= 0\text{ for every }0\leq l<2m.
\end{equation}
Denote by $\R_k[x,y]$ the linear space of  homogenous polynomial of degree $k$.
The space $\R_k[x,y]$ coincides with the subspace $\C_{k,\R}[z,\ov{z}]$ of complex homogenous polynomials $\C_{k}[z,\ov{z}]$
of the form  $\sum_{j=0}^kc_jz^j\ov{z}^{k-j}$ such that $\ov{c}_j=c_{k-j}$ for $0\leq j\leq k$. For every $P\in \R_k[x,y]$ denote by
$\widehat{P}\in\C_{k,\R}[z,\ov{z}]$ the unique polynomial
such that $\widehat{P}(z,\ov{z})=P(x,y)$.

As $Q(x,y)=\sum_{j=0}^ka_jx^jy^{k-j}\in\R_k[x,y]$ is non-zero by assumptions, the corresponding polynomial $\widehat{Q}(z,\ov{z})=\sum_{j=0}^kc_jz^j\ov{z}^{k-j}$ is also non-zero.
Note that
\begin{align}\label{eq:GQ}
\begin{split}
\sum_{j=0}^ka_j\Gamma_a^{k,j}(G)&=\sum_{j=0}^ka_j\mathfrak{G}(G_1^jG_2^{k-j})={\mathfrak{G}}(Q(G_1,G_2))\\
&=\mathfrak{G}(\widehat{Q}(G,\ov{G}))=\sum_{j=0}^kc_j\mathfrak{G}(G^j\ov{G}^{k-j}).
\end{split}
\end{align}
As $k<m(2a-1)\leq m-2$, in view of \eqref{eq:contr} and \eqref{eq:GQ}, for every $0\leq l\leq 2k$ we have
\begin{align*}
0&=\sum_{j=0}^kc_j\Big(\mathfrak{G}\big((\omega^lG_0)^j(\ov{\omega^lG_0})^{k-j}\big)+
\mathfrak{G}\big((\omega^{l+1}\ov{G_0})^j(\ov{\omega^{l+1}\ov{G_0}})^{k-j}\big)\Big)\\
&=\omega^{-kl}\sum_{j=0}^k\omega^{2jl}c_j\big(\mathfrak{G}(G_0^j\ov{G_0}^{k-j})+\omega^{2j-k}\mathfrak{G}(\ov{G_0}^j{G_0}^{k-j})\big).
\end{align*}
Let us consider the matrix $\Omega_k=[\omega^{2lj}]_{0\leq l,j\leq k}\in M_{(k+1)\times(k+1)}(\C)$. As $k<m$, by  the Vandermonde determinant,
\[\det \Omega_k=\prod_{0\leq i<j\leq k}(\omega^{2j}-\omega^{2i})\neq 0.\]
This gives
\[c_j\Big(\mathfrak{G}\big(G_0^j\ov{G_0}^{k-j}\big)+\omega^{2j-k}\ov{\mathfrak{G}\big({G^j_0}\ov{G_0}^{k-j}\big)}\Big)=0\text{ for all }0\leq j\leq k.\]
As $c_{k-j}=\ov{c_j}$ and $\widehat{Q}$ is non-zero, there exists $0\leq j\leq k/2$ such that $c_j\neq 0$. Then
\begin{align*}
0&=\omega_0^{k-2j}\Big(\mathfrak{G}\big(G_0^j\ov{G_0}^{k-j}\big)+\omega_0^{2j-k}\ov{\mathfrak{G}\big({G^j_0}\ov{G_0}^{k-j}\big)}\Big)\\
&=\mathfrak{G}\big((\omega^{-1}_0G_0)^j\ov{(\omega^{-1}_0G_0)}^{k-j}\big)+\ov{\mathfrak{G}\big({(\omega^{-1}_0G_0)}^j\ov{(\omega^{-1}_0G_0)}^{k-j}\big)}.
\end{align*}
Hence
\[0=\Re\mathfrak{G}\big((\omega^{-1}_0G_0)^j\ov{(\omega^{-1}_0G_0)}^{k-j}\big)=\mathfrak{G}\big(|G_0|^j\Re((\omega_0\ov{G_0})^{k-2j})\big).\]
Since $G_0$ is the principal $m$-th root, for all $u,s>0$ we have $\Arg G_0(u,s)\in(0,\tfrac{\pi}{2m})$. Hence
\[\Arg (\omega_0\ov{G_0(u,s)})\in(0,\tfrac{\pi}{2m})\text{ and }\Arg (\omega_0\ov{G_0(u,s)})^{k-2j}\in(0,(k-2j)\tfrac{\pi}{2m})\subset(0,\tfrac{\pi}{2}),\]
so $\Re((\omega_0\ov{G_0}(u,s))^{k-2j})>0$.
As $2a-1-\frac{k}{m}>0$, by the definition of $\mathfrak{G}$ we have
\[\mathfrak{G}\big(|G_0|^j\Re((\omega_0\ov{G_0})^{k-2j})\big)<0,\]
 which is a contradiction.
\end{proof}

\section{Global properties of the operator $f\mapsto\varphi_f$ and correcting  operators}\label{sec:regularity}
In this section, we use the results of the previous section to prove an extended version of  Theorem~\ref{thm:ftophi}, which is Theorem~\ref{thm:phifform}.

For every $\sigma \in \mathrm{Fix}(\psi_\R)$, let $G_\sigma:\C\to\C$ be the principal $m_\sigma$-root map and let $\omega_\sigma$ be the principal $2m_\sigma$-root of unity. For every $0\leq k\leq m_\sigma-2$, recall that
\begin{equation}\label{def;ab}
a(\sigma) = \frac{m_\sigma-2}{m_\sigma}, \quad b(\sigma,k) = \frac{m_\sigma-2 - k}{m_\sigma}.
\end{equation}
Denote by $C^m_{\sigma,k}(M)$ the space of maps $f\in C^m(M)$, which vanish on $\bigcup_{\sigma'\in \mathrm{Fix}(\psi_\R)\setminus\{\sigma\}} U_{\sigma'}$ and
$f^{(j)}(\sigma)=0$ for all $0\leq j<k$.


\begin{theorem}\label{thm:phifform} The following statements hold:

(i) For every $f\in C^m(M)$ we have $\varphi_f \in \pag(\sqcup_{\alpha\in \mathcal{A}}
I_{\alpha})$ and $\varphi_{|f|} \in \operatorname{\widehat{P}_{a}}(\sqcup_{\alpha\in \mathcal{A}}
I_{\alpha})$, where $a=\frac{m-2}{m}$. Moreover, the operator
\[f \in C^{m}(M) \mapsto \varphi_f \in \pag(\sqcup_{\alpha\in \mathcal{A}}I_{\alpha})\]
is bounded. More precisely, there exists $C>0$ such that $\|\varphi_f\|_{a}\leq C\|f\|_{C^1}$ for every $f\in C^m(M)$.

(ii) For every $\sigma\in \mathrm{Fix}(\psi_\R)\cap M'$ and $0\leq k\leq m_\sigma-2$, there exists $C_{\sigma,k}>0$ such that
 for every $f\in C^m_{\sigma,k}(M)$ we have $\varphi_f \in \operatorname{P_{b(\sigma,k)}G}(\sqcup_{\alpha\in \mathcal{A}}
I_{\alpha})$, $\|\varphi_f\|_{b(\sigma,k)}\leq C_{\sigma,k}\|f\|_{C^{k+1}}$ and
$\varphi_{|f|} \in \operatorname{\widehat{P}_{b(\sigma,k)}}(\sqcup_{\alpha\in \mathcal{A}}
I_{\alpha})$.

(iii) Moreover, if additionally $\psi_\R$ is minimal on $M$ (i.e.\ $M'=M$), then for every $f\in C^m_{\sigma,k}(M)$ and
 for every $\alpha\in\mathcal{A}$ the quantity
$C^\pm_\alpha(\varphi_f)$ is zero or is of the form
\begin{equation}\label{eq:Cpm}
-\frac{1}{m^2_\sigma}\sum_{j=0}^k\binom{k}{j}\partial_\sigma^{(j,k-j)}(f)\Big(\Gamma_{\frac{m_\sigma-1}{m_\sigma}}^{k,j}
\big(\omega_{\sigma}^{l}G_{\sigma}\big)+
\Gamma_{\frac{m_\sigma-1}{m_\sigma}}^{k,j}\big(\omega_{\sigma}^{l+1}\overline{G_{\sigma}}\big)\Big)
\end{equation}
for some $0\leq l<2m_\sigma$. On the other hand, for every  $0\leq l<2m_\sigma$, there exists
$\alpha\in\mathcal{A}$ such that $C^\pm_\alpha(\varphi_f)$  is of the form \eqref{eq:Cpm}.
\end{theorem}

\begin{proof}
Without loss of generality we can assume that $\psi_\R$ is minimal on $M$. The proof of (i) and (ii) in the general case
proceeds in the same way up to some complications in notation.

Choose $\vep>0$ such that $D_{\sigma,\vep}\subset U_\sigma$ for any $\sigma\in \mathrm{Fix}(\psi_\R)$, where  $D_{\sigma,\vep}$ is
a closed neighborhood of $\sigma$ defined in \S\ref{sec:locHam}.
Denote by $g:I\to\R_{>0}\cup\{+\infty\}$ the {first return time} map. Since the flow $\psi_\R$ is smooth and $f$ is of class $C^m$, both $g$ and $\varphi_f$ belong to $C^1(\sqcup_{\alpha\in \mathcal{A}}
I_{\alpha})$. Moreover, for every $x\in \bigcup_{\alpha\in \mathcal{A}}
Int I_{\alpha}$ we have
\[|\varphi_f'(x)|\leq |g'(x)|\|f\|_{C^0}+\|f'\|_{C^0}\int_0^{g(x)}\Big\|\frac{d\psi_t (x)}{dx}\Big\|dt.\]
If additionally $\dist(x,End(T))\geq \vep$ then $|\varphi_f'(x)|\leq C_\vep\|f\|_{C_1}$, where
\[C_\vep:=\max\Big\{|g'(x)|+\int_0^{g(x)}\Big\|\frac{d\psi_t(x)}{dx}\Big\|dt:x\in I, \dist(x,End(T))\geq \vep\Big\}<+\infty.\]

Let $e\in End(T)$ and suppose that $e$ is the first backward intersection of a separatrix incoming to a fixed point $\sigma\in \mathrm{Fix}(\psi_\R)$. For every $x\in (e-\vep,e)\cup(e,e+\vep)$,
let $0<\tau_1(x)<\tau_2(x)<g(x)$ be the entrance ($\tau_1(x)$) and the exit  ($\tau_2(x)$) time of the orbit segment $\{\psi_t x:0\leq t\leq g(x)\}$ to $D_{\sigma,\vep}$.
Let us consider $\varphi_f^1,\varphi_f^2:(e-\vep,e)\cup(e,e+\vep)\to\R$ given by
\[\varphi_f^1(x)=\int_{\tau_1(x)}^{\tau_2(x)}f(\psi_t x)\,dt,\quad \varphi_f^2(x)=\int_{0}^{\tau_1(x)}f(\psi_t x)\,dt+\int_{\tau_2(x)}^{g(x)}f(\psi_t x)\,dt.\]
Of course, $\varphi_f(x)=\varphi_f^1(x)+\varphi_f^2(x)$ for every $x\in (e-\vep,e)\cup(e,e+\vep)$.
In view of Lemma~\ref{lem:segm} and Remark~\ref{rem:eeD}, there exists $0\leq l<m_\sigma$ such that
for every $s\in(0,\vep)$ we have
\begin{align*}
m_\sigma^2\varphi_f^1(e+s)
&=\int_{0}^{ \vep}
\frac{(f\cdot V)(\omega_{\sigma}^{2l}G_{\sigma}(u,s))}{(u^2+s^2)^{\frac{m_\sigma-1}{m_\sigma}}}du+\int_{0}^{\vep}
\frac{(f\cdot V)(\omega_{\sigma}^{2l+1}\overline{G_{\sigma}}(u,s))}{(u^2+s^2)^{\frac{m_\sigma-1}{m_{\sigma}}}}du,
\\
m_\sigma^2\varphi_f^1(e-s)
&=\int_{0}^{\vep}
\frac{(f\cdot V)(\omega_{\sigma}^{2l+1}{G_{\sigma}}(u,s))}{(u^2+s^2)^{\frac{m_\sigma-1}{m_\sigma}}}du+
\int_{0}^{ \vep}
\frac{(f\cdot V)(\omega_{\sigma}^{2l+2}\overline{G_{\sigma}}(u,s))}{(u^2+s^2)^{\frac{m_\sigma-1}{m_\sigma}}}du.
\end{align*}
Note that  $2\big(\tfrac{m_\sigma-1}{m_\sigma}\big)-1=a(\sigma)$.
In view of \eqref{eq:hatP}, Lemma~\ref{lem:pag} and Remark~\ref{rem:passvep}, it follows that
for every $x\in (e-\vep,e)\cup(e,e+\vep)$ we have
\begin{align}
\label{eq:fipie}
|x-e|^{a(\sigma)}\varphi_{|f|}^1(x)&\leq \Gamma_{\frac{m_\sigma-1}{m_\sigma}}\|f\cdot V\|_{C^0}\text{ if }m_\sigma>2,\\
\varphi_{|f|}^1(x)&\leq 2\|f\cdot V\|_{C^0}(2+|\log(\vep s)|)\text{ if }m_\sigma=2,\\
|x-e|^{a(\sigma)+1}|(\varphi_f^1)'(x)|&\leq \frac{4\vep^{-1/m_\sigma}\Gamma_{\frac{3(m_\sigma-1)}{2m_\sigma}}}{m^2_\sigma}\|f\cdot V\|_{C^1}
\leq \frac{\vep^{-1/m}}{m_\sigma^2}
\Gamma_{3/4}\|V\|_{C^1}
\|f\|_{C^1},\\
\label{eq:fiost}
\lim_{x\to e^{\pm}}|x-e|^{a(\sigma)+1}(\varphi_f^1)'(x)&=\mp\frac{2(a(\sigma)+1)}{m_\sigma^2}f(\sigma)V(\sigma)\Gamma_{\frac{2m_\sigma-1}{m_\sigma}}.
\end{align}
If additionally
$f^{(j)}(\sigma)=0$ for all $0\leq j<k$ ($1\leq k\leq m_\sigma-2$), then by Lemma~\ref{lemma;s2a}, we have
\begin{align}
\label{eq:locf}
|x-e|^{a(\sigma)-\frac{k}{m_\sigma}+1}|(\varphi_f^1)'(x)|&\leq \frac{2\Gamma_{1}}{m_\sigma^2(k-1)!}\|f\cdot V\|_{C^k} \leq
\frac{2\Gamma_{1}}{m_\sigma^2}\|V\|_{C^k}
\|f\|_{C^k},\\
\label{eq:loc|f|}
|x-e|^{a(\sigma)-\frac{k}{m_\sigma}}\varphi_{|f|}^1(x)&\leq \frac{2\Gamma_{\frac{2m_\sigma-1-k}{2m_\sigma}}}{m_\sigma^2k!}\|f\cdot V\|_{C^k},
\end{align}
and, by Lemma~\ref{lem:Cpm},
\begin{align}
\begin{split}\label{eq:stepk1}
\lim_{x\to e^+}&(x-e)^{a(\sigma)-\frac{k}{m_\sigma}+1}(\varphi_f^1)'(x)\\
&=\frac{1}{m^2_\sigma}\sum_{j=0}^k\binom{k}{j}\partial_\sigma^{(j,k-j)}(f)
\Big(\Gamma_{\frac{m_\sigma-1}{m_\sigma}}^{k,j}\big(\omega_{\sigma}^{2l}G_{\sigma}\big)+
\Gamma_{\frac{m_\sigma-1}{m_\sigma}}^{k,j}\big(\omega_{\sigma}^{2l+1}\overline{G_{\sigma}}\big)\Big),
\end{split}
\\
\begin{split}\label{eq:stepk2}
\lim_{x\to e^-}&(e-x)^{a(\sigma)-\frac{k}{m_\sigma}+1}(\varphi_f^1)'(x)\\
&=-\frac{1}{m^2_\sigma}\sum_{j=0}^k\binom{k}{j}\partial_\sigma^{(j,k-j)}(f)
\Big(\Gamma_{\frac{m_\sigma-1}{m_\sigma}}^{k,j}\big(\omega_{\sigma}^{2l+1}G_{\sigma}\big)+
\Gamma_{\frac{m_\sigma-1}{m_\sigma}}^{k,j}\big(\omega_{\sigma}^{2l+2}\overline{G_{\sigma}}\big)\Big).
\end{split}
\end{align}
Since $\tau^1$ and $g-\tau^2$ can be $C^1$-extended to the intervals $[e-\vep,e]$ and $[e,e+\vep]$,  for every $x\in (e-\vep,e)\cup(e,e+\vep)$ we have
$|(\varphi_f^2)'(x)|\leq C_{\sigma,\vep}\|f\|_{C^1}$, where
\begin{align*}
C_{\sigma,\vep}:&=\max\Big\{|\tau_1'(x)|+\int_{0}^{\tau_1(x)}\Big\|\frac{d\psi_t (x)}{dx}\Big\|dt:0<|x-e|<\vep\Big\}\\
&\quad+\max\Big\{|(g-\tau_2)'(x)|+\int_{0}^{g(x)-\tau_2(x)}\Big\|\frac{d\psi_{-t} (Tx)}{dx}\Big\|dt:0<|x-e|<\vep\Big\}<+\infty.
\end{align*}
As $a(\sigma)+1=\frac{2(m_\sigma-1)}{m_\sigma}\leq \frac{2(m-1)}{m}=a+1$ for every $\sigma \in \mathrm{Fix}(\psi_\R)$, in view of \eqref{eq:fipie}-\eqref{eq:fiost},
it follows that
$\varphi_f \in \pag(\sqcup_{\alpha\in \mathcal{A}}I_{\alpha})$ and $\varphi_{|f|} \in \operatorname{\widehat{P}_{a}}(\sqcup_{\alpha\in \mathcal{A}}
I_{\alpha})$ for every $f\in C^m(M)$ and
\[p_a(\varphi)\leq \Big(\sum_{\sigma\in \mathrm{Fix}(\psi_\R)}\big(\vep^{-1/m} \Gamma_{3/4}\|V\|_{C^1} + m_\sigma\vep^{a+1}C_{\sigma,\vep}\big)+
|I|^{a+1}C_\vep\Big)
\|f\|_{C^1}.\]
As $\|\varphi_f\|_{L^1}\leq \|f\|_{L^1}\leq \mu(M)\|f\|_{C^0}$, there exists $C>0$ such that $\|\varphi_f\|_{a}\leq C\|f\|_{C^1}$ for every $f\in C^m(M)$.

\medskip

Since $a(\sigma)-\frac{k}{m_\sigma}=b(\sigma,k)$, applying similar arguments for functions $f\in C^m_{\sigma,k}(M)$
and using \eqref{eq:locf}-\eqref{eq:stepk2} (instead of \eqref{eq:fipie}-\eqref{eq:fiost}), we obtain
$\varphi_f \in \operatorname{P_{b(\sigma,k)}G}(\sqcup_{\alpha\in \mathcal{A}}
I_{\alpha})$, $\varphi_{|f|} \in \operatorname{\widehat{P}_{b(\sigma,k)}}(\sqcup_{\alpha\in \mathcal{A}}
I_{\alpha})$  and the existence of $C_{\sigma,k}>0$ such that  $\|\varphi_f\|_{b(\sigma,k)}\leq C_{\sigma,k}\|f\|_{C^k}$ for every
$f\in C^m_{\sigma,k}(M)$.

Moreover, \eqref{eq:stepk1} applied to $e=l_\alpha$ and \eqref{eq:stepk2} applied to $e=r_\alpha$ and combined with the inequality
$|(\varphi_f^2)'(x)|\leq C_{\sigma,\vep}\|f\|_{C^1}$, yields either
\begin{align*}
C^+_\alpha(\varphi_f)&=-\frac{1}{m^2_\sigma}\sum_{j=0}^k\binom{k}{j}\partial_\sigma^{(j,k-j)}(f)
\Big(\Gamma_{\frac{m_\sigma-1}{m_\sigma}}^{k,j}\big(\omega_{\sigma}^{2l}G_{\sigma}\big)+
\Gamma_{\frac{m_\sigma-1}{m_\sigma}}^{k,j}\big(\omega_{\sigma}^{2l+1}\overline{G_{\sigma}}\big)\Big),\\
C^-_\alpha(\varphi_f)&=-\frac{1}{m^2_\sigma}\sum_{j=0}^k\binom{k}{j}\partial_\sigma^{(j,k-j)}(f)
\Big(\Gamma_{\frac{m_\sigma-1}{m_\sigma}}^{k,j}\big(\omega_{\sigma}^{2l+1}G_{\sigma}\big)+
\Gamma_{\frac{m_\sigma-1}{m_\sigma}}^{k,j}\big(\omega_{\sigma}^{2l+2}\overline{G_{\sigma}}\big)\Big)
\end{align*}
or  $C^\pm_\alpha(\varphi_f)=0$ whenever the forward semi-orbit of $e$ returns to $I$ for the first time to one of
its ends without visiting singular points. The latter option appears exactly twice.
On the other hand, since every incoming separatix crosses $I$, it follows that every number of the form \eqref{eq:Cpm}
is obtained as $C^\pm_\alpha(\varphi_f)$ for some $\alpha$.
\end{proof}

\subsection{Correcting operators for observables}\label{sec;reduction}
Let us consider the basis $h_1,\ldots, h_{g}$ of  $U_{g+1} \subset H(\pi^{(0)})$, defined in Remark ~\ref{rem:osel},
 such that
$
\lim_{k \rightarrow \infty} \frac{1}{k}\norm{Q(k)h_i} = \lambda_i$  for  $1\leq i \leq g$.
Given $0\leq b<1 $ we choose $2\leq j \leq g+1$ such that
 $\lambda_{j}\leq \lambda_1 b < \lambda_{j-1}$.
Since $h_1,\ldots, h_{j-1}$ is a basis of  $U_{j}$ (see also Remark ~\ref{rem:osel}) and the correction operator
$\mathfrak{h}_j:\pbg(\sqcup_{\alpha \in \mathcal{A}} I_\alpha)\to U_j$ (defined in the proof of Theorem~\ref{thm;correction}) is bounded,
for every $1\leq i<j$ there exists a bounded operator $d_{b,i} : \pbg(\sqcup_{\alpha \in \mathcal{A}} I_\alpha) \rightarrow \R$   such that
\begin{equation}\label{eqn;corhj}
\mathfrak{h}_j(\varphi) = \sum_{i=1}^{j-1} d_{b,i}(\varphi)h_i \quad \text{for every } \varphi \in  \pbg(\sqcup_{\alpha \in \mathcal{A}} I_\alpha).
\end{equation}
By Lemma~7.4 in \cite{Co-Fr}, for every $h_i \in U_{g+1}\subset H(\pi^{(0)})$ ($1\leq i\leq g$) there
exists $f_i\in C^{\infty}(M)$ such that $\varphi_{f_i}=h_i$
and $f_i$ vanishes on $\bigcup_{\sigma\in \mathrm{Fix}(\psi_\R)}U_\sigma$.

Finally, for every $\sigma\in \mathrm{Fix}(\psi_\R)\cap M'$ and any $0\leq k\leq m_\sigma-2$ we define a \emph{correcting operator}
$\mathfrak{R}_{\sigma,k}:C^m_{\sigma,k}\to C^m_{\sigma,k}$ given by
\begin{equation}\label{def;reduction}
\mathfrak{R}_{\sigma,k}(\xi) = \xi - \sum_{i=1}^{j-1} d_{b(\sigma,k),i}(\varphi_\xi)f_i.
\end{equation}
The correcting operator does not change the observable around all fixed points, but it removes the influence of Lyapunov exponents of the K-Z cocycle on the
asymptotic of Birkhoff integrals.
\begin{proposition}\label{prop:xi}
Let $\sigma\in \mathrm{Fix}(\psi_\R)\cap M'$ and $0\leq k\leq m_\sigma-2$. Then for every  $\xi\in C^m_{\sigma,k}$ we have
\begin{align*}
&\limsup_{T \rightarrow \infty} \frac{\log{\big|\int_0^T\mathfrak{R}_{\sigma,k}(\xi)(\psi_tx)\,dt\big|}}{\log T}  \leq b(\sigma,k)\text{ for a.e. }x\in M'; \\
&\limsup_{T \rightarrow \infty} \frac{\log{\big|\int_0^T\mathfrak{R}_{\sigma,k}(\xi)\circ\psi_t\,dt\big|}_{L^1(M')}}{\log T}  \leq b(\sigma,k).
\end{align*}
\end{proposition}
\begin{proof}
In view of Theorem~\ref{thm:phifform},
$\varphi_{\mathfrak{R}_{\sigma,k}(\xi)}\in \pbg(\sqcup_{\alpha \in \mathcal{A}} I_\alpha)$ and $\varphi_{|\mathfrak{R}_{\sigma,k}(\xi)|}\in {\wpb}(\sqcup_{\alpha\in \mathcal{A}}
I_{\alpha})$ for $b=b(\sigma,k)$.
Therefore, by the definition of the correcting operator,
\begin{align*}
\mathfrak{h}_j(\varphi_{\mathfrak{R}_{\sigma,k}(\xi)})&=\mathfrak{h}_j(\varphi_{ \xi})-\sum_{i=1}^{j-1} {d_{b,i}(\varphi_{ \xi})}
\mathfrak{h}_j(\varphi_{f_i})
= \mathfrak{h}_j(\varphi_{\xi})-\sum_{i=1}^{j-1} {d_{b,i}(\varphi_{\xi})}
\mathfrak{h}_j(h_{i})\\
& = \mathfrak{h}_j(\varphi_{\xi})-\sum_{i=1}^{j-1} d_{b,i}(\varphi_{\xi})h_i=0.
\end{align*}
Hence, by Corollary \ref{cor;correction}, for every $\tau >0$ we have
$
\|\mathcal{M}^{(k)}(S(k)\varphi_{\mathfrak{R}_{\sigma,k}(\xi)})\| = O(e^{(\lambda_1 b+\tau)k})$.
Finally, Theorems~\ref{thm:expbL1}~and~\ref{thm;L1bound} applied to $g=\varphi_1$ ($a=(m-2)/2$) and $f = \mathfrak{R}_{\sigma,k}(\xi)$,
complete the proof.
\end{proof}

Note that, in view of \eqref{eqn;lowerbndb}, the inequalities are optimal whenever $\psi_\R$ on $M$ is minimal.
Hence, it follows that the correction provided by the operator $\mathfrak{R}_{\sigma,k}$ is the most optimal one.

\section{Complete power deviation spectrum of Birkhoff integrals}\label{sec:pmt}
In this section by combining previous results, we prove the full deviation spectrum of Birkhoff integrals for locally Hamiltonian flows.

\begin{proof}[Proof of Theorem \ref{theorem;main}]
The proof  splits into five parts.\medskip

\emph{Part I: Deviations around fixed points.}
For every $\sigma\in \mathrm{Fix}(\psi_\R)$ and  $\alpha \in\Z_{\geq 0}\times \Z_{\geq 0}$ with $|\alpha| <m_\sigma-2$,
choose a map $\bar \xi_\sigma^\alpha\in C^m(M)$ supported on the  neighborhood $U_\sigma$ of the fixed point $\sigma$ so that
$\partial_\sigma^\beta (\bar \xi_\sigma^\alpha) =  \delta_{\alpha \beta}$ for all $\beta \in\Z_{\geq 0}\times \Z_{\geq 0}$ with $|\beta| \leq m$,
where $\delta_{\alpha\beta}$ is the Kronecker delta, i.e.\ $\delta_{\alpha\beta}=1$ if $\alpha=\beta$ and $\delta_{\alpha\beta}=0$ if $\alpha\neq\beta$.
By definition, $\bar \xi_\sigma^\alpha\in C_{\sigma, |\alpha|}^m(M)$. Let
$ \xi_\sigma^\alpha:=\mathfrak{R}_{\sigma, |\alpha|}(\bar \xi_\sigma^\alpha)\in C_{\sigma, |\alpha|}^m(M)$ and
 $c_{\sigma,\alpha}(T,x):= \int_0^T \xi_\sigma^\alpha (\psi_t(x)) dt$.
Then, in view of Proposition~\ref{prop:xi} applied to $\xi=\bar \xi_\sigma^\alpha$, for every $\sigma\in \mathrm{Fix}(\psi_\R)\cap M'$ and  $\alpha \in\Z_{\geq 0}\times \Z_{\geq 0}$ with $|\alpha| <m_\sigma-2$
we have \eqref{res;dev1} and \eqref{res;dev2}. Recall that, by Lemma~\ref{lem;partial-invariance}, the corresponding distribution $\partial_\sigma^\alpha$ is  bounded and  $\psi_\R$-invariant.

\medskip

\emph{Part II: Construction of the remainder.}
Let us consider $f_r \in C^m(M)$ given by
\begin{equation}\label{def:f_r}
f = \sum_{\sigma\in \mathrm{Fix}(\psi_\R)}\sum_{\substack{\alpha\in\Z^2_{\geq 0}\\|\alpha| < m_\sigma-2}}\partial^{\alpha}_{\sigma}(f)\xi_\sigma^\alpha+ f_r.
\end{equation}
In view of Theorem~\ref{thm:phifform}, we have $\varphi_{f_r}\in\pog(\sqcup_{\alpha \in \mathcal{A}} I_\alpha)$.
Indeed, let $\{\chi_\sigma:\sigma\in \mathrm{Fix}(\psi_\R)\}\subset C^m(M)$ be a partition of unity such that for any pair of fixed points $(\sigma,\sigma')$ we have
$\chi_\sigma(x)=\delta_{\sigma\sigma'}$ for all $x\in U_{\sigma'}$. Then
\begin{align*}
f_r=\sum_{\sigma\in \mathrm{Fix}(\psi_\R)}\Big(f\cdot\chi_\sigma-\sum_{|\alpha| < m_\sigma-2}\partial^{\alpha}_{\sigma}(f)\xi_\sigma^\alpha\Big)
\end{align*}
and $f_\sigma:=f\cdot\chi_\sigma-\sum_{|\alpha| < m_\sigma-2}\partial^{\alpha}_{\sigma}(f)\xi_\sigma^\alpha$ vanishes on $\bigcup_{\sigma'\in \mathrm{Fix}(\psi_\R)\setminus\{\sigma\}}U_{\sigma'}$ and for every $\beta\in\Z^2_{\geq 0}$ with $|\beta|<m_\sigma-2$
we have
\[\partial_\sigma^\beta(f_\sigma)=\partial_\sigma^\beta(f\cdot\chi_\sigma)-\sum_{|\alpha| < m_\sigma-2}\partial^{\alpha}_{\sigma}(f)\partial_\sigma^\beta(\xi_\sigma^\alpha)=\partial_\sigma^\beta(f)-\sum_{|\alpha| < m_\sigma-2}\delta_{\alpha\beta}\partial^{\alpha}_{\sigma}(f)=0.\]
Therefore, $(f_\sigma)^{(l)}(\sigma)=0$ for all $0\leq l<m_\sigma-2$. As $f_\sigma\in C^m_{\sigma,m_\sigma-2}$, in view of Theorem~\ref{thm:phifform}, it follows that
 for every $\sigma\in \mathrm{Fix}(\psi_\R)\cap M'$ we have $\varphi_{f_\sigma}\in\pog(\sqcup_{\alpha \in \mathcal{A}} I_\alpha)$ and
\[\|\varphi_{f_\sigma}\|_0\leq C\|f_\sigma\|_{C^{m_\sigma-1}}\leq C\|f\|_{C^{m_\sigma-1}}\|\chi_\sigma\|_{C^{m_\sigma-1}}
+C\sum_{|\alpha| < m_\sigma-2}|\partial^{\alpha}_{\sigma}(f)|\|\xi_\sigma^\alpha\|_{C^{m_\sigma-1}}.\]
By definition, for every $\sigma\in \mathrm{Fix}(\psi_\R)$ and $\alpha\in\Z_{\geq 0}^2$ there exists $C_{\sigma,\alpha}>0$ such that
$|\partial^{\alpha}_{\sigma}(f)|\leq C_{\sigma,\alpha}\|f\|_{C^{|\alpha|}}$ for every $f\in C^m(M)$. It follows that there exists another $C>0$ such that
$\|\varphi_{f_r}\|_0\leq C\|f\|_{C^{m-1}}$ for every $f\in C^m(M)$.
\medskip

\emph{Part III: Deviation of the remainder $f_r$.}
Applying Theorem~\ref{thm;correction} to $a=0$, we have a bounded (correction) operator $\mathfrak{h}_{g+1} : \pog(\sqcup_{\alpha \in \mathcal{A}} I_\alpha) \rightarrow U_{g+1} \subset H(\pi^{(0)})$ such that ${\mathfrak{h}_{g+1}}(h) = h$ for every $h \in U_{g+1}$.
Let us consider bounded operators $ {d_i} : \pog(\sqcup_{\alpha \in \mathcal{A}} I_\alpha) \rightarrow \R$ for $1\leq i \leq g$   such that
\begin{equation}\label{eqn;corh}
\mathfrak{h}_{g+1}(\varphi) = \sum_{i=1}^{g} {d_i}(\varphi)h_i  \quad \text{for every } \varphi \in  \pog(\sqcup_{\alpha \in \mathcal{A}} I_\alpha).
\end{equation}
Let $D_i : C^m(M) \rightarrow \R$,  $1\leq i \leq g$ be  operators  given by
\[D_i({f}) = d_i(\varphi_{f_r}), \text{ for } f \in C^m(M).\]
Since $C^m(M)\ni f\mapsto \varphi_{f_r}\in \pog(\sqcup_{\alpha \in \mathcal{A}} I_\alpha)$ and $ {d_i} : \pog(\sqcup_{\alpha \in \mathcal{A}} I_\alpha) \rightarrow \R$ are bounded linear operator, the operators $D_i$ are also bounded.

Recall that we have $f_i\in C^{\infty}(M)$ such that $\varphi_{f_i}=h_i$ for $1\leq i\leq g$. Let us consider $f_e\in C^m(M)$ given by
\begin{equation}\label{def:f_e}
f_{r} = \sum_{i=1}^g  D_{i}(f) f_i + f_e.
\end{equation}
For every $1\leq i\leq g$ let $u_i(T,x):= \int_0^T f_i (\psi_t(x)) dt$.
As $
\lim_{k \rightarrow \infty} \frac{1}{k}\norm{Q(k)h_i} = \lambda_i$  for  $1\leq i \leq g$, in view of Proposition~\ref{thm:specflownonzero}, we have  \eqref{res;dev3} and \eqref{res;dev4} with $\nu_i=\lambda_i/\lambda_1$.

By the definition of $f_e$, we have
\begin{align*}
\varphi_{f_e} =  \varphi_{f_r} - \sum_{i=1}^g  D_{i}(f)\varphi_{f_i}=\varphi_{f_r} - \sum_{i=1}^g  D_{i}(f)h_i.
\end{align*}
As $\varphi_{f_r}\in \pog(\sqcup_{\alpha \in \mathcal{A}} I_\alpha)$, we have $\varphi_{f_e}\in \pog(\sqcup_{\alpha \in \mathcal{A}} I_\alpha)$
and
\begin{align*}
\mathfrak{h}_{g+1}(\varphi_{f_e}) &
=  \mathfrak{h}_{g+1}(\varphi_{f_r}) - \sum_{i=1}^g  D_{i}(f) \mathfrak{h}_{g+1}(h_i)
= \mathfrak{h}_{g+1}(\varphi_{f_r}) - \sum_{i=1}^g  {d_i}(\varphi_{f_r}) h_i = 0.
\end{align*}
By Corollary \ref{cor;correction}, it follows that
\[
\|\mathcal{M}^{(k)}(S(k)\varphi_{f_e})\| = O(e^{\tau k})\text{ for every }\tau >0.
\]
Let $err(f,T,x) = \int_0^T f_e (\psi_t(x)) dt$. If $f_e\neq 0$ then, in view of Proposition~\ref{thm:specflowzeroexp} and Remark~\ref{rem:dodeq},
this gives \eqref{eqn;dev-remain} and \eqref{eqn;dev-remain-lp}.

By \eqref{def:f_r} and \eqref{def:f_e}, we have
\begin{equation}\label{eq:fdecomp}
f = \sum_{\sigma\in \mathrm{Fix}(\psi_\R)}\sum_{\substack{\alpha\in\Z^2_{\geq 0}\\|\alpha| < m_\sigma-2}}\partial^{\alpha}_{\sigma}(f)\xi_\sigma^\alpha+ \sum_{i=1}^g  D_{i}(f) f_i + f_e.
\end{equation}
Passing to the Birkhoff integrals, we obtain \eqref{eqn;dev-complete}.

\medskip

\emph{Part IV: Invariance of distributions.} We need to show that the distributions $D_i$ for $1\leq i\leq g$ are $\psi_\R$-invariant.
By \eqref{eq:fdecomp}, for every $s\in\R$ we have
\[f\circ\psi_s = \sum_{\sigma\in \mathrm{Fix}(\psi_\R)}\sum_{\substack{\alpha\in\Z^2_{\geq 0}\\|\alpha| < m_\sigma-2}}\partial^{\alpha}_{\sigma}(f\circ\psi_s)\xi_\sigma^\alpha+ \sum_{i=1}^g  D_{i}(f\circ\psi_s) f_i + (f\circ\psi_s)_e.\]
Since $\partial^{\alpha}_{\sigma}(f\circ\psi_s)=\partial^{\alpha}_{\sigma}(f)$ (see Lemma~\ref{lem;partial-invariance}), it follows that
\[\bar{f}:=\sum_{i=1}^g  D_{i}(f\circ\psi_s-f) f_i=f\circ\psi_s-f+f_e-(f\circ\psi_s)_e.\]
Note that for any $T>0$ we have
\[\left|\int_0^T(f\circ\psi_s-f)(\psi_tx)\,dt \right| \leq \int_0^s|f(\psi_tx)|\,dt+\int_T^{T+s}|f(\psi_tx)|\,dt\leq 2s\|f\|_{C^0}.\]
In view of \eqref{eqn;dev-remain}, it follows that
\begin{equation}\label{neq:negexpbarf}
\limsup_{T\to+\infty}\frac{\log \left|\int_0^T\bar{f}(\psi_tx)\,dt\right|}{\log T} \leq 0\text{ for a.e. }x\in M'.
\end{equation}
On the other hand,
\[\varphi_{\bar{f}}=\sum_{i=1}^g  D_{i}(f\circ\psi_s-f) \varphi_{f_i}=\sum_{i=1}^g  D_{i}(f\circ\psi_s-f) h_i\in U_{g+1}.\]
Suppose that, contrary to our claim, $D_{i}(f\circ\psi_s-f)\neq  0$ for some $1\leq i\leq g$. As $h_1,\ldots,h_g$ are linearly independent,
\[h:=\varphi_{\bar{f}}=\sum_{i=1}^g  D_{i}(f\circ\psi_s-f) h_i\neq 0.\]
In view of \eqref{neq:posexph}, it follows that
\[\lambda(h):=\lim_{k\to\infty}\frac{\log\|Q(k)h\|}{k}\geq \lambda_g>0.\]
By Proposition~\ref{thm:specflownonzero}, this yields
\[\limsup_{T\to+\infty}\frac{\log \left|\int_0^T\bar{f}(\psi_t x)\,dt \right|}{\log T} =\frac{\lambda(h)}{\lambda_1}>0\text{ for a.e. }x\in M',\]
contrary to \eqref{neq:negexpbarf}.
Consequently, $D_{i}(f\circ\psi_s)=D_i(f)$ for all $1\leq i\leq g$ and $s\in\R$.
\medskip

\emph{Part V: Lower bounds.}
Suppose that $M' = M$. Let us consider any $f\in C^m_{\sigma,l}$ with
$f^{(l)}(\sigma)\neq 0$ for $0\leq l< m_\sigma-2$ and $\sigma\in \mathrm{Fix}(\psi_\R)$. Then $\varphi = \varphi_f\in {\pbg}(\sqcup_{\alpha\in \mathcal{A}}
I_{\alpha})$ with $b=b(\sigma,l)>0$. The purpose of this part is to show \eqref{eqn;lowerbndb}.

In view of Theorem~\ref{thm:phifform} (the last sentence) combined with Lemma~\ref{lem:nonnegsum}
applied to $a=\frac{m_\sigma-1}{m_\sigma}$ and $a_i=\partial^{(i,l-i)}_\sigma(f)$ for $0\leq i\leq l$ (at least one of them is non-zero, since $f^{(l)}(\sigma)\neq 0$),
there exists $\alpha\in \mathcal{A}$ such that  $C^+_\alpha(\varphi_f)\neq 0$ or $C^-_\alpha(\varphi_f)\neq 0$.
We focus only on the case $C^+_\alpha(\varphi_f)\neq 0$. In the latter case, the proof runs similarly.

For every $k\geq 1$ let us consider the interval $\big(l^{(k)}_\alpha,l^{(k)}_\alpha+\vep|I^{(k)}_\alpha|\big]$ with
\[\vep:=\left(\frac{|C^+_\alpha|}{2^{\frac{4+4b}{b}}\kappa^{1+b}d\zeta(1+b)p_b(\varphi)}\right)^{1/(1+b)}.\]
As $d,\kappa,\zeta(b+1) \geq 1$, 
  by \eqref{neq:Cpa}, we have $ 16\vep^b \leq 1$.
In view of Proposition~\ref{prop:renormbelow}, for $k$ large enough ($|I^{(k)}|\leq \delta$) and for every
$x\in \big(l^{(k)}_\alpha,l^{(k)}_\alpha+\vep|I^{(k)}_\alpha|\big]$, we have
\begin{equation*}
\big|(x-l^{(k)}_\alpha)^{1+b}(S(k)\varphi)'(x)\big| \geq \frac{|C^+_\alpha|}{2}-2^{2+b}\kappa^{1+b}d\zeta(1+b)p_b(\varphi)
\left(\frac{\vep|I^{(k)}_\alpha|}{|I^{(k)}_\alpha|}\right)^{1+b}\geq \frac{|C^+_\alpha|}{4} >0.
\end{equation*}
In view of Lemma~\ref{lem:belowest}, there exists an interval $\widehat{J}\subset \big(l^{(k)}_\alpha,l^{(k)}_\alpha+\vep|I_\alpha^{(k)}|\big]$
such that
\begin{equation}\label{eq:skbelow}
|\widehat{J}|\geq \frac{\vep|I^{(k)}_\alpha|}{4}\text{ and }\ |(S(k)\varphi)(x)| \geq \frac{|C^+_\alpha|}{4(\vep|I^{(k)}_\alpha|)^{b}}
\geq \frac{|C^+_\alpha|}{|I^{(k)}_\alpha|^{b}}
\text{ for }x\in \widehat{J}. 
\end{equation}
Finally we can choose an interval $J^{(k)}\subset\widehat{J}$ such that
\begin{gather}
\label{eq:Jk1}
|J^{(k)}|\geq \vep|I^{(k)}_\alpha|/(2^4\kappa);\\
\label{eq:Jk2}
\dist(J^{(k)},End(T^{(k)}))\geq \vep|I^{(k)}_\alpha|/2^4;\\
\label{eq:Jk3}
\dist((T^{(k)})^{-1}J^{(k)},End(T^{(k)}))\geq \vep|I^{(k)}_\alpha|/(2^4\kappa).
\end{gather}
Let us consider the set
\begin{align*}
B_k:&=\{T^g_t(x,0):x\in (T^{(k)})^{-1}J^{(k)},0\leq t<(S(k)g)(x)\}\\
&
=\{T^g_t(x,0):x\in J^{(k)},-(S(k)g)((T^{(k)})^{-1}x)\leq t<0\}.
\end{align*}
As $g\geq \underline{g}>0$, by \eqref{eq:skbelow}, \eqref{eq:Jk1}
 and \eqref{eq:qlambda}, we have
\[Leb(B_k)=\int_{(T^{(k)})^{-1}J^{(k)}}S(k)g(x)\,dx\geq \underline{g}|J^{(k)}|\min_{\beta\in\mathcal{A}}Q_\beta(k)\geq
\frac{\delta\vep\underline{g}}{2^4\kappa^2}|I|.\]
For every $(x',r')=T^g_r(x,0)\in B_k$ let
\begin{align*}
\tau_k^0 &=\tau_k^0(x',r'):= (S(k)g)(x)-r\text{ and }\\
 \tau_k^1&=\tau_k^1(x',r'):= (S(k)g)(x)+(S(k)g)(T^{(k)}x)-r.
\end{align*}
Then $T^g_{\tau_k^0}(x',r')=(T^{(k)}x,0)$ and
\begin{align*}
\int_0^{\tau_k^1}f(T^g_t(x',r'))\,dt-\int_0^{\tau_k^0}f(T^g_t(x',r'))\,dt& =\int_0^{(S(k)g)(T^{(k)}x)}f(T^g_t(T^{(k)}x,0))\,dt\\
&=(S(k)\varphi_f)(T^{(k)}x).
\end{align*}
As $x\in (T^{(k)})^{-1}J^{(k)}$, by \eqref{eq:skbelow}, we have $|(S(k)\varphi_f)(T^{(k)}x)|\geq {|C^+_\alpha|}/|I^{(k)}_\alpha|^{b}$.
It follows that
\begin{equation}\label{eq:maxT}
\max\Big\{\Big|\int_0^{\tau_k^1}f(T^g_t(x',r'))\,dt\Big|,\Big|\int_0^{\tau_k^0}f(T^g_t(x',r'))\,dt\Big|\Big\}\geq \frac{|C^+_\alpha|}{2|I^{(k)}_\alpha|^{b}}.
\end{equation}
Choose $\tau_k=\tau_k(x',r')$ among $\tau_k^0$ and $\tau_k^1$ such that
\[
\left|\int_0^{\tau_k}f(T^g_t(x',r'))\,dt\right| = \max\Big\{\Big|\int_0^{\tau_k^1}f(T^g_t(x',r'))\,dt\Big|,\Big|\int_0^{\tau_k^0}f(T^g_t(x',r'))\,dt\Big|\Big\}.
\]
As $x\in (T^{(k)})^{-1}J^{(k)}$, in view of \eqref{eq:Jk2}, \eqref{eq:Jk3} and \eqref{eqn;upperboundvarphi}, we have
\begin{align*}
|(S(k)g)((T^{(k)})^{-1}x)| &\leq \|\mathcal{M}^{(k)}(S(k)g)\| +p_a(S(k)g)O(|I^{(k)}|^{-a})\\
|(S(k)g)(x)| &\leq \|\mathcal{M}^{(k)}(S(k)g)\| +p_a(S(k)g)O(|I^{(k)}|^{-a}).
\end{align*}
Moreover, by \eqref{eqn;Mknorm}, \eqref{nase} and \eqref{eqn;renormpaos2}, we have
\[\|\mathcal{M}^{(k)}(S(k)g)\|\leq \frac{2\kappa}{|I^{(k)}|}\|S(k)g\|_{L^1(I^{(k)})}\leq \frac{2\kappa\|g\|_{L^1(I^{(0)})}}{|I^{(k)}|},\quad p_a(S(k)g)\leq O(p_a(g)).\] 
Therefore,
\begin{align*}
|(S(k)g)((T^{(k)})^{-1}x)| \leq O(|I^{(k)}|^{-1}), \quad
|(S(k)g)(x)| \leq O(|I^{(k)}|^{-1}).
\end{align*}
Hence there exists $C>0$ such that for every $k\geq 1$ and $(x',r')\in B_k$ we have
$\tau_k(x',r')\leq  C |I^{(k)}|^{-1}$.
In view of \eqref{eq:maxT}, it follows that for every $(x',r')\in B_k$ we have
\[\frac{\log|\int_0^{\tau_k}f(T^g_t(x',r'))\,dt|}{\log \tau_k}\geq \frac{\log (|C^+_\alpha||I^{(k)}|^{-b}/2)}{\log (C|I^{(k)}|^{-1})}.\]%
Since $(B_k)_{k\geq 1}$ is a sequence of asymptotically invariant sets (i.e.\ for every $t\in\R$ we have $Leb(B_k\triangle T^g_t B_k)\to 0$ as $k\to\infty$)
and their measures are separated from zero, by the ergodicity of the flow, a.e.\ $(x,r)\in I^g$ belongs to $B_k$ for infinitely many $k$.
It follows that for a.e.\ $(x,r)\in I^g$ we have
\begin{align*}
\limsup_{T\to+\infty}\frac{\log|\int_0^{T}f(T^g_t(x,r))\,dt|}{\log T}&\geq \limsup_{k\to+\infty}\frac{\log|\int_0^{\tau_k}f(T^g_t(x,r))\,dt|}{\log \tau_k}\\
&\geq \lim_{k\to+\infty}\frac{\log (|C^+_\alpha||I^{(k)}|^{-b}/2)}{\log (C|I^{(k)}|^{-1})}=b.%
\end{align*}

Finally, \eqref{eqn;lowerbnc} follows directly from \eqref{res;dev1} and \eqref{eqn;lowerbndb}, since $c_{\sigma,\alpha}(T,x)= \int_0^T \xi_\sigma^\alpha (\psi_t(x)) dt$ and $\xi_\sigma^\alpha\in C_{\sigma, |\alpha|}^m(M)$ with $\partial_\sigma^{\alpha}(\xi_\sigma^\alpha)=1$.
\end{proof}

\section*{Acknowledgements}
M.K. would like to thank the Center of Excellence ``Dynamics, mathematical analysis and artificial intelligence'' at the Nicolaus Copernicus University in Toru\'n for hospitality during his post-doc grant.
Research was partially supported by  the Narodowe Centrum Nauki Grant 2017/27/B/ST1/00078.

\appendix

\section{Proof of Theorem~\ref{thm;FDCRTC}}\label{sec:App1}
We review the natural extension of the Rauzy-Veech induction and prove the full measure of IETs satisfying \ref{FDC}.

\subsection{Extension of the Rauzy-Veech induction}
Let $\mathcal{G} \subset \mathcal{S}^0_{\mathcal{A}}$ be a Rauzy class and set
$\Delta^{\mathcal{A}} := \{\lambda \in \R_{>0}^{\mathcal{A}} : |\lambda|=1 \}.$
Let ${\mathcal{R}}: \GG \times \R_{>0}^{\mathcal{A}} \rightarrow \GG
\times \R_{>0}^{\mathcal{A}}$ be the standard Rauzy-Veech map defined in \S\ref{sec;RVI} by
\[\mathcal{R}(\pi,\lambda)=(\widetilde{\pi},\widetilde{\lambda}),\text{ where }\widetilde{\lambda}=A^{-1}(\pi,\lambda)\text{
and $\widetilde{\pi}$ is given by \eqref{def:pi}.}\]
Then we
define (normalized) Rauzy-Veech renormalization
\[\widetilde{\mathcal{R}}: \GG \times \Delta^{\mathcal{A}} \rightarrow \GG
\times \Delta^{\mathcal{A}}, \quad \widetilde{\mathcal{R}}(\pi,\lambda) = (\tilde \pi,\tilde \lambda / |\tilde \lambda|).\]
By Veech \cite{Ve1}, there exists an $\widetilde{\mathcal{R}}$-invariant ergodic recurrent measure $\mu_{\mathcal{G}}$ which is equivalent to the product of the counting measure on $\mathcal{G}$ and the Lebesgue measure on $\Delta^{\mathcal{A}}$. For every $\pi \in \mathcal{S}^0_{\mathcal{A}}$, let
\[
\Theta_{\pi}:= \Big\{\tau \in \R^\mathcal{A} : \sum_{\pi_0(\alpha) \leq k}\tau_\alpha>0, \sum_{\pi_1(\alpha) \leq k}\tau_\alpha<0 \text{ for } 1 \leq k \leq d\Big\}
\]
and let
\[X(\mathcal{G}) := \bigcup_{\pi \in \mathcal{G}}\left\{(\pi,\lambda,\tau) \in \{\pi\} \times
\Delta^{\mathcal{A}} \times \Theta_{\pi} : \langle \lambda, \Omega_\pi \tau \rangle = 1 \right\}.\]
Then  the natural (invertible) extension of $\widetilde{\mathcal{R}}$ is of the form
\[
\widehat{\mathcal{R}}:X(\mathcal{G})\to X(\mathcal{G}),\quad
\widehat{\mathcal{R}}(\pi,\lambda,\tau) = \left(\tilde\pi, \frac{A^{-1}(\pi,\lambda)\lambda}{|A^{-1}(\pi,\lambda)\lambda|},|A^{-1}(\pi,\lambda)\lambda|A^{-1}(\pi,\lambda)\tau\right).
\]
The natural extension, constructed by  Veech in \cite{Ve1}, of the measure $\mu_{\mathcal{G}}$ on $X(\GG)$ is denoted by $\widehat \mu_{\mathcal{G}}$. Then $\widehat \mu_{\mathcal{G}}$ is $\widehat{\mathcal{R}}$-invariant and
$\widehat{\mathcal{R}}$ is ergodic and recurrent with respect to $\widehat \mu_{\mathcal{G}}$.

We extend the map $A: \mathcal{G} \times \Delta^{\mathcal{A}}  \rightarrow   SL_\mathcal{A}(\Z)$ defined in \S\ref{sec;RVI} to $\widehat{A}: X(\mathcal{G}) \rightarrow SL_\mathcal{A}(\Z)$ given by
$\widehat{A}(\pi,\lambda,\tau) = {A}(\pi,\tau)$.
Let us consider the extended cocycle $\widehat{A}:\Z \times X(\mathcal{G}) \rightarrow SL_\mathcal{A}(\Z)$
\[
\widehat{A}^{(n)}(\pi,\lambda,\tau) =
\begin{cases} \widehat{A}(\pi,\lambda,\tau)\cdot \widehat{A}(\widehat{\mathcal{R}}(\pi,\lambda,\tau)) \cdots  \widehat{A}(\widehat{\mathcal{R}}^{n-1}(\pi,\lambda,\tau))  & \text{ if } n \geq 0\\
\widehat{A}(\widehat{\mathcal{R}}^{-1}(\pi,\lambda,\tau))\cdot \widehat{A}(\widehat{\mathcal{R}}^{-2}(\pi,\lambda,\tau)) \cdots   \widehat{A}(\widehat{\mathcal{R}}^{-n}(\pi,\lambda,\tau)) & \text{ if } n < 0.
\end{cases}
\]
Then
\begin{equation}\label{eq:An>0}
\widehat{A}^{(n)}(\pi,\lambda,\tau) = {A}^{(n)}(\pi,\tau) \text{ if } n \geq 0.
\end{equation}

Let $Y \subset X(\mathcal{G})$  be a subset such that $0 < \widehat \mu_{\mathcal{G}}(Y) < \infty$. For a.e $(\pi,\lambda,\tau) \in Y$, let $r(\pi,\lambda,\tau) \geq 1$ by the first return time of $(\pi,\lambda,\tau)$ for the map $\widehat{\mathcal{R}}$.
Denote by $\widehat{\mathcal{R}}_Y : Y \rightarrow Y$ the induced map and by $\widehat{{A}}_Y : Y \rightarrow SL_\mathcal{A}(\Z)$ the induced cocycle
\[\widehat{\mathcal{R}}_Y(\pi,\lambda,\tau) = \widehat{\mathcal{R}}^{r(\pi,\lambda,\tau)}(\pi,\lambda,\tau), \quad \widehat{{A}}_Y(\pi,\lambda,\tau) = \widehat{{A}}^{(r(\pi,\lambda,\tau))}(\pi,\lambda,\tau) \]
for a.e $(\pi,\lambda,\tau) \in Y$. Let $\widehat{\mu}_Y$ be the restriction of $\widehat{\mu}_{\mathcal{G}}$ to $Y$.

\subsection{Oseledets splitting}
Assume that $\log\|\widehat{A}_Y\|$ and $\log\|\widehat{A}^{-1}_Y\|$ are $\widehat{\mu}_Y$-integrable. By the Oseledets multiplicative theorem, the symplecticity of $\widehat{A}_Y$ (see \cite{Zor}) and
the simplicity of spectrum (see \cite{Av-Vi}), there exist
 Lyapunov exponents $\lambda_1>\ldots>\lambda_g>\lambda_{g+1}=0$ such that for $\widehat{\mu}_Y$-a.e. $(\pi,\lambda,\tau) \in Y$ we have a splitting
\[\R^{\mathcal{A}} = \bigoplus_{1 \leq i \leq g+1} F_i(\pi,\lambda,\tau)\oplus \bigoplus_{1 \leq i \leq g} F_{-i}(\pi,\lambda,\tau),\]
for which
\begin{gather}
\label{eq:F1}
\lim_{n\to \pm \infty}\frac{1}{n}\log\|\widehat{A}^{(n)}_Y(\pi,\lambda,\tau)^t\upharpoonright_{F_i(\pi,\lambda,\tau)}\| = \lambda_i \ \text{ if }
1\leq i \leq g+1\\
\label{eq:F2}
\lim_{n\to \pm \infty}\frac{1}{n}\log\|\widehat{A}^{(n)}_Y(\pi,\lambda,\tau)^t\upharpoonright_{F_{-i}(\pi,\lambda,\tau)}\| = -\lambda_i \ \text{ if } 1\leq i \leq g \\
\label{eq:F3}
\widehat{A}^{(n)}_Y(\pi,\lambda,\tau)^tF_i(\pi,\lambda,\tau)=F_i(\widehat{\mathcal{R}}^n_Y(\pi,\lambda,\tau))
\end{gather}
for all $i\in\{-g,\ldots,-1,1,\ldots,g+1\}$ and $n\in\Z$ and
\[\dim F_{\pm i}(\pi,\lambda,\tau)=1\quad\text{for}\quad i=1,\ldots,g.\]
Moreover, for every partition $\{I_1,I_2\}$ of the set $\{-g,\ldots,-1,1,\ldots,g+1\}$ we have
\begin{equation}\label{eq:F4}
\lim_{n\to \pm \infty}\frac{1}{n}\log\Big|\sin\angle\Big(\bigoplus_{i\in I_1}F_i(\pi,\lambda,\tau),\bigoplus_{i\in I_2}F_i(\pi,\lambda,\tau)\Big)\Big| =0,
\end{equation}
and
\begin{equation}\label{def;Hpi}
H(\pi) := \bigoplus_{i\neq g+1}F_i(\pi,\lambda,\tau).
\end{equation}
For every $1\leq j\leq g+1$ let
\begin{align*}
E_j(\pi,\lambda,\tau)&:=\bigoplus_{j\leq i\leq g+1}F_i(\pi,\lambda,\tau)\oplus \bigoplus_{1\leq i\leq g}F_{-i}(\pi,\lambda,\tau)\\
U_j(\pi,\lambda,\tau)&:=\bigoplus_{1\leq i<j}F_i(\pi,\lambda,\tau)\subset H(\pi).
\end{align*}
Then
\[E_j(\pi,\lambda,\tau)\oplus U_j(\pi,\lambda,\tau)=\R^{\mathcal{A}}.\]
By \eqref{eq:F3}, for every $n\in\Z$ we have
\begin{align*}
\widehat{A}^{(n)}_Y(\pi,\lambda,\tau)^tE_j(\pi,\lambda,\tau)=E_j(\widehat{\mathcal{R}}^n_Y(\pi,\lambda,\tau)),\
 \widehat{A}^{(n)}_Y(\pi,\lambda,\tau)^tU_j(\pi,\lambda,\tau)=U_j(\widehat{\mathcal{R}}^n_Y(\pi,\lambda,\tau)).
\end{align*}
In view of \eqref{eq:F1}, \eqref{eq:F2} and \eqref{eq:F4}, for every $1\leq j\leq g+1$ we have
\begin{gather}
\label{eq:lae}
\lim_{n\to +\infty}\frac{1}{n}\log\|\widehat{A}^{(n)}_Y(\pi,\lambda,\tau)^t\upharpoonright_{E_j(\pi,\lambda,\tau)}\| = \lambda_j\\
\label{eq:lagamma}
\lim_{n\to +\infty}\frac{1}{n}\log\|\widehat{A}^{(-n)}_Y(\pi,\lambda,\tau)^t\upharpoonright_{U_j(\pi,\lambda,\tau)}\| = -\lambda_{j-1}\text{ if } j\geq 2\\
\label{eq:laangle}
\lim_{n\to \pm \infty}\frac{1}{n}\log \left|\sin \angle \left(E_j(\widehat{\mathcal{R}}^n_Y(\pi,\lambda,\tau)), U_j(\widehat{\mathcal{R}}^n_Y(\pi,\lambda,\tau))\right) \right| =0\text{ if }j\geq 2.
\end{gather}

\subsection{Proof of Theorem~\ref{thm;FDCRTC}}
The arguments used in the proof runs similarly to that used to prove Theorem 3.8 in \cite{Fr-Ul2}. We will omit some repetitive arguments.
\begin{proof}[Proof of Theorem~\ref{thm;FDCRTC}]

Let us consider a subset $Y\subset X(\mathcal{G})$ which satisfies
the assumptions below:
\begin{itemize}
\item[$(i)$] the projection $\underline{Y}$ of $Y$ on
$\mathcal{G}\times\Lambda^\mathcal{A}$ is  precompact with respect to the Hilbert metric;
\item[$(ii)$] there exists $0<\delta<1$ such that for every $(\pi,\lambda,\tau)\in Y$ we have
\begin{equation*}
\min \Big\{\Big\{\sum_{\pi_0(\alpha)\leq k}\tau_\alpha:1\leq k<d\Big\}\cup\{(\Omega_\pi(\tau))_\alpha:\alpha\in\mathcal
A \} \Big\}>\delta\max\{(\Omega_\pi(\tau))_\alpha:\alpha\in\mathcal A\};
\end{equation*}
\item[$(iii)$] $\widehat{\mu}_Y$ is finite;
\item[$(iv)$] the functions $\log\|\widehat{A}_Y\|$ and $\log\|\widehat{A}_Y^{-1}\|$ are $\widehat{\mu}_Y$-integrable.
\end{itemize}
\medskip

\noindent\emph{Acceleration.} In view of \eqref{eq:lagamma} and \eqref{eq:lae}, for every $\tau>0$  the maps
\begin{gather*}
Y\ni(\pi,\lambda,\tau)\mapsto\sup_{n\geq 0}e^{-(\lambda_j-\tau)n}
\|\widehat{A}^{(n)}_Y(\pi,\lambda,\tau)^t\upharpoonright_{E_j(\pi,\lambda,\tau)}\|\in \R\text{ for }1\leq j\leq g+1,\\
Y\ni(\pi,\lambda,\tau)\mapsto\sup_{n\geq 0}
e^{(\lambda_{j-1}+\tau)n}\|\widehat{A}^{(-n)}_Y(\pi,\lambda,\tau)^t\upharpoonright_{U_j(\pi,\lambda,\tau)}\|\in
\R\text{ for }2\leq j\leq g+1
\end{gather*}
are a.e.\ defined and measurable.
Therefore, there exists a closed  subset $K \subset Y$ with
$\widehat{\mu}_Y(K)/\widehat{\mu}_Y(Y)>1-\tau/2$ and a constant $C>0$
such that if $(\pi,\lambda,\tau)\in K$ then for every $n\geq 0$  we have
\begin{gather}
\label{eq:ac}
\|\widehat{A}^{(n)}_Y(\pi,\lambda,\tau)^t\upharpoonright_{E_j(\pi,\lambda,\tau)}\|\leq Ce^{(\lambda_j+\tau)n}\text{ for }1\leq j\leq g+1,\\
\label{eq:ac1}
\|(\widehat{A}^{(n)}_Y(\mathcal{R}_Y^{-n}(\pi,\lambda,\tau))^t\upharpoonright_{U_j(\pi,\lambda,\tau)})^{-1}\|\leq Ce^{(-\lambda_{j-1}+\tau)n}\text{ for }2\leq j\leq g+1.
\end{gather}

Let $\widehat{\mathcal{R}}_K:K\to K$ be the induced map  and let $\widehat{A}_K:K\to SL_\mathcal{A}(\Z)$
be the induced cocycle, i.e.
\[\widehat{\mathcal{R}}_K(\pi,\lambda,\tau)=\widehat{\mathcal{R}}^{r_K(\pi,\lambda,\tau)}_Y(\pi,\lambda,\tau),\] where $r_K(\pi,\lambda,\tau)\geq
1$
is the first return time of $(\pi,\lambda,\tau)\in K$ to $K$ for the
map $\widehat{\mathcal{R}}_Y$ and
\[\widehat{A}_K^{(n)}=\widehat{A}_Y^{(r_K^{(n)})}\text{ for every }n\geq 0,\]
where $r_K^{(n)}:=\sum_{0\leq
i<n}r_K\circ \widehat{\mathcal{R}}_K^i$ for every $n\geq 0$.
Then
\begin{equation}\label{eq:retK}
\frac{r_K^{(n)}(\pi,\lambda,\tau)}{n}\to\frac{\widehat{\mu}_Y(Y)}{\widehat{\mu}_Y(K)}\in(1,1+\tau)\text{ for a.e. }(\pi,\lambda,\tau)\in K.
\end{equation}
In view of \eqref{eq:ac} and \eqref{eq:ac1}, for every $(\pi,\lambda,\tau)\in K$,
\begin{gather}
\label{eq:acK}
\|\widehat{A}^{(n)}_K(\pi,\lambda,\tau)^t\upharpoonright_{E_j(\pi,\lambda,\tau)}\|\leq Ce^{(\lambda_j+\tau) r_K^{(n)}(\pi,\lambda,\tau)}\text{ for }
1\leq j\leq g+1,\\
\label{eq:acK1}
\|(\widehat{A}^{(n)}_K(\mathcal{R}_K^{-n}(\pi,\lambda,\tau))^t\upharpoonright_{U_j(\pi,\lambda,\tau)})^{-1}\|\leq Ce^{(-\lambda_{j-1}+\tau) r_K^{(n)}(\mathcal{R}_K^{-n}(\pi,\lambda,\tau))}
\end{gather}
for $2\leq j\leq g+1$.
Moreover, for a.e.\ $(\pi,\lambda,\tau)\in K$,
\begin{equation}\label{eq:maxlap}
\lim_{n\to+\infty}\frac{1}{n}\log\|\widehat{A}^{(n)}_K(\pi,\lambda,\tau)\|=\lambda_1\frac{\widehat{\mu}_Y(Y)}{\widehat{\mu}_Y(K)}\in(\lambda_1 ,\lambda_1
(1+\tau)).
\end{equation}
Since the maps $\log\|\widehat{A}_K\|$ and
$\log\|\widehat{A}_K^{-1}\|$ are $\widehat{\mu}_K$-integrable, for a.e.\
$(\pi,\lambda,\tau)\in K$,
\begin{equation}\label{eq:zzero}
\lim_{n\to+\infty}\frac{1}{n}\log\|\widehat{A}_K(\widehat{\mathcal{R}}_K^n(\pi,\lambda,\tau))\|=0.
\end{equation}

By the ergodicity of $\widehat{\mathcal{R}}:X(\mathcal{G})\to X(\mathcal{G})$, for a.e.\
$(\pi,\lambda,\tau)\in X(\mathcal{G})$
\begin{align}\label{eq:n1}
\begin{split}
&\text{there exists }n_1\geq 0\text{ such that }\widehat{\mathcal{R}}^{n_1}(\pi,\lambda,\tau)\in K\\
&\text{and $\widehat{\mathcal{R}}^{n_1}(\pi,\lambda,\tau)$ satisfies \eqref{eq:laangle}, \eqref{eq:retK}, \eqref{eq:maxlap} and \eqref{eq:zzero}}.
\end{split}
\end{align}
By Fubini argument, there exists a measurable subset $\Xi\subset
\mathcal{G}\times\Lambda^{\mathcal A}$ such that
$\mu_\mathcal{G}(\mathcal{G}\times\Lambda^{\mathcal A}\setminus \Xi)=0$
and for every $(\pi,\lambda)\in \Xi$, there exists $\tau\in\Theta_\pi$
such that $(\pi,\lambda,\tau)\in X(\mathcal{G})$ satisfies
\eqref{eq:n1}.
\medskip

\noindent\emph{Full measure.} We now show that every $(\pi,\lambda)\in \Xi$ satisfies the \ref{FDC}.

Suppose that $(\pi,\lambda)\in \Xi$ and $(\pi,\lambda,\tau)\in
X(\mathcal{G})$ satisfies \eqref{eq:n1}. Then the  accelerating sequence $(n_k)_{k\geq 0}$ required by Definition~\ref{def;FDC} is defined as follows:
\begin{itemize}
\item $n_0=0$;
\item for $k\geq 1$ we take $n_k$ so that
$\widehat{\mathcal{R}}^{n_k}(\pi,\lambda,\tau)=\widehat{\mathcal{R}}_K^{k-1}\widehat{\mathcal{R}}^{n_1(\pi,\lambda,\tau)}(\pi,\lambda,\tau)$.
\end{itemize}
Since $(\pi,\lambda,\tau)$ is Oseledets generic, $(\pi,\lambda)$ is also Oseledets generic with the Oseledets filtration
\begin{align*}
\{0\}=E_{0}(\pi,\lambda)\subset E_{-1}(\pi,\lambda)\subset\ldots\subset E_{-g}(\pi,\lambda)\subset E_{cs}(\pi,\lambda)\\
=E_{g+1}(\pi,\lambda)\subset E_{g}(\pi,\lambda)\subset\ldots\subset E_{1}(\pi,\lambda)=\Gamma
\end{align*}
given by
\[E_{j}(\pi,\lambda):=E_{j}(\pi,\lambda,\tau)\text{ for }j=-1,-2,\ldots, -g, g+1, g,\ldots, 2,1.\]
We can define a complementary filtration $\{0\}=U_1\subset U_2\subset\ldots\subset U_g\subset U_{g+1}$ by
\[U_{j}=U_{j}(\pi,\lambda,\tau)\text{ for }1\leq j\leq g+1.\]
Then for every $k\geq 1$,
\[E_j^{(k)}=E_j(\widehat{\mathcal{R}}_K^{k-1}(\widehat{\mathcal{R}}^{n_1}(\pi,\lambda,\tau)))\text{ and }U_j^{(k)}=U_j(\widehat{\mathcal{R}}_K^{k-1}(\widehat{\mathcal{R}}^{n_1}(\pi,\lambda,\tau))).\]
By the definition of $Q$ and \eqref{eq:An>0},
$Q(k,l)=\widehat{A}_K^{(l-k)}(\widehat{\mathcal{R}}_K^{k-1}(\widehat{\mathcal{R}}^{n_1}(\pi,\lambda,\tau)))^t$ for $1\leq k\leq l$,
so
\begin{align*}
\|Q|_{E_j}(k,l)\|&=\big\|\widehat{A}_K^{(l-k)}(\widehat{\mathcal{R}}_K^{k-1}(\widehat{\mathcal{R}}^{n_1}(\pi,\lambda,\tau)))^t\upharpoonright_{E_j(\widehat{\mathcal{R}}_K^{k-1}(\widehat{\mathcal{R}}^{n_1}(\pi,\lambda,\tau)))}\big\|\\
\|Q|_{U_j}(k,l)^{-1}\|&=\big\|(\widehat{A}_K^{(l-k)}(\widehat{\mathcal{R}}_K^{k-1}(\widehat{\mathcal{R}}^{n_1}(\pi,\lambda,\tau)))^t\upharpoonright_{U_j(\widehat{\mathcal{R}}_K^{k-1}(\widehat{\mathcal{R}}^{n_1}(\pi,\lambda,\tau)))})^{-1}\big\|\\
&=\big\|(\widehat{A}_K^{(l-k)}(\widehat{\mathcal{R}}_K^{-(l-k)}(\widehat{\mathcal{R}}_K^{l-1}\circ\widehat{\mathcal{R}}^{n_1}(\pi,\lambda,\tau)))^t\upharpoonright_{U_j(\widehat{\mathcal{R}}_K^{k-1}(\widehat{\mathcal{R}}^{n_1}(\pi,\lambda,\tau)))})^{-1}\big\|.
\end{align*}
Since
$\widehat{\mathcal{R}}_K^{k-1}\circ\widehat{\mathcal{R}}^{n_1}(\pi,\lambda,\tau),\widehat{\mathcal{R}}_K^{l-1}\circ\widehat{\mathcal{R}}^{n_1}(\pi,\lambda,\tau)\in K$ for every $1\leq k\leq l$, by \eqref{eq:acK} and \eqref{eq:acK1}, we
have
\begin{gather*}
\|Q|_{E_j}(k,l)\|\leq Ce^{(\lambda_j+\tau) r_K^{(l-k)}(\widehat{\mathcal{R}}_K^{k-1}\circ\widehat{\mathcal{R}}^{n_1}(\pi,\lambda,\tau))}
\text{ for }
1\leq j\leq g+1,\\
\|Q|_{U_j}(k,l)^{-1}\|\leq Ce^{(-\lambda_{j-1}+\tau) r_K^{(l-k)}(\widehat{\mathcal{R}}_K^{k-1}\circ\widehat{\mathcal{R}}^{n_1}(\pi,\lambda,\tau))}
\text{ for }
2\leq j\leq g+1.
\end{gather*}
Let us consider a sequence $(r_n)_{n\geq 0}$ given by $r_0=1$ and  for $n\geq 1$,
$$r_n=r_K(\widehat{\mathcal{R}}_K^{n-1}\circ\widehat{\mathcal{R}}^{n_1}(\pi,\lambda,\tau))).$$
 Then for all $1\leq k\leq l$ we have $r(k,l)=r_K^{(l-k)}(\widehat{\mathcal{R}}_K^{k-1}\circ\widehat{\mathcal{R}}^{n_1}(\pi,\lambda,\tau))$, so
\[\|Q|_{E_j}(k,l)\|\leq Ce^{(\lambda_j+\tau) r(k,l)}\text{ and }\|Q|_{U_j}(k,l)^{-1}\|\leq Ce^{(-\lambda_{j-1}+\tau) r(k,l)}.\]
Both inequalities extend to the case where $k = 0$. Then the constant $C$ must be multiplied additionally by $\max\{\|Q(0,1)\|,e^{\lambda_1}\|Q(0,1)^{-1}\|\}$.
This gives \eqref{def;sdc1} and \eqref{def;sdc12}.
\medskip

\noindent\emph{Return time estimate.}
Since $\widehat{\mathcal{R}}^{n_1}(\pi,\lambda,\tau)\in K$ satisfies \eqref{eq:retK}, \eqref{eq:maxlap} and \eqref{eq:zzero}, we also have
\begin{gather}
\label{eq:r0n}
\frac{r(0,n)}{n}=\frac{1+r_K^{(n-1)}(\mathcal{R}^{n_1}(\pi,\lambda,\tau))}{n}\to\frac{\widehat{\mu}_Y(Y)}{\widehat{\mu}_Y(K)}\in(1,1+\tau),\\
\nonumber
\frac{1}{n}\log\|Z(n+1)\|=\frac{1}{n}\log\|\widehat{A}_K(\widehat{\mathcal{R}}_K^{n-1}\circ\mathcal{R}^{n_1}(\pi,\lambda,\tau))\|\to 0,\\
\nonumber
\frac{1}{n}\log\|Q(1,n)\|=\frac{1}{n}\log\|\widehat{A}^{(n-1)}_K(\mathcal{R}^{n_1}(\pi,\lambda,\tau))\|\to
\lambda_1\frac{\widehat{\mu}_Y(Y)}{\widehat{\mu}_Y(K)}\in(\lambda_1 ,\lambda_1 (1+\tau)).
\end{gather}
As $\|Q(0,1)^{-1}\|^{-1}\|Q(1,n)\|\leq \|Q(0,n)\|\leq \|Q(0,1)\|\|Q(1,n)\|$, this gives
\begin{equation}\label{eqn;logqn}
\frac{1}{n}\log\|Q(n)\|\to
\lambda_1\frac{\widehat{\mu}_Y(Y)}{\widehat{\mu}_Y(K)}\in(\lambda_1 ,\lambda_1 (1+\tau)).
\end{equation}
The above convergences lead directly to \eqref{def;sdc0}, \eqref{def;sdc2} and \eqref{def;sdc3}.
\medskip

\noindent\emph{Angle estimate.}
Since,
\begin{align*}
 \frac{\log \left|\sin \angle \left(E_j^{(k)}, U_j^{(k)}\right) \right|}{\log\|Q(k)\|}=&
 \frac{\log \left|\sin \angle \left(E_j(\widehat{\mathcal{R}}_K^{k-1}(\widehat{\mathcal{R}}^{n_1}(\pi,\lambda,\tau))), U_j(\widehat{\mathcal{R}}_K^{k-1}(\widehat{\mathcal{R}}^{n_1}(\pi,\lambda,\tau)))\right) \right|}{r_K^{(k-1)}(\mathcal{R}^{n_1}(\pi,\lambda,\tau))}\\
 &\cdot\frac{r_K^{(k-1)}(\mathcal{R}^{n_1}(\pi,\lambda,\tau))}{k}\frac{k}{\log\|Q(k)\|},
\end{align*}
in view of \eqref{eq:laangle}, \eqref{eq:r0n} and \eqref{eqn;logqn},
\[\lim_{k\to\infty}\frac{\log \left|\sin \angle \left(E_j^{(k)}, U_j^{(k)}\right) \right|}{\log\|Q(k)\|}=\frac{0}{\lambda_1}=0>-\tau.\]
This leads to \eqref{def;sdc5}.
\medskip

\noindent\emph{Rokhlin tower condition.}
Since $\widehat{\mathcal{R}}^{n_k}(\pi,\lambda,\tau)\in Y$ and the set $Y \subset X(\mathcal{G})$ is chosen to satisfy the conditions $(i)$ and $(ii)$, by the proof of Lemma~3.6 in \cite{Fr-Ul2},
 $(n_k)_{k\geq0}$
is a Rokhlin-balanced accelerating sequence. Thus $(n_k)_{k\geq0}$ satisfies \eqref{def;sdc4} and \eqref{def:FDC-g}.
\end{proof}

\end{document}